\documentclass[11pt,reqno]{amsart}
\usepackage{fullpage}
\usepackage[T1]{fontenc}                                   
\usepackage{amsfonts}
\usepackage[utf8]{inputenc}                               
\usepackage{comment}                                       
\usepackage{mparhack}                                      
\usepackage{amsmath,amssymb,amsthm,mathrsfs,eucal}                      
\usepackage{booktabs}                                                                          
\usepackage{graphicx,subfig}                                                                
\usepackage{wrapfig}                          \usepackage{multicol}                                                  
\usepackage[bookmarks=true,colorlinks=true]{hyperref}                      
\usepackage{bm}                                                                                   

\newtheorem{defn}{Definition}[section]
\newtheorem{thm}{Theorem}[section]
\newtheorem{prop}{Proposition}[section]
\newtheorem{lem}{Lemma}[section]
\newtheorem{cor}{Corollary}[section]

\DeclareMathOperator*{\sign}{sign}

\newcommand{\R}{\mathbb{R}}

\newcommand{\oP}{{\mathcal{P}}\left[g\right](t,w)}

\def\XXint#1#2#3{{\setbox0=\hbox{$#1{#2#3}{\int}$}
\vcenter{\vspace{-1pt}\hbox{$#2#3$}}\kern-.5\wd0}}
\def\Xint#1{\mathchoice {\XXint\displaystyle\textstyle{#1}}{\XXint\textstyle\scriptstyle{#1}}{\XXint\scriptstyle\scriptscriptstyle{#1}}{\XXint\scriptscriptstyle\scriptscriptstyle{#1}}\!\int}
\def\intmed{\hbox{\ }\Xint{\hbox{\vrule height -0pt width 10pt depth 1pt}}}

\title[Opinion Formation Systems] 
      {Opinion Formation Systems via Deterministic Particles Approximation}

\author[Simone Fagioli and Emanuela Radici]{}


 \email{simone.fagioli@univaq.it}
 \email{emanuela.radici@epfl.ch}

\begin{document}
\maketitle

\centerline{\scshape Simone Fagioli}
\medskip
{\footnotesize
 \centerline{DISIM, Universit\`a degli Studi dell’Aquila}
   \centerline{ via Vetoio 1 (Coppito), 67100 L’Aquila (AQ), Italy}
} 

\medskip

\centerline{\scshape Emanuela Radici}
\medskip
{\footnotesize
 \centerline{EPFL SB, Station 8,}
   \centerline{CH-1015 Lausanne, Switzerland}
}

\begin{abstract}
We propose an ODE-based derivation for a generalized class of opinion formation models either for single and multiple species (followers, leaders, trolls). The approach is purely deterministic and the evolution of the single opinion is determined by the competition between two mechanisms: the opinion diffusion and the compromise process. Such deterministic approach allows to recover in the limit  an aggregation/(nonlinear)diffusion system of PDEs for the macroscopic opinion densities.
\end{abstract}

\section{Introduction}\label{sec:intro}
The study of phenomena in social sciences trough mathematical modelling has gained an increasing attention in the scientific community, in particular in the last decades, see \cite{BeMaTo,CaFoLo,Ga,MoTa,NaPaTo,PaTobook,St}. The huge increasing of private and public communications on social networks like Facebook and Twitter speed up the attention on social phenomena, mainly due to a huge amount of data coming from empirical observations. The large information exchange push the study on the analysis of how these interactions influence the process of opinion formation, \cite{AlPaToZa,Be,BoLo,KlShSh,LaMa,SlLa,Sz,YaRoSc}.

Social networks are now particularly important for the political leaders communication, since they give the possibility of driving selected information to potential electors. Actually, the phenomenon of opinion leaders and their possible control strategy on the public opinion dates back to the work of  Lazarsfeld, \cite{Laz}, in the study of the USA presidential elactions in 1940, see also \cite{BiSe,BoSa}. The social networks, however, show an innovative feature: the possibility of \emph{measuring} the popularity of a given leader through factors such as the number of followers or the number of likes, see \cite{CrLaPTe}. The drawback is that these \emph{measure} can be easily falsified to guide the behaviour of real users and persuading them to vote for a specific candidate, see \cite{DeCFrJoKASh,KrGuHa}. In particular, in \cite{Bo} it was estimated that the most followed political US accounts on Twitter posses the $25\%$ of fake followers.

In the present work, we introduce a model general enough to catch such phenomenon.

Among the possible mathematical models, the kinetic formulation of opinion formation,  introduced in the seminal paper \cite{T1}, has gained a lot of attention mainly because of its flexibility to describe the phenomenon at different levels: microscopic, based on the pairwise interaction between agents, mesoscopic for the distribution of the opinions, and last but not least, macroscopic, useful for describing the trend of the opinion density. 
In this kinetic model, interactions among agents are supposed to be governed by two relevant concepts: compromise describes the way in which pairs of agents reach a compromise after exchanging opinions (its structure and other important features had been intensively studied in \cite{AlPaToZa,Be,DeAmWeFa,HeKr}), and self-thinking, modelled by a random variable, describes how agents change their opinions in an unpredictable way \cite{Be,T1}. As a result of the above considerations, one can consider the following pairwise interactions law between two agents with opinion $w$ and $v$ respectively
\begin{equation}\label{eq:kin_intro}
\begin{split}
& w' =w-P(w)(w-v)+\eta_wD(w),\\
& v' =v-P(v)(v-w)+\eta_vD(v),
\end{split}
\end{equation}
where  $(w',v' )$ denotes the opinions after the interaction, the functions $P$ and $D$ describe the local relevance of the compromise and the self-thinking (diffusion) for a given opinion, and $\eta_w$ and $\eta_v$ are two random variables. After a suitable scaling process, called \emph{quasi-invariant opinion limit}, the author in \cite{T1} obtains a Fokker-Planck type equation for the opinion density, precisely of the form
\begin{equation}\label{eq:toscani}
\partial_t u(t,w)= \frac{\lambda^2}{2}\partial_{ww}\left(D^2(w)u(t,w)\right) +\partial_w\left(P(w)(w-m) u(t,w) \right),
\end{equation}
 which results to be a good approximation for the stationary profiles of the kinetic equation associated to \eqref{eq:kin_intro}. A similar approach was also used to model more general opinion formation processes, such as the presence of opinion leaders, choice formation, control and networks (see \cite{AlPaToZa,AlPaZa1,AlPaZa2,AlPaZa3,during,DuWo,ToToZa}), as well as the different class of trading models for goods and wealth distribution (see \cite{PaTo,ToBrDe}).

We propose an alternative derivation which is based on a deterministic many particle limit and allows to consider a generalized version of the above kinetic model.

\subsection{Formal derivation of the generalized opinion formation model}
This section is devoted to introduce and formally derive, via many particle limit, the set of equations we want to investigate in the present paper.

\subsubsection{Basic one-species model} Consider a population composed of $N+1$ individuals with given initial opinion $W_i^0$, for $i=0,\ldots,N$ and assume that opinions can only range in a bounded set of values, say $I=\left[-1,1\right]$ where $W=\pm 1$ represent the extreme opinions. We further assume that each opinion has a certain \emph{strength} $\sigma_i$, for $i=0,\ldots,N$. According to the kinetic model described above, the $i$-individual can modify its own opinion depending on two possible mechanisms: the compromise  process (interaction) with the others individuals and the diffusion of that given opinion. Therefore, the time evolution of the opinion of the $i$-individual is described by the following ODE
 \begin{equation}\label{eq:ode_intro}
     \dot{W}_i = \mbox{ compromise} + \mbox{ diffusion }.
 \end{equation}
 It is standard to assume that the local or non-local relevance of the compromise depends on the distance between two different opinions. Then, by calling $P(\cdot)$ the function describing the local relevance of the compromise, the interaction of the $i-$individual with the other individuals can be modelled by
 \begin{equation}\label{eq:nonlocal_ode}
    \mbox{ compromise} = -\sum_{j=0}^{N}\sigma_j P(W_i,W_j)(W_i-W_j).
 \end{equation}
 
 For what concerns the diffusive part, instead, we assume that the opinions evolve with a speed equal to the \emph{osmotic velocity} associated to the diffusion process, as firstly introduced in \cite{russo1}.  More precisely, if we denote by $D(\cdot)$ the diffusion capacity of a given opinion, then the diffusive process is described by
  \begin{equation}\label{eq:diffusion_ode}
    \mbox{ diffusion } = \frac{\lambda^2}{2\sigma_i} D^2(W_i)\left(\frac{\sigma_{i-1}}{W_i-W_{i-1}}-\frac{\sigma_{i}}{W_{i+1}-W_{i}}\right),
 \end{equation}
 where $\lambda$ is a fixed diffusion coefficient. Let us observe, that, accordingly to the results in \cite{FaRa,gosse_toscani}, it is possible to generalize the diffusion law \eqref{eq:diffusion_ode} to more general \\emph{non linear} expressions. 
 Indeed, the quantity
 
\begin{equation}\label{eq_disc_d
ens_intro}
     u_i = \frac{\sigma_{i}}{W_{i+1}-W_{i}},
 \end{equation}
 represents the local density between two consecutive opinions, then  one can consider a generic nonlinear non-decreasing real valued function $\phi$ to rephrase \eqref{eq:diffusion_ode} as
   \begin{equation}\label{eq:diff_non_ode}
    \mbox{ diffusion } = \frac{\lambda^2}{2\sigma_i} D^2(W_i)\left(\phi(u_{i-1})-\phi(u_{i})\right).
 \end{equation}
 A further, hopefully realistic, assumption of our model concerns the boundary conditions. More precisely, we impose that the extreme opinions cannot alter during the evolution. As a consequence, we set
   \begin{align}\label{eq:ode_main_intro}
  \nonumber  &\dot{W}_0 = 0,\\
    & \dot{W}_i  = \frac{\lambda^2}{2\sigma_i} D^2(W_i)\left(\phi(u_{i-1})-\phi(u_{i})\right)-\sum_{j=0}^{N}\sigma_j P(W_i,W_j)(W_i-W_j),\, i\neq 0,N,\\  \nonumber
& \dot{W}_N = 0.
 \end{align}
 By formally sending $N\to\infty$ and calling $\mathcal{W}(z,t)$ the piecewise linear interpolation of the values $W_i$, we get
 \begin{equation*}
 \partial_t\mathcal{W}(t,z) = \frac{\lambda^2}{2\sigma_i} D^2(\mathcal{W} )\partial_z\left(\phi\left(\frac{1}{\partial_z\mathcal{W} }\right)\right)-\int_{0}^\sigma P(\mathcal{W} (z),\mathcal{W} (\zeta))\left(\mathcal{W} (z)-\mathcal{W} (\zeta)\right)d\zeta.
 \end{equation*}
Once here, if $\mathcal{U}(t,w)$ were the inverse of the $\mathcal{W}(t,z)$ and $u(t,w)$ were $\partial_w \mathcal{U}(t,w)$ then, inspired by the results in \cite{DiFFaRa,FaRa}, $u$ would be a weak solutions of the following aggregation-diffusion PDE 
 \begin{equation}\label{eq:PDE_main}
\partial_t u(t,w)= \partial_w\left(\frac{\lambda^2}{2}D^2(w)\partial_w \phi(u(t,w)) -\oP u(t,w) \right),
\end{equation}
for $(t,w)\in \left[0,T\right]\times I$, endowed with zero flux boundary condition, where the nonlocal operator $\oP$ is defined by
\[
 \oP = \int_I P(w,v)(v-w)u(t,v)dv.
\]
Note that, even in the linear diffusion case $\phi(u)=u$, this derivation leads to a diffusion of Fick type, differently  from the one in \eqref{eq:toscani}, that is of Laplacian type. However, expanding the inner derivative in \eqref{eq:toscani}, we get
\[
 \partial_{ww}(D^2(w)u(t,w))=\partial_w(D^2(w)\partial_w u(t,w))+\partial_w(\partial_wD^2(w)u(t,w)),
\]
where the first term on the r.h.s. corresponds to the diffusion in \eqref{eq:PDE_main}, while the second term plays the role of a local nonlinear transport term. A similar particle derivation can be performed in order to reconstruct the this transport, see \cite{DiFSt}.  Another difference between the two models is the fact that the \emph{mean opinion}
\begin{equation}\label{eq:mean_op_intro}
 m_1(t) = \int_I w u(t,w) dw,
\end{equation}
is not preserved in time. From a modelling point of view, this fact can be interpreted as a more \emph{realistic} compromise process, see the discussion in section \ref{sec:num}.

\subsubsection{Leaders-followers model}
We consider now a situation where the population is divided in subgroups: one group of followers and two (or more) groups of leaders, see \cite{during}. Hence, by denoting with $F_i$ the opinion of the $i-$th follower and with $L_h$ and $R_k$ the opinion of the $h-$th leader in the \emph{left} group and $k-$th leader of the \emph{right} group respectively, we have that the $i$-follower opinion evolves according to
\begin{align*}
 \dot{F}_i = \frac{\lambda_f^2}{2\sigma_{f,i}} D_f^2(F_i)\left(\phi_f(f_{i-1})-\phi_f(f_{i})\right)&-\sum_{j=1}^{N_f}\sigma_{f,j} P_{ff}(F_i,F_j)(F_i-F_j)\\
 &-\sum_{h=1}^{N_l}\sigma_{l,j} P_{fl}(F_i,L_h)(F_i-L_h)\\
 &-\sum_{k=1}^{N_r}\sigma_{r,j} P_{fr}(F_i,R_k)(F_i-R_k),
\end{align*}
where the last two sums concern the interactions between the $i$-th follower and the different leaders. Similarly, we can write
\begin{align*}
 \dot{L}_i = \frac{\lambda_l^2}{2\sigma_{l,i}} D_l^2(L_i)\left(\phi_l(l_{i-1})-\phi_l(l_{i})\right) &-\sum_{h=1}^{N_l}\sigma_{l,j} P_{ll}(L_i,L_h)(L_i-L_h)\\
 &-\sum_{k=1}^{N_r}\sigma_{r,j} P_{lr}(L_i,R_k)(L_i-R_k),
\end{align*}
for the generic $i$-th left leader and
\begin{align*}
 \dot{R}_i = \frac{\lambda_r^2}{2\sigma_{r,i}} D_r^2(R_i)\left(\phi_r(r_{i-1})-\phi_r(r_{i})\right) &-\sum_{k=1}^{N_r}\sigma_{r,j} P_{rr}(R_i,R_k)(R_i-R_k)\\
 &-\sum_{h=1}^{N_l}\sigma_{l,j} P_{lr}(R_i,L_h)(R_i-L_h),
\end{align*}
for the $i$-th right leader. In the previous equations we are considering \emph{strong} opinion leaders, i.e. the interaction with the followers does not affect the leader's opinion. In other words, we assume that $P_{lf} = P_{rf} = 0$. From the above system of ODEs, we can formally derive the following system of PDEs
\begin{equation}\label{eq:sys_FL}
\begin{split}
& \partial_t f = \partial_w \Big(\frac{\lambda_f^2}{2}D_f^2 (w) \partial_w \phi_f(f) - f \big( \mathcal{P}_{ff}[f] + \mathcal{P}_{fl}[l] + \mathcal{P}_{fr}[r] \big)\Big),\\
&  \partial_t l = \partial_w \Big(\frac{\lambda_l^2}{2}D_l^2 (w) \partial_w \phi_l(l) - l \big(\mathcal{P}_{ll}[l] + \mathcal{P}_{lr}[r]\big)\Big) ,\\
&  \partial_t r = \partial_w \Big(\frac{\lambda_r^2}{2}D_r^2 (w) \partial_w \phi_r(r) - r \big(\mathcal{P}_{rr}[r] + \mathcal{P}_{rl}[l]\big)\Big).
\end{split}
\end{equation}

\subsubsection{Leaders-followers-trolls model}
As a last example, we consider a case that is of interest for the understanding of opinion formation and opinion clustering in social media. Accordingly to the consideration mentioned in the Introduction, we assume that one of the group of leaders \emph{introduces} a new group of \emph{fake} agents, commonly known as \emph{trolls}. The trolls are indistinguishable from the followers, they only interact with the reference leaders and they cannot diffuse their opinion.
It is reasonable to assume that, in this new setting, the leaders opinion evolution is not affected by the presence of the trolls, while the followers have an additional compromise term that models the followers-trolls interaction. Clearly, since the followers cannot distinguish trolls, the compromise function of the follower-troll interaction remains $P_{ff}$. Therefore, the ODEs for the followers becomes 
\begin{align*}
 \dot{F}_i =& \frac{\lambda_f^2}{2\sigma_{f,i}} D_f^2(F_i)\left(\phi_f(f_{i-1})-\phi_f(f_{i})\right)-\sum_{j=1}^{N_f}\sigma_{f,j} P_{ff}(F_i,F_j)(F_i-F_j)\\
 &-\sum_{h=1}^{N_l}\sigma_{l,j} P_{fl}(F_i,L_h)(F_i-L_h)-\sum_{k=1}^{N_r}\sigma_{r,j} P_{fr}(F_i,R_k)(F_i-R_k)\\
 &-\sum_{j=1}^{N_f}\sigma_{q,j} P_{ff}(F_i,Q_j)(F_i-Q_j).
\end{align*}
Besides, consistently with the two features  that trolls only interact with the corresponding group of leaders and that they cannot diffuse their opinion, we deduce the trolls evolution law
\begin{align*}
 \dot{Q}_i =&-\sum_{k=1}^{N_r}\sigma_{r,k} P_{qr}(Q_i,R_k)(Q_i-R_k).
\end{align*}
In the macroscopic limit we then have
\begin{equation}\label{eq:sys_FLT}
\begin{split}
& \partial_t f = \partial_w \Big(\frac{\lambda_f^2}{2}D_f^2 (w) \partial_w \phi_f(f) - f \big( \mathcal{P}_{ff}[f]+ \mathcal{P}_{ff}[q] + \mathcal{P}_{fl}[l] + \mathcal{P}_{fr}[r] \big)\Big),\\
&  \partial_t l = \partial_w \Big(\frac{\lambda_l^2}{2}D_l^2 (w) \partial_w \phi_l(l) - l \big(\mathcal{P}_{ll}[l] + \mathcal{P}_{lr}[r]\big)\Big) ,\\
 &  \partial_t r = \partial_w \Big(\frac{\lambda_r^2}{2}D_r^2 (w) \partial_w \phi_r(r) - r \big(\mathcal{P}_{rr}[r] + \mathcal{P}_{rl}[l]\big)\Big),\\
&  \partial_t q = -\partial_w \Big(q \mathcal{P}_{qr}[r] \Big).
\end{split}
\end{equation}
In the present paper we provide a solid existence theory for the sample models presented so far, nevertheless it can be easily adapted to many other combinations of populations/interactions the readers may wish to consider.  

\bigskip

The paper is structured as follows. In section \ref{sec:prel} we introduce the main assumptions, the rigorous statement of the discrete setting and we state our main result Theorem \ref{thm:main}. Section \ref{sec:conti} is devoted to the proof of Theorem \ref{thm:main}. There, we show how a proper piecewise constant reconstruction of density, built from the \emph{microscopic} ODEs, converges to weak a solution of the corresponding PDEs system. In section \ref{sec:num}, instead, we study the large-time behaviour of the macroscopic model and we conclude with some numerical simulations based on the particle scheme, which validate the analytical results of the paper.

\section{Preliminaries and main result}\label{sec:prel}
\subsection{Main assumptions}\label{sec:assumptions}
In this section we state the setting and the assumptions of our main result. 
Let $w$ the a generic opinion belonging to the interval  $I=\left[-1,1 \right]$, and let $u(t,w)$ the macroscopic opinion density at time $t$ and opinion $w$ of a fixed population. Then the equation for $u$ is given by
\begin{equation}\label{main}
\partial_t u(t,w)= \partial_w\left(\frac{\lambda_u}{2}D_u^2(w)\partial_w \phi_u(u(t,w)) -u(t,w) \sum_{h \in \{f,l,r,q\}}\mathcal{P}_{uh}[h](t,w) \right).
\end{equation}
where $\lambda_u>0$ is a diffusion coefficient. Equation \eqref{main} is endowed with zero-flux boundary condition and the initial condition $u(0,w):= \bar{u}$, for some $\bar{u}$ satisfying the following assumptions:
\begin{itemize}
\item[(In1)] $\bar{u} \in BV(I; \R^+)$ with $\|\bar{u}\|_{L^1(I)} = \sigma_u$, for some $\sigma_u>0$,
\item[(In2)] there exist $m_u,M_u >0$ such that $m_u \leq \bar{u}(w) \leq M_u$ for every $w \in I$.
\end{itemize}
We recall that the function $D_u$ models the diffusion of a given opinion and plays the role of a mobility function. We assume that
\begin{itemize}
\item[(D1)]  $D_u\in C^2(I)$,  $0 \leq D_u \leq \| D_u\|_{L^\infty} < \infty $ and $D_u(\pm 1)=0$. 
\item[(D2)] $\partial_w D_u^2$ is uniformly bounded in $I$.
\end{itemize}
The prototype example for this diffusivity function is
\begin{equation}\label{eq:mainD}
D_u(w)=\left(1-w^2\right)^\frac{\alpha}{2},\quad \alpha >0.
\end{equation}
Observe that, under this choice, condition (D2) corresponds to the values $\alpha\geq 1$, see Figure \ref{fig:dDalpha}.
\begin{figure}[htbp]
\begin{center} 
    \includegraphics[scale=0.6]{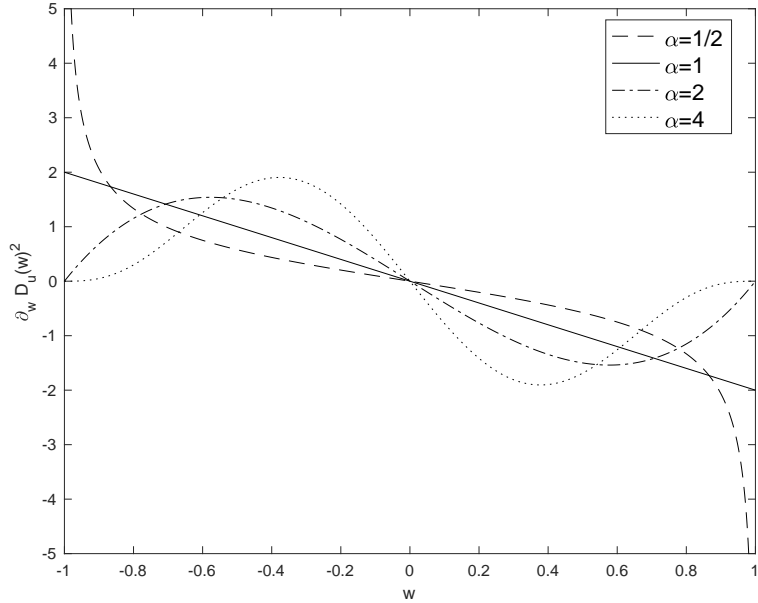}\par 
\caption{Behaviours of $\partial_w D_u^2$ for different values of $\alpha$.}
\label{fig:dDalpha}
\end{center}
\end{figure}

We further assume that
\begin{itemize}
\item[(Dif)] $\phi_u: [0,\infty) \to \R$ is a nondecreasing Lipshcitz function, with $\phi_u(0)=0$.
\end{itemize}
Remember that the nonlocal operator  $\mathcal{P}_{uh}[h](t,w)$ is modelling the local relevance of compromise among the same population (case $h=u$) or between different populations (case $h \neq u$), and is defined by
\[
\mathcal{P}_{uh}[h](t,w) = \int_I P_{uh}(w,v)(v-w)h(t,v)dv.
\]
We will denote with $\partial_i P_{uh}$ the derivative of $P_{uh}$ w.r.t. the $i-$th variable, for $i=1,2$, and we assume that
\begin{itemize}
\item[(P)] for every $h \in \{f,l,r,q\}$, $P_{uh}(\cdot,\cdot)$ is a non negative and uniformly bounded function. Moreover, $P_{uh}$ and $\partial_1 P_{uh}$ are  Lipschitz w.r.t. to both components, uniformly in the other:
\[
|P_{uh}(w_1,w^*)-P_{uh}(w_2,w^*)| + |P_{uh}(w^*,w_1) - P_{uh}(w^*,w_2)|\leq Lip(P_{uh})|w_1-w_2|,\] 
for all $w_1,\,w_2,\ w^*\in I$.
\end{itemize}
Typical choices for the function $P_{uh}$ are $P_{uh}(w,v)=1$, that correspond to the Sznajd-Weron model \cite{Sz}, $P_{uh}(w,v)=1-|w|$ or $P_{uh}(w,v)=1-|w-v|$.
\subsection{Rigorous statement of the discrete setting}\label{sec:pippone}
Let us denote with $I_T=\left[0,T\right]\times \left[-1,1\right]$ and with $\Gamma_T=\left[0,T\right]\times \left\{\pm 1\right\}$. We notice that we can rewrite both systems \eqref{eq:sys_FL} and \eqref{eq:sys_FLT} in a compact form as following: let $u$ be either $f,l,r$ or $q$, then it follows
\begin{equation}\label{main_comp}
\begin{cases}
\partial_t u= \partial_w\left(\frac{\lambda_u}{2}D_u^2\partial_w \phi_u(u) -\sum_{h \in \{f,l,r,q \}}u \mathcal{P}_{uh}[h] \right), &(t,w)\in I_T,\\
\left(\frac{\lambda_u}{2}D_u^2\partial_w \phi_u(u) -\sum_{h \in \{f,l,r,q \}}u\mathcal{P}_{uh}[h]\right) =0,& (t,w)\in\Gamma_T,\\
u(0,w) = \bar{u}(w), & w\in I.
\end{cases}
\end{equation}
where $P_{uh}$ and $D_u$ are eventually zero for some $h \in \{f,l,r,q \}$ and $\bar{u}: I\to \mathbb{R}$ is the initial datum for the generic species $f,l,r$ or $q$, satisfying (In1) and (In2). Fix  $N \in \mathbb{N}$, and introduce $\sigma_u^N=\sigma_u/N$ as the infinitesimal mass associated to each opinion. In this way, we will have the same number $N$ of particles for each population. This is just to simplify the computation when we prove the convergence of the scheme. Indeed, our argument works the same even when we have different amounts of particles $N^h$ for different populations, provided that we perform the particle limit at the same time for any all populations. Even if this choice would be more suitable from the modelling point of view, for example in order to catch multi-scale phenomena, taking the same number of particles $N$ is not restrictive since we are interested in recovering  the macroscopic equations for all species.
Thus, we atomize $\bar{u}$ in $N+1$ discrete opinions: set $W_0^0=-1$ and define recursively 
\begin{equation}\label{eq:ini_part}
    W_i^0 = \sup \left\lbrace x\in \R : \int_{W_{i-1}^0}^x \bar{u}(z)dz < \sigma_u^N \right\rbrace,\quad \forall\, i=1,\ldots, N-1.
\end{equation}
By construction we have that $W_N^0=1$ and $W_{i+1}^0 - W_i^0 \geq \frac{\sigma_u^N}{M_u}$ for every $i=0,\ldots, N-1$, where $M_u$ is defined in (In2). Given this initial set of particles, we let the positions $W_i$ evolve according to
\begin{equation}\label{eq:particles_gen}
\left\lbrace\begin{array}{llll}
  &\dot{W}_i(t) = \dot{W}_i^d(t) + \dot{W}_i^{p}(t), \quad  &i=1,\ldots,N-1 \\
  &\dot{W}_0(t)=\dot{W}_{N}(t)=0, \\
  &W_i(0)=W_i^0, \qquad &i=1,\ldots,N-1
\end{array}\right.
\end{equation}
where
\begin{equation}\label{eq:z_d}
    \dot{W}_i^d(t) =  \frac{\lambda_u }{2\sigma_u^N}D_u^2(W_i(t))\big(\phi_u (u_{i-1}(t)) - \phi_u(u_i(t))\big), 
\end{equation}
and
\begin{equation}\label{eq:z_n}
    \dot{W}_i^{p}(t) = - \sum_{h \in \{f,l,r,q \}} \sigma_h^N \sum_{j=0}^{N+1} P_{uh}(W_i(t),H_j(t))(W_i(t) - H_j(t)).
\end{equation}
In \eqref{eq:z_d} and \eqref{eq:z_n}, $H_j$ indicates the opinions of the population $h$ and $u_i(t)$ is a local reconstructions for the density \begin{equation}
 u_i(t):= \frac{\sigma^N_u}{ (W_{i+1}(t) - W_i(t))},\quad i=0,\ldots,N.
 \end{equation} 
 Note that 
\begin{equation}\label{eq:uidot1}
\dot{u}_i (t)= - \frac{u_i^2 (t)}{\sigma_u^N} (\dot{W}_{i+1} (t) - \dot{W}_i (t)) = \dot{u}_i^d  (t)+ \dot{u}_i^{p} (t)
\end{equation}
 with 
\begin{equation}\label{eq:uidot2}
    \dot{u}_i^d(t) = - u_i(t) \frac{\dot{W}_{i+1}^d(t) - \dot{W}_i^d(t)}{W_{i+1}(t) - W_i(t)}, \qquad \dot{u}_i^{p}(t) = - u_i(t)\frac{\dot{W}_{i+1}^{p}(t) - \dot{W}_i^{p}(t)}{W_{i+1}(t) - W_i(t)}.
\end{equation}
In order to lighten the notation, we will normalize the diffusion constant $\frac{\lambda_u }{2}$ to $1$ and denote by $\mathfrak{H} = \{f,l,r,q\}$ the tag-set of the different populations. We further introduce the constant
\begin{equation}\label{eq:theta}
\Theta_{u} :=\sum_{h \in \mathfrak{H}} \big( Lip[P_{uh}] + \| P_{uh}\|_{L^\infty} + Lip[\partial_1 P_{uh}]\big),
\end{equation}
that will be largely used in the following. 
The following Lemma shows that there is an a priori control from above and below on the mutual distance between particles. In particular, the order of the particles is maintained during the whole evolution.

\begin{lem}[Discrete Min-Max Principle]\label{lem:minmax system}
Let $T>0$ be fixed and $\bar{u}$ under the assumptions (In1) and (In2). Let $c_u$ be a positive constant satisfying $c_u> \Theta_{u}$, with $\Theta_{u}$ defined in \eqref{eq:theta}. Then 
\begin{equation}\label{eq: minmax system}
\frac{\sigma_u^N}{M_u}e^{-c_u T} \leq W_{i+1}(t) - W_i(t) \leq \frac{\sigma_u^N}{m_u}e^{c_u T},
\end{equation} 
for every $i=0,\ldots,N-1$ and $t \in [0,T]$. Moreover 
\begin{equation}\label{eq: minmaxg system}
    e^{-c_u T}m_u \leq u_i(t)\leq e^{c_u T}M_u,
\end{equation}
for all $i=0,...,N-1$ and $t\in[0,T]$.
\end{lem}

\begin{proof}
The proof of both inequalities is argued by contradiction. We only prove the lower bound of \eqref{eq: minmax system}, the proof of the upper bound can be easily obtained by adapting the following steps. Thanks to (In2), we know that $$W_{i+1}(0)-W_i(0)\geq \frac{\sigma_u^N}{M_u}, \quad i=0,...,N.$$
To simplify the notation, let us introduce the \emph{lower bound} function 
\[
lb_u(t)=\frac{\sigma_u^Ne^{-c_ut}}{M_u}.
\]
Let now $t_1$ be 
\begin{equation}\label{contr_1 system}
 t_1=\inf_{t \in [0,T] } \left\{W_{i+1}(t)-W_i(t) = lb_u(t),\,W_{j+1}(t)-W_j(t) \geq lb_u(t), j\neq i\right\}.
\end{equation} 
Suppose, by contradiction, that there exists some $0<t_1<t_2 \leq T$ such that
\[
W_{i+1}(t)-W_i(t) < lb_u(t), \quad t\in\left(t_1,t_2\right],
\]
then the contradiction follows as soon as we show that $$\frac{d}{dt}\left[ e^{c_u t}(W_{i+1}(t) - W_i(t))\right]_{|_{t=t_1}} > 0.$$
We can compute the time derivative as follows
\begin{align*}
&\frac{d}{dt}\left[ e^{c_u t}(W_{i+1}(t) - W_i(t))\right]_{|_{t=t_1}}\\
& = e^{c_u t_1}\left[(\dot{W}_{i+1}(t_1) -\dot{W}_i(t_1)) + c_u(W_{i+1}(t_1) - W_i(t_1))\right] \\
&= e^{c_u t_1} \left[(\dot{W}^d_{i+1}(t_1) -\dot{W}^d_i(t_1)) + (\dot{W}^{p}_{i+1}(t_1) -\dot{W}^{p}_i(t_1))
\right]+ \frac{c_u\sigma_u^N}{M_u},
\end{align*}
with $\dot{W}_i^d$, $\dot{W}_i^{p}$ defined in \eqref{eq:z_d} and \eqref{eq:z_n} respectively. Thanks to assumption \eqref{contr_1 system}, 
\begin{equation*}
W_{i+1}(t_1)-W_i(t_1)\leq W_{j+1}(t_1)-W_j(t_1), \mbox{ for all } j\neq i,
\end{equation*} 
then $W_{i \pm 1}(t_1)\leq W_i(t_1)$ and this directly implies that $(\dot{W}^d_{i+1}(t_1) -\dot{W}^d_i(t_1)) \geq 0$, using the assumption (Dif). 
Recalling the definition of $\dot{W}_i^{p}$ and the assumptions (P) on the generic $P_{uh}$, it is immediate to get  
\begin{align*}
&\dot{W}^{p}_{i+1}(t_1) -\dot{W}^{p}_i(t_1)  \\
&= - \sum_{h \in \mathfrak{H}} \sigma_h^N \sum_{j=0}^{N} [P_{uh}(W_{i+1},H_j)(W_{i+1} - H_j)   - P_{uh}(W_i,H_j)(W_i - H_j)]  \\
& = \sum_{h \in \mathfrak{H}} \sigma_h^N  \sum_{j=0}^{N} \big[ (P_{uh}(W_{i+1},H_j) - P_{uh}(W_{i},H_j))(W_{i+1} - H_j)\big] \\
& \quad + \sum_{h \in \mathfrak{H}} \sigma_h^N  \sum_{j=0}^{N} \big[    P_{uh}(W_{i},H_j))(W_{i+1} - W_i)\big] \\
& \geq -\Theta_{u} (W_{i+1}(t_1) - W_i(t_1)).
\end{align*}
Finally, by using the above lower bound on the nonlocal part, the $0$ lower bound of the diffusive part and by recalling that $c_u > \Theta_{u}$,
we deduce 
\begin{equation}\label{max system}
\frac{d}{dt}\left[ e^{c_u t}(W_{i+1}(t) - W_i(t))\right]_{|_{t=t_1}}\geq \frac{\sigma_u^N}{M_u}\left[c_u - \Theta_{u} \right] > 0,
\end{equation}
 which gives the desired contradiction. 
\end{proof}

We are now in position to define the $N$-discrete density as 
\begin{equation}\label{rhoN}
u^N(t,\,w):= \sum_{i=0}^{N-1}  u^N_i (t) \chi_{[W^N_i(t),\,W^N_{i+1}(t))}(w), 
\end{equation}  

because the intervals $[W^N_i(t),\,W^N_{i+1}(t))$ are well defined for every $t \in [0,T]$. Moreover, 
$$ 
\| u^N(t,\cdot)\|_{L^1(I)} = \sum_{i=0}^{N-1} u^N_i(t)\int_{W^N_i(t)}^{W^N_{i+1}(t)}1 \,dw = \sum_{i=0}^{N-1} \frac{\sigma_u}{N} = \sigma_u,
$$
for every $t \in [0,T]$ and independently on $N$. As a consequence of \eqref{eq: minmaxg system}, we deduce that $(u^N(t,w))_N$ is also uniformly bounded in $L^\infty(I_T)$.

In the following Definition we introduce the notion of weak solutions we are dealing with.

\begin{defn}[Weak Solution]\label{weaksolutiondfn}
Let $u$ be either $f,l,r$ or $q$. The function $u$ is a \emph{weak solution} of \eqref{main_comp} if $u \in L^\infty  \cap BV(I_T)$,  $u(0,\cdot)= \bar{u}$ and satisfies
\[ \int_{I_T} \Big[u \partial_t \zeta + \phi_u(\rho)\partial_w\left(D_u^2\partial_w\zeta\right) - u\sum_{h\in \mathfrak{H}}\mathcal{P}_{uh}\left[h\right] \partial_w\zeta \Big] dw\,dt = 0, \]
 for all $\zeta \in C^{\infty}_{0}(I_T)$.
\end{defn}

The main result of this manuscript is stated in the following Theorem.

\begin{thm}\label{thm:main}
For every $u,h$ among $\{f,l,r,q\}$,
consider $\phi_u,\,D_u,\,P_{uh}$ under the assumptions $(Dif)$, $(D1)$, $(D2)$ and $(P)$ respectively. Let $\bar{u}: I \to \R$ be as in $(In1)$ and $(In2)$ and let $T>0$ be fixed. Then, up to a subsequence, the discretized densities $u^N$  defined in \eqref{rhoN} strongly converge in $L^1(I_T)$ to a limit in $L^{\infty} \cap BV (I_T)$ which solves the Initial-boundary value Problem \eqref{main_comp} in the sense of Definition \ref{weaksolutiondfn}.
\end{thm}

Note that, the well-posedness for equations with a nonlinear \emph{space-}dependent mobility  $D$ degenerating at the boundary, as in \eqref{eq:PDE_main}, has not been subject yet of a complete and satisfactory theory, see \cite{AlNaTo,AnFa,FuPuTeTo}.

\subsection{Tools from Optimal Transport}\label{sec:opt_tra}
The purpose of this section is to collect and present some tools from Optimal Transport that will be useful in the sequel. We refer to \cite{AGS,S,V1} for an extensive treatment of the concepts mentioned in this section.  
In our setting, for arbitrary $t \in [0,T]$, the functions $u^N(t,\cdot)$ are all densities on $I$ with same mass, independently on $N$. The Wasserstein distance is the right notion of distance that allows us to evaluate how far from each other the two measures $u^N(t,\cdot)$ and $u^N(s,\cdot)$ are at different times $t,s \in [0,T]$. In this section we introduce a well known characterization of the Wasserstein distance in the one-dimensional setting, we refer to \cite{CT} for the detailed proof. 

For a fixed mass $\sigma>0$, we consider the space
\begin{equation*}
  \mathfrak{M}_\sigma = \bigl\{\mu \hbox{ Radon measure on $\R$ } \colon \mu\ge 0 \text{ and }\mu(\R)=\sigma \bigr\}.
\end{equation*}
Given $\mu\in \mathfrak{M}_\sigma$, we introduce the pseudo-inverse function $X_\mu \in L^1([0,\sigma];\R)$ as
\begin{equation}\label{eq:pseudoinverse}
X_\mu(z) = \inf \bigl\{ x \in \R \colon \mu((-\infty,x]) > z \bigr\}.
\end{equation}
In particular, if $\sigma=1$, then $\mathfrak{M}_1$ is the set of non-negative probability densities on $\R$ and it is possible to consider the one-dimensional \emph{$1$-Wasserstein distance} between each pair of densities $\rho_1,\rho_2\in \mathfrak{M}_1$. As shown in \cite{CT}, in the one dimensional setting the \emph{$1$-Wasserstein distance} can be equivalently defined in terms of the $L^1$-distance between the respective pseudo-inverse mappings as
\[
d_{W^1}(\rho_1,\rho_2) = \|X_{\rho_1}-X_{\rho_2}\|_{L^1([0,1];\R)}.
\]
For generic $\sigma >0$, we recall the definition for the \emph{scaled $1$-Wasserstein distance} between $\rho_1,\rho_2\in \mathfrak{M}_\sigma$ as
\begin{equation}\label{eq:wass_equiv0}
    d_{W^1,\sigma}(\rho_1,\rho_2)= \|X_{\rho_1}-X_{\rho_2}\|_{L^1([0,\sigma];\R)},
\end{equation}
see \cite{CT,V1}.

\section{Proof of the main result}\label{sec:conti}
This section is devoted to the proof of Theorem \ref{thm:main}. We will first focus on the $L^1$ compactness for the piecewise constant interpolation $(u^N(t,w))_N$ for $u=f,l,n,g$, defined in \eqref{rhoN}. The main tool available in this direction is a generalized version of the Aubin-Lions Lemma \cite{RoSa}, that we report here in a simplified version adapted to our setting.

\begin{thm}\label{Aubin Lions}

Let $T >0$ be fixed and $\rho^N(t,\cdot): [a,b] \to \mathbb{R}$ be a sequence of non negative probability densities for every $t \in [0,T]$ and for every $N \in \mathbb{N}$. Moreover, assume that $\| \rho^N(t,\cdot )\|_{L^\infty} \leq M $ for some constant $M$ independent on $t$ and $N$. If 
\begin{itemize}
    \item[I)] $\sup_N \int_0^T TV[\rho^N(t,\cdot)]dt < \infty$,
    \item[II)] $d_{W^1}(\rho^N(t,\cdot),\rho^N(s,\cdot)) < C|t-s|$ for all $t,s \in [0,T]$, where $C$ is a positive constant independent on $N$,
\end{itemize}
then $\rho^N$ is strongly relatively compact in $L^1([0,T]\times[a,b])$.
\end{thm}  

As a consequence, the desired compactness of $(u^N)_N$ relies on showing a uniform bound of the total variation and the equi-continuity of the $1-$Wasserstein distance with respect to the time variable. 

We first focus on the bounds on the Total Variation.

\begin{prop}[BV bound for Systems]\label{LemmastimaBV system}
Let $T>0$ and $\bar{u}$ be under assumptions $(In1)$ and $(In2)$. Then the discrete densities $u^N$ defined in \eqref{rhoN} satisfy
\begin{equation}\label{stimaBV system}
TV \left[u^N(t, \cdot)\right] \leq TV \left[\bar{u}\right]C, \qquad \forall \,\, t \in[0, T ],\,\, \forall\,\, N \in \mathbb{N},
\end{equation}
where the constant $C$ is such that $C=C(D_u,\phi_u, \Theta_u,T)$.
\end{prop}

\begin{proof}
The proof is based on a Gronwall type estimate for $TV [u^N]$.
To shorten the notation in the following computations we introduce the auxiliar functions 
\begin{equation*}
    s_i(t) = \sign(u_i(t)-u_{i+1}(t))-\sign(u_{i-1}(t)-u_i(t)),
\end{equation*}
for $i=1,...,N-1$. Standard computations lead to the following expression 
\begin{align*}
    & \frac{d}{dt}TV \left[u^N(t, \cdot)\right] = \dot{u}_0(t)+\dot{u}_{N-1}(t)+\sum_{i=1}^{N-2}\sign\left(u_{i+1}(t)-u_i(t)\right)\left(\dot{u}_{i+1}(t)-\dot{u}_i(t)\right)\\
    &= \dot{u}_0(t)+\dot{u}_{N-1}(t)
    + \sign(u_0(t)-u_1(t))\dot{u}^d_0(t) - \sign(u_{N-2}(t)-u_{N-1}(t))\dot{u}^d_{N-1}(t) \\
& \quad    +\sum_{i=1}^{N-2}\sign\left(u_{i+1}(t)-u_i(t)\right)\left(\dot{u}^{p}_{i+1}(t)-\dot{u}^{p}_i(t)\right) +\sum_{i=1}^{N-2}s_i(t)\dot{u}^d_i(t),
\end{align*}
where we used \eqref{eq:uidot1}. 
We first show that the boundary terms involving $\dot{u}_0$ and $\dot{u}_{N-1}$ are uniformly bounded with respect to $N$. Consider, for example, the contribution of the left boundary term
\[\dot{u}_0(t) + \sign(u_0(t) - u_1(t)) \dot{u}_0^d.\] 
When  
$u_0(t) \leq u_1(t)$, the diffusive part cancels out and only the term $\dot{u}_0^{p}$ survives.
If, instead, $u_1(t) \leq u_0(t)$, from the monotonicity of $\phi_u$
we deduce that 
\[
\dot{u}_0 + \sign(u_0 - u_1) \dot{u}_0^d = \dot{u}_0^{p} + 2\dot{u}_0^d = \dot{u}_0^{p} - 2 u_0 \frac{ D^2(W_1)(\phi_u(u_0) - \phi_u(u_1))}{\sigma_u^N(W_1 - W_0)} \leq \dot{u}_0^{p}.
\]
Then, in both the cases, it is enough to show that the non-local part of $\dot{u}_0$ is uniformly bounded. 
From the positivity of $P_{uh}$ and the bounds \eqref{eq: minmax system}, it is easy to see that 

\begin{align*}
\dot{u}_0^{p} &= \frac{u_0}{W_1 - W_0} \sum_{h \in \mathfrak{H}} \sigma_h^N \sum_{j=0}^{N} P_{uh}(W_1,H_j)(W_1 - H_j) \\
&\leq \frac{u_0}{W_1 - W_0} \sum_{h \in \mathfrak{H}} \sigma_h^N \sum_{j : H_j < W_1}P_{uh}(W_1,H_j)(W_1 - H_j) \\
&\leq \frac{u_0}{W_1 - W_0} \sum_{h \in \mathfrak{H}} \sigma_h^N \sum_{j : H_j < W_1}P_{uh}(W_1,H_j)(W_1 - W_0) \\
&\leq u_0 \sum_{h \in \mathfrak{H}} [\sigma_h \|P_{uh}\|_{L^\infty}].
\end{align*}

With a symmetric argument, one can see that the term also satisfies
\[\dot{u}_{N-1} - \sign(u_{N-2} - u_{N-1})\dot{u}_{N-1}^d \leq u_{N-1} \sum_{h \in \mathfrak{H}} [\sigma_h \|P_{uh}\|_{L^\infty}]. \] 

Let us now prove that  
\begin{equation}\label{D_i system}
   \sum_{i=1}^{N-2}s_i(t)\dot{u}_i^d(t) \leq 0. 
\end{equation}
If at time $t$ one has $u_{i+1}(t) \leq u_i(t) \leq u_{i-1}(t)$, or $u_{i+1}(t) \geq u_i(t) \geq u_{i-1}(t)$ then $s_i(t)=0$. When, instead, $u_{i \pm 1}(t)$ are both bigger or both smaller than $u_i(t)$, the monotonicity of $\phi_u$ implies the desired 
estimate. Indeed, if we assume for example that $u_{i \pm 1}(t) < u_i(t)$, then $s_i(t)= 2$ and, recalling the definition of $\dot{u}_i^d$, we get 
\[
s_i\dot{u}_i^d = 2 \frac{u_i^2}{\left(\sigma_u^N\right)^2} \left[D_u^2(W_{i}) (\phi_u(u_{i-1}) - \phi_u(u_{i}))
- D_u^2(W_{i+1})(\phi_u(u_i) - \phi_u(u_{i+1})) \right] <0.
\]
The other case $u_{i \pm 1}(t) > u_i(t)$ follows analogously. To conclude, we are left to estimate the term concerning the non-local contribution. Using the definition of $\dot{u}_i^{p}$ in\eqref{eq:uidot2}, we deduce that 
\begin{align}\label{nonlocal 1}
\nonumber
    &\sum_{i=1}^{N-2} \sign(u_{i+1}(t) - u_i(t))(\dot{u}_{i+1}^{p}(t) - \dot{u}_i^{p}(t)) \\
\nonumber
  &\leq \sum_{i=1}^{N-2} |u_{i+1}(t) - u_i(t)| \frac{|\dot{W}_{i+2}^{p}(t) - \dot{W}_{i+1}^{p}(t)|}{W_{i+2}(t) - W_{i+1}(t)} \\
  &+ \sum_{i=1}^{N-2} u_i(t) \left| \frac{\dot{W}_{i+2}^{p}(t) - \dot{W}_{i+1}^{p}(t)}{W_{i+2}(t) - W_{i+1}(t)} - \frac{\dot{W}_{i+1}^{p}(t) - \dot{W}_{i}^{p}(t)}{W_{i+1}(t) - W_{i}(t)} \right|.
\end{align}
Thanks to the assumption (P), we can easily bound the term in \eqref{nonlocal 1} with the total variation of $u^N$
\begin{equation}\label{nonlocal 3}
  \sum_{i=1}^{N-2} |u_{i+1}(t) - u_i(t)| \frac{|\dot{W}_{i+2}^{p}(t) - \dot{W}_{i+1}^{p}(t)|}{W_{i+2}(t) - W_{i+1}(t)}  \leq  \Theta_{u}TV[u^N(t,\cdot)] . 
\end{equation}
On the other hand, since the functions $P_{uh}$ are Lipschitz, there exist some $\bar{W}_{i+1}^{i+2} \in [W_{i+1}, W_{i+2}]$ and $\bar{W}_{i}^{i+1} \in [W_{i}, W_{i+1}]$ such that
\begin{align*}
    &\frac{\dot{W}_{i+2}^{p} - \dot{W}_{i+1}^{p}}{W_{i+2} - W_{i+1}} - \frac{\dot{W}_{i+1}^{p} - \dot{W}_{i}^{p}}{W_{i+1} - W_{i}} = \\
    = &- \sum_{h \in \mathfrak{H}} \sigma_h^N \sum_{j=0}^{N} \frac{P_{uh}(W_{i+2},H_j) (W_{i+2} - H_j) - P_{uh} (W_{i+1},H_j)(W_{i+1} - H_j)}{W_{i+2} - W_{i+1}} \\
    & +  \sum_{h \in \mathfrak{H}} \sigma_h^N \sum_{j=0}^{N} \frac{P_{uh}(W_{i+1},H_j) (W_{i+1} - H_j) - P_{uh} (W_{i},H_j)(W_{i} - H_j)}{W_{i+1} - W_i}\\
    \leq &\sum_{h \in \mathfrak{H}} \sigma_h^N \sum_{j=0}^{N} \left|  \partial_1 P_{uh}(\bar{W}_{i+1}^{i+2},H_j)(\bar{W}_{i+1}^{i+2} - H_j) - \partial_1 P_{uh}(\bar{W}_{i}^{i+1},H_j)(\bar{W}_{i}^{i+1} - H_j)\right|  \\
    &+ \sum_{h \in \mathfrak{H}} \sigma_h^N \sum_{j=0}^{N} \big| P_{uh}(\bar{W}_{i+1}^{i+2},H_j)  - P_{uh}(\bar{u}_{i}^{i+1},H_j) \big| \\
    \leq &(W_{i+2} -W_i)  \sum_{h \in \mathfrak{H}} \big(2 Lip[\partial_1 P_{uh}] + \| \partial_1 P_{uh}\|_{L^\infty} + Lip[P_{uh}]\big).
\end{align*}
Using now the upper bound of \eqref{eq: minmax system} in the above estimate, we can handle also the second term of \eqref{nonlocal 1} and get 
\begin{equation}\label{nonlocal 4}
    \sum_{i=1}^{N-2} u_i(t) \left| \frac{\dot{W}_{i+2}^{p}(t) - \dot{W}_{i+1}^{p}(t)}{W_{i+2}(t) - W_{i+1}(t)} - \frac{\dot{W}_{i+1}^{p}(t) - \dot{W}_{i}^{p}(t)}{W_{i+1}(t) - W_{i}(t)} \right| \leq 4\frac{e^{c_u T}}{m_u} \Theta_u.
\end{equation}

Summarizing, \eqref{nonlocal 3} and \eqref{nonlocal 4}, together with \eqref{D_i system} and the estimates on the boundary terms, imply that 
$$
\frac{d}{dt} TV[u^N(t,\cdot)] \leq \Theta_u\left(4\frac{e^{c_u T}}{m_u} +   + TV[u^N(t,\cdot)] \right)
$$
and the conclusion follows via a Gronwall type argument. 
\end{proof}

We then show that the sequence $u^N(t,w)$ satisfies and equi-continuity in time with respect to the $1$-Wasserstein distance, the proof realise on a, right now, quite standard argument introduced in \cite{BUMI}, that we report here for completeness.  

\begin{prop}\label{Lemmacontinuity in 1W system}
Let $u^{N}$ be any of $f^{N},l^{N},r^{N},q^{N}$, defined in \eqref{rhoN}.
Let $T > 0$ and $\bar{u}$ under assumptions $(In1)$ and $(In2)$. If $D_u$, $\phi_u$, $P_{uh}$ are under the assumptions $(D1)$, $(D2)$, $(Dif)$ and $(P)$ respectively, then there exists a constant $C>0$ such that 
\begin{equation}\label{continuity in 1W system}
    d_{W^1}(u^{N}(t,\cdot),u^{N}(s,\cdot)) \leq C|t-s| \qquad \mbox{for all $s,t \in [0,T]$ and $N \in \mathbb{N}$.}
\end{equation}
\end{prop}

\begin{proof}
The proof of this result is very standard in the one dimensional setting, and we will use the tools introduced in  Section \ref{sec:opt_tra}. The pseudo-inverse map associated to the piecewise constant probability density $u^N$ can be very easily computed and it is given by
\[
  X_{u^{N}(t,\cdot)}(x) = \sum_{i=0}^{N-1} \left( W_i(t) + \left(w - \sigma_u^N i\right) \frac{1}{u_i(t)}\right)\chi_{J_u^N(i)}(w),
\]
where $J_u^N(i)=\left[\sigma_u^N i, \sigma_u^N(i+1)\right)$. For any $t>s$ we have
\begin{align*}
    d_{W^1}(u^{N}(t,\cdot), &u^{N}(s,\cdot)) = \| X_{u^{N}(t,\cdot)} - X_{u^{N}(s,\cdot)}\|_{L^1([0,1])} \\
    &\leq \sum_{i=0}^{N-1} \int_{J_u^N(i)} \left|W_i(t) - W_i(s) + \left(w - \sigma_u^N i\right)\big(\frac{1}{u_i(t)} - \frac{1}{u_i(s)}\big)\right|dw \\
    &\leq \sigma_u^N\sum_{i=0}^{N-1}|W_i(t) - W_i(s)| + \frac{(\sigma_u^N)^2}{2} \sum_{i=0}^{N-1} \int_s^t \left| \frac{d}{d\tau}\frac{1}{u_i(\tau)}\right| d\tau \\
    &\leq 3\sigma_u^N \sum_{i=0}^{N -1} \int_s^t |\dot{W}_i(\tau)|d\tau\\
    &\leq 3 \|D_u^2\|_{L^\infty} Lip[\phi_u] (t-s)\sum_{i=0}^{N-1} |W_{i+1} - W_i| + 3 (t-s)\Theta_{u} \\
    &\leq C(D_u,\phi_u,\bar{u},P_{uh}) (t-s).
\end{align*}
where we have used the assumption on the uniform bound of both on the interaction potentials $P_{uh}$ and the diffusivity $D_u$, and the estimate \eqref{stimaBV system}.
\end{proof}

Gathering together Theorem \ref{Aubin Lions} with
Propositions \ref{LemmastimaBV system} and \ref{Lemmacontinuity in 1W system}, we conclude the following.

\begin{cor}\label{cmpctness}
Let $u^{N}$ be any of $f^{N},l^{N},r^{N},q^{N}$, defined in \eqref{rhoN}. Then there exists some $u \in L^1 \cap L^{\infty}(I_T)$ such that $\|u^{N} - u\|_{L^1} \to 0$ as $N \to \infty$.
\end{cor}

We now focus on the identification of the limit given by Corollary \ref{cmpctness} as weak solution of the system \eqref{main_comp}. 
This will be, indeed, a straightforward consequence of Propositions \ref{prop convergenza una specie} and \ref{Propconsistencysystem}.

\begin{prop}\label{prop convergenza una specie}
Let $T > 0$ and $\bar{u}$ under assumptions $(In1)$ and $(In2)$. Let $D_u$, $\phi_u$, $P_{uh}$ be under the assumptions $(D1)$, $(D2)$, $(Dif)$ and $(P)$ respectively. Consider $u^{N}$ and $h^N$ be any two sequences among $\{f^{N},l^{N},r^{N},q^{N}\}$ and the respective $L^1$-strong limits $u,\,h$ provided by Corollary \ref{cmpctness}. Then for every $\zeta \in C^{\infty}_0((0,T)\times (-1,1))$ we get
\begin{equation}\label{convergenza una specie}
\begin{split}
&\lim_{N \to \infty} \int_0^T \int_I u^{N} \partial_t \zeta = \int_0^T \int_I u \partial_t\zeta, \\
&\lim_{N \to \infty} \int_0^T \int_I \phi_u (u^{N}) \partial_w(D_u^2(w)\partial_w\zeta) = \int_0^T \int_I \phi_u(u)\partial_w(D_u^2(w)\partial_w\zeta), \\
&\lim_{N \to \infty} \int_0^T \int_I u^{N}\sum_{h \in \mathfrak{H}}\mathcal{P}_{uh}[\hat{h}^{N}]\partial_w\zeta =  \int_0^T \int_I u\sum_{h \in \mathfrak{H}}\mathcal{P}_{uh}[h]\partial_w\zeta, 
\end{split}
\end{equation}
where $\hat{h}^{N} (t,w) := \frac{1}{N}\sum_{j=0}^{N}  \delta_{H_j(t)}(w)$ indicates the sequence of empirical measures. 

\end{prop}
\begin{proof}
The first two of \eqref{convergenza una specie} are follow as direct consequence of
the strong $L^1$-compactness obtained in Corollary \ref{cmpctness}, together with the Lipschitz regularity of the nonlinear diffusion $\phi_u$ and the uniform bound of $\partial_w D^2$ on $I$. Concerning the third part, we need to first show that the empirical measures $\hat{h}^{N}$ and the piecewise constant densities $h^{N}$ share the same limit $h$ with respect to a suitable topology. 
Indeed, it is possible to prove that 
\begin{equation}\label{empiriche constants}
   d_{W^1} \left(h(t,\cdot), \hat{h}^{N}(t,\cdot) \right) \longrightarrow 0 \quad \mbox{ as  $N \to \infty$ for all $t \in [0,T]$}.  
\end{equation}
This is follows by the identity  $d_{W^1}(\mu,\nu) = \| X_\mu - X_\nu\|_{L^1}$ between probability measures. Observing that the pseudo-inverse mapping of an empirical measure is piecewise constant, it is easy to see that 
\begin{align*}
   \| X_{h^{N}(t,\cdot)} - X_{\hat{h^{N}}(t,\cdot)}\|_{L^1([-1,1])}  &\leq \sum_{i=0}^{N -1} \int_{\sigma_h^N i}^{\sigma_h^N(i+1)} \left|\left( w - \sigma_h^N \right)\frac{1}{h_i(t)} \right|dw \\
   &= \frac{\sigma_h^N}{2} \sum_{i=0}^{N -1} (H_{i+1}(t) - H_i(t)) \leq \sigma_h^N.
\end{align*}
Recalling that the sequence $h^N$ also converges to $h$ with respect to the $1$-Wasserstein distance, we deduce
\[
d_{W_1}(h(t,\cdot),\hat{h}^{N}(t,\cdot)) \leq d_{W_1}(h(t,\cdot),h^{N}(t,\cdot)) + d_{W_1}(h^{N}(t,\cdot),\hat{h}^{N}(t,\cdot)) \leq C \sigma_h^N,
\]
for some positive geometric constant $C$, and then \eqref{empiriche constants} follows. We are now in position to prove the third limit in \eqref{convergenza una specie}, which is equivalent to show the following 
\[
    \int_0^T \int_I (u^{N} - u) \sum_{h \in \mathfrak{H}} \mathcal{P}_{uh}[\hat{h}^{N}] \partial_w \zeta \to 0, \]
    and 
    \[  \int_0^T \int_I u \sum_{h \in \mathfrak{H}}  (\mathcal{P}_{uh}[\hat{h}^{N}] - \mathcal{P}_{uh}[h]) \partial_w \zeta \to 0.
\]
The first limit is immediate because of the boundedness assumption on any of the $ P_{uh}$. On the other hand, for the second limit, we can always consider an optimal plan $\gamma_h^{N}(t)$ between the probability measures $\hat{h}^{N}(t,\cdot)$ and $h^{N}(\cdot)$ with respect to the cost $|\cdot|$. Then for every $t \geq 0$ fixed, calling $C_1 = \| \partial_w \zeta\|_{L^\infty} M_u e^{c_u T}$ we get 
\begin{align*}
    &\int_I \left| u \sum_{h \in \mathfrak{H}}  (\mathcal{P}_{uh}[\hat{h}^{N}] - \mathcal{P}_{uh}[h]) \partial_w \zeta\right| dw \\
    \leq & C_1 \sum_{h \in \mathfrak{H}}  \int_I \left|\int_I P_{uh}(w,v)(v-w)d h^{N}(t,v)  - \int_I P_{uh}(w,v')(v'-w)d \hat{h}^{N}(t,v')\right| dw \\
    = & C_1\sum_{h \in \mathfrak{H}} \int_I \left| \int_{I \times I} \big(P_{uh}(w,v)(v-w) - P_{uh}(w,v')(v'-w)\big) d\gamma_h^{N}(t)(v,v') \right|dw \\
    \leq & C_1 \sum_{h \in \mathfrak{H}} (Lip[P_{uh}] + \|P_{uh}\|_{L^\infty}) \int_I   \int_{I \times I} |v- v'| d\gamma_h^{N}(t)(v,v') dw \\
    \leq & C_1 \Theta_{u}d_{W_1}(h^N(t,\cdot),\hat{h}^{N}(t,\cdot))\,\leq\, C_1 \Theta_{u}\sigma_h^N.
\end{align*}
Then, 
$$
\int_0^T \int_I \left| u \sum_{h \in \mathfrak{H}}  (\mathcal{P}_{uh}[\hat{h}^{N}] - \mathcal{P}_{uh}[h]) \partial_w \zeta \right| dw dt \leq C_1 T \Theta_u \frac{1}{N} \sum_{h \in \mathfrak{H}}\sigma_h
$$
and this concludes the proof.
\end{proof}

The remaining part of this section is devoted to prove that the sequence $u^N$ satisfies the the weak formulation of \eqref{main_comp} in the limit as $N\to\infty$.

\begin{prop}\label{Propconsistencysystem}
Under the same assumptions of Proposition \ref{prop convergenza una specie}, for every $\zeta \in C^{\infty}_c ((0,T)\times(-1,1))$, one has 
\begin{equation}\label{consistency}
    \lim_{N \to \infty} \int_0^T \int_I u^{N} \partial_t \zeta + \phi_u(u^{N}) \partial_w\big(D_u^2(w) \partial_w\zeta\big)  - u^{N} \sum_{h \in \mathfrak{H}} \mathcal{P}_{uh}\left[\hat{h}^{N} \right] \partial_w\zeta dwdt = 0.
\end{equation}
\end{prop}

\begin{proof}
For simplicity we will denote the three contributions on the r.h.s. of \eqref{consistency} as follows
\begin{equation}
\begin{split}
 &I:=  \int_0^T \int_I u^{N} \partial_t\zeta dw dt, \quad  II:= \int_0^T \int_I \phi(u^{N})\partial_w(D^2(w)\partial_w \zeta) dw dt,\\
 & III:= -\int_0^T \int_I u^{N} \sum_{h \in \mathfrak{H}}\mathcal{P}_{uh}[\hat{h}^{N} ]\partial_w\zeta dw dt.
 \end{split}
\end{equation}
Recalling the definition of $u^{N}$ and performing some standard computations, as discrete integration by parts and reconstruction of the derivative, it is easy to rewrite 
\begin{align*}
&I = \sum_{i=0}^{N -1} \int_0^T \int_{W_i(t)}^{W_{i+1}(t)} u_i^{N}(t) \partial_t \zeta (t,w) dw\,dt \\
&\quad = \sum_{i=0}^{N-1} \int_0^T  u_i^N(t) \dot{W}_{i+1}(t)\int_{W_i(t)}^{W_{i+1}(t)}  \left(\intmed_{W_i(t)}^{W_{i+1}(t)}\zeta(t,w) - \zeta(t,W_{i+1}(t))\right)dwdt \\
&\qquad - \sum_{i=0}^{N-1} \int_0^T u_i^N(t) \dot{W}_{i}(t) \int_{W_i(t)}^{W_{i+1}(t)} \left(\intmed_{W_i(t)}^{W_{i+1}(t)}\zeta(t,w) - \zeta(t,W_{i}(t))\right) dw dt,
\end{align*}
thus, expanding $\zeta(t,w)$ at first order with respect to $W_{i+1}(t)$ and $W_i(t)$ respectively, we obtain that for some $\alpha^i_w \in [w,W_{i+1}(t)]$ and $\beta^i_w \in [W_i(t),w]$ for which 
\begin{align*}
  I &= \frac{\sigma_u^N}{2}\sum_{i=0}^{N-1}\int_0^T [-\partial_w \zeta(t,W_{i+1}(t))\dot{W}_{i+1}(t) - \partial_w \zeta(t,W_i(t))\dot{W}_i(t)]dt\\
  &\quad + \frac{1}{2}\sum_{i=0}^{N-1}\int_0^T u^N_i(t) \dot{W}_{i+1}(t) \int_{W_i(t)}^{W_{i+1}(t)} \partial_{ww}\zeta(t,\alpha^i_w)(w-W_{i+1}(t))^2 dw dt\\
  &\quad - \frac{1}{2}\sum_{i=0}^{N-1}\int_0^T u^N_i(t) \dot{W}_{i}(t) \int_{W_i(t)}^{W_{i+1}(t)} \partial_{ww}\zeta(t,\beta^i_w)(w-W_{i}(t))^2 dw dt.
\end{align*}
Recalling that $\dot{W}_i(t) = \dot{W}_i^d(t) + \dot{W}_i^p(t)$, we can split the contribution of the first term of the r.h.s. in a diffusive part and a non local part. Then we can rewrite $I$ as the sum of the following three terms
\begin{align*}
    &I_1:= -\frac{\sigma_u^N}{2}\sum_{i=0}^{N-1}\int_0^T [\partial_w \zeta(t, W_{i+1}(t))\dot{W}^d_{i+1}(t) + \partial_w \zeta(t,W_i(t))\dot{W}^d_i(t)] dt \\
    &I_2:= -\frac{\sigma_u^N}{2}\sum_{i=0}^{N-1}\int_0^T [\partial_w \zeta(t,W_{i+1})\dot{W}^p_{i+1}(t) + \partial_w \zeta(t,W_i(t))\dot{W}^p_i(t)]dt \\
    &I_3:= \frac{1}{2}\sum_{i=0}^{N-1}\int_0^T u^N_i(t) \dot{W}_{i+1}(t) \int_{W_i(t)}^{W_{i+1}(t)} \partial_{ww}\zeta(t,\alpha^i_w)(w-W_{i+1}(t))^2\\
    &\qquad - \frac{1}{2}\sum_{i=0}^{N-1}\int_0^T u^N_i(t) \dot{W}_{i}(t) \int_{W_i(t)}^{W_{i+1}(t)} \partial_{ww}\zeta(t,\beta^i_w)(w-W_{i}(t))^2.
\end{align*}
In the sequel we do not explicit the dependence on $N$ or on the variable $t$ whenever this is clear from the context. For sake of clarity, we divide the rest of the proof in three steps.

\textbf{Step 1.} We first prove that $I_1+II = 0$. Standard algebraic computations and discrete integration by parts lead to 
    \begin{align*}
        I_1 + II &= \sum_{i=2}^{N-2} \int_0^T \partial_w \zeta(t,W_i) D_u^2(W_i) [\phi_u(u_i) - \phi_u(u_{i-1})]dt \\
        &\quad+ \sum_{i=1}^{N-1} \partial_w \zeta(t,W_i) D_u^2(W_i)(\phi_u(u_{i-1}) - \phi_u(u_i)) \\
        &\quad + \frac{1}{2}\int_0^T [\partial_w\zeta(t,W_N)D_u^2(W_N)\phi_u(u_{N-1}) - \partial_w\zeta(t,W_0) D_u^2(W_0)\phi_u(u_0)] dt \\
        &\quad -\int_0^T \partial_w\zeta(t,W_{N-1})D_u^2(W_{N-1})(\phi_u(u_{N-2})-\phi_u(u_{N-1}))dt \\
        &\quad-\int_0^T \partial_w\zeta(t,W_1)D_u^2(W_1)(\phi_u(u_0) - \phi_u(u_1))dt \\
        &= \frac{1}{2}\int_0^T \partial_w\zeta(t,Z_N)D^2(Z_N)\phi_u(W_{N-1}) - \partial_w\zeta(t,W_0) D_u^2(W_0)\phi_u(u_0)dt = 0,
    \end{align*}
    where the last equality holds because $\partial_w \zeta(t,W_0) = \partial_w \zeta(t,W_N) = \partial_w \zeta(t, \pm 1) = 0$ for all $t \in [0,T]$.

\textbf{Step 2.} In this step we show that $I_2+III$ vanishes as $N\to\infty$. A discrete integration by parts, together with the fact that $\partial_w \zeta(t,W_0) = \partial_w \zeta(t,W_N) = 0$, imply
\begin{align*}
       I_2 + III &= \sum_{i=1}^{N-1} \sum_{h \in \mathfrak{H}} \sum_{j=0}^{N} \sigma_u^N \sigma_h^N \int_0^T \partial_w \zeta(t,W_i) P_{uh}(W_i,H_j)(W_i-H_j)dt \\
       &\quad - \sum_{i=0}^{N-1} \sum_{h \in \mathfrak{H}}\sum_{j=0}^{N} \sigma_h^N \int_0^T u_i \int_{W_i}^{W_{i+1}} \partial_w \zeta(t,w) P_{uh}(w,H_j)(w-H_j)dt.
    \end{align*}     
    Then, expanding $\partial_w \zeta$ at first order with respect to $W_i$ and still using that $\zeta$ has compact support in $(-1,1)$, there is $\gamma^i_w \in [W_i,w]$ such that
\begin{align*}
      & I_2 + III = \\
    = & \sum_{i=1}^{N-1} \sum_{h \in \mathfrak{H}} \sum_{j=0}^{N} \sigma_h^N \int_0^T u_i \partial_w \zeta(t,W_i) \int_{W_i}^{W_{i+1}} P_{uh}(W_i,H_j)(W_i - H_j)  dw dt\\
     & - \sum_{i=1}^{N-1} \sum_{h \in \mathfrak{H}} \sum_{j=0}^{N} \sigma_h^N \int_0^T u_i \partial_w \zeta(t,W_i) \int_{W_i}^{W_{i+1}} P_{uh}(w,H_j)(w-H_j) dw dt \\
      & -  \sum_{i=1}^{N-1} \sum_{h \in \mathfrak{H}} \sum_{j=0}^{N} \sigma_h^N\int_0^T u_i \int_{W_i}^{W_{i+1}} P_{uh}(w,H_j)(w-H_j)(w-W_i)\partial_{ww}\zeta(t,\gamma^i_w) dwdt.
\end{align*}
    Since $P_{uh}$ is Lipschitz in the first component, we obtain 
    \begin{align*}
       |I_2 + III| &\leq  C \sum_{h \in \mathfrak{H}} Lip[P_{uh}]\sum_{i=1}^{N-1}\int_0^T \left[u_i (W_{i+1}-W_i)^2 + u_i (W_{i+1}-W_i)^2 \right] dt\\
       &\leq 2CT\, \sigma_u^N \sum_{h \in \mathfrak{H}}Lip[P_{uh}] 
    \end{align*}
where $C = \big(3 \| \zeta_w \|_{\infty} + \| \zeta_{ww} \|_{\infty}\big)$,  thus the term $I_2 + III \to 0$ as $N \to \infty$.

\textbf{Step 3.}  In this last step we prove that $I_3$ also vanishes as $N\to\infty$.  For simplicity we introduce the rest functions 
    \begin{align*}
    &\mathcal{R}_{i,\alpha}(t)= \int_{W_i}^{W_{i+1}} \partial_{ww}\zeta(t,\alpha^i_w)(w-W_{i+1})^2 dw, \\ &\mathcal{R}_{i,\beta}(t)= \int_{W_i}^{W_{i+1}} \partial_{ww}\zeta(t,\beta^i_w)(w-W_{i+1})^2 dw,
    \end{align*}
    
    in particular, thanks to upper bound of \eqref{eq: minmax system}, we observe
   \begin{equation}\label{stime sul resto}
       |\mathcal{R}_{i,\alpha}|,|\mathcal{R}_{i,\beta}| < \| \partial_{ww} \zeta\|_{L^\infty}(W_{i+1} - W_i)^3 < \| \partial_{ww} \zeta\|_{L^\infty} \left(\frac{e^{c_u T}}{m_u}\right)^3 (\sigma^N_u)^3.
    \end{equation}
   Using $\mathcal{R}_{i,\alpha/\beta}(t)$ we can rewrite $I_3$ as 
    \begin{align*}
        I_3 &= \frac{1}{2}\sum_{i=0}^{N-1}\int_0^T u_i [(\dot{W}_{i+1}^d - \dot{W}_i^d) + (\dot{W}_{i+1}^p - \dot{W}_i^p)]\mathcal{R}_{i,\alpha} +  u_i \dot{W}_{i}
        [\mathcal{R}_{i,\alpha} - \mathcal{R}_{i,\beta}]   dt \\
        &\quad +\frac{1}{2}\sum_{i=0}^{N-1}\int_0^T u_i \dot{W}_{i} \int_{W_i}^{W_{i+1}}
        \partial_{ww}\zeta(t,\beta^i_w)[(w-W_{i+1})^2 - (w-W_{i})^2] dw dt
    \end{align*}
    and by replacing the expression of $\dot{W}_i^d$ and $\dot{W}_i^p$, we obtain
    \begin{align*}
        I_3 &= \frac{1}{2\sigma_u^N}\sum_{i=1}^{N-2} \int_0^T u_i (D_u^2(W_{i+1}) -D_u^2(W_i))(\phi_u(u_i) - \phi_u(u_{i+1})) \mathcal{R}_{i,\alpha}dt \\
        &\,\, + \frac{1}{2\sigma_u^N}\sum_{i=1}^{N-2} \int_0^T u_i D_u^2(W_i)[(\phi_u(u_i) - \phi_u(u_{i+1}))-(\phi_u(u_{i-1})-\phi_u(u_i))]\mathcal{R}_{i,\alpha}dt 
        \\
        & + \frac{1}{2} \sum_{i=1}^{N-2} \sum_{h \in \mathfrak{H}} \sum_{j=0}^{N} \sigma_h^N \int_0^T u_i [P_{uh}(W_{i+1},H_j)-P_{uh}(W_i,H_j))(H_j- W_{i+1}]\mathcal{R}_{i,\alpha}dt 
        \\
        &+ \frac{1}{2} \sum_{i=1}^{N-2} \sum_{h \in \mathfrak{H}} \sum_{j=0}^{N} \sigma_h^N \int_0^T u_i P_{uh}(W_i,H_j)(W_i - H_j)\mathcal{R}_{i,\alpha}(t)dt 
        \\
        &  + \frac{1}{2}\int_0^T \left[ u_0 \dot{W}_1 \mathcal{R}_{0,\alpha}  -  u_{N-1}\dot{W}_{N-1}\mathcal{R}_{N-1,\alpha}
         + \sum_{i=0}^{N-1} u_i \dot{W}_{i}(\mathcal{R}_{i,\alpha} - \mathcal{R}_{i,\beta}) \right] dt\\
        & +\frac{1}{2}\sum_{i=0}^{N-1}\int_0^T u_i \dot{W}_{i} \int_{W_i}^{W_{i+1}}
        \partial_{ww}\zeta(t,\beta^i_w)[(w-W_{i+1})^2 - (w-W_{i})^2] dw dt.
    \end{align*}
Thanks to estimates \eqref{stime sul resto} and \eqref{stimaBV system}, and the fact that 
$$
|\dot{W}_i| < 2\frac{Lip[D_u]}{\sigma^N_u} + \sum_{h \in \mathfrak{H}}\|P_{uh}\|_{L^\infty} \qquad \mbox{ for all } i\ in \{0, \ldots, N\},
$$
we deduce that all the terms of $I_3$ except the last one can be estimated from above by $C\sigma_u^N$, where $C$ is some positive constant depending on $\zeta,\,T$ and on the proper bounds on $D_u,\,\phi_u,\, P_{uh}$ provided by the assumptions (D1), (D2), (Dif) and (P).
Concerning the last term of $I_3$, standard computations, the above estimate on $|\dot{W}_i|$ and \eqref{eq: minmax system} and \eqref{eq: minmaxg system} bring to 
    \begin{align*}
        &\frac{1}{2}\sum_{i=0}^{N-1} \int_0^T u_i \dot{W}_i \int_{W_i}^{W_{i+1}} \partial_{ww}\zeta(t,\beta_w^i)[(w-W_{i+1})^2 - (w- W_i)^2]dwdt \\
        \leq &  2\sum_{i=0}^{N-1} \int_0^T u_i \dot{W}_i \int_{W_i}^{W_{i+1}}\partial_{ww}\zeta \big[(w - W_{i+1})(W_{i+1}-W_i) +(W_{i+1} - W_i)^2 \big]dwdt  \\
        \leq & 2\|\partial_{ww}\zeta\|_{\infty} \sum_{i=0}^{N -1} \int_0^T u_i |\dot{W}_i| (W_{i+1} - W_i)^3dt \\
        \leq & \|\partial_{ww}\zeta\|_{\infty}   \frac{2 M_u e^{4c_u T}}{m_u^3}(\sigma_u^N)^3 \sum_{i=0}^{N -1} \int_0^T |\dot{W}_i| dt \\
        \leq & \|\partial_{ww}\zeta\|_{\infty}   \frac{2 T \sigma_u M_u e^{4 c_u T}}{m_u^3}  \left(2Lip[D_u] + \sum_{h \in \mathfrak{H}} \| P_{uh}\|_{L^\infty}\right) \sigma_u^N. 
    \end{align*}
    
Finally, there exists some constant $C$ depending on the data of the problem and on the function $\zeta$ so that 
    \[ |I_3| \leq C \sigma_u^N, \]
and this concludes the proof of the Step 3 and of Proposition \ref{Propconsistencysystem}. 
\end{proof}

\section{Large-time behaviour}\label{sec:num}
This last section is devoted to the study of the large-time behaviour for the model \eqref{main}. The main difficulty of such study concerns the evolution of the mean opinion which is, unfortunately, not close in general. The mean opinion corresponds to the first moment function 
\[
 m_1(t)=\int_I v u(t,v)dv.
\]
Consider, for example, the easiest case of one species, with $P(w,v)=1$ in \eqref{main}. Then the compromise part corresponds to
\begin{equation}\label{eq:puno}
\oP = \int_I (v-w)u(t,v)dv = m_1(t)-\sigma w,
\end{equation}
and $m_1(t)$ evolves according to
\begin{align*}
     \frac{d}{dt}m_1(t)&= \int_I w \partial_w\left(\frac{\lambda^2}{2}D^2(w)\partial_w \phi(u(t,w)) - \left(m_1(t)-\sigma w \right)u(t,w) \right) dw\\
     &=- \int_I \left(\frac{\lambda^2}{2}D^2(w)\partial_w \phi(u(t,w)) - \left(m_1(t)-\sigma w \right)u(t,w) \right) dw\\
   &= - \int_I \frac{\lambda^2}{2}D^2(w)\partial_w \phi(u(t,w))dw.
\end{align*}
It is then evident that the evolution of $m_1$ is independent on higher order moments of $u$ and it strongly depend on the choice of the mobility $D$. As a consequence, the exact evaluation of the limit 
\begin{equation}\label{eq:limit_m}
    m_1^\infty = \lim_{t\to\infty} m_1(t),
\end{equation}
is quite difficult to investigate, we refer to \cite{AlNaTo} for a more detailed discussion on this topic.
\subsection{Stationary states for the single species model}
We start investigating the stationary states for equation \eqref{main} under the choice $P=1$, as in \eqref{eq:puno}. 
 We assume that the limit value $m_1^\infty$ exist, than the equation for the stationary states writes as
\begin{equation}\label{eq:stationary}
   \partial_w\left(\frac{\lambda^2}{2}D^2(w)\partial_w \phi(u(w)) -\left(m_1^\infty-\sigma w \right)u(w) \right) = 0.
\end{equation}
Moreover, to simplify the following study, we introduce the quantity
\begin{equation}\label{eq:primD}
    \mathfrak{D}_\alpha(w,m_1^\infty)=\int \frac{\left(m_1^\infty-\sigma w \right)}{D^2(w)} d
\end{equation}
and we distinguish between the two sub-cases of linear and nonlinear diffusion.
\subsubsection{Stationary states in the case of linear diffusion}
\begin{figure}[htbp]
\begin{center}
\begin{multicols}{2}
    \includegraphics[width=6cm,height=5cm]{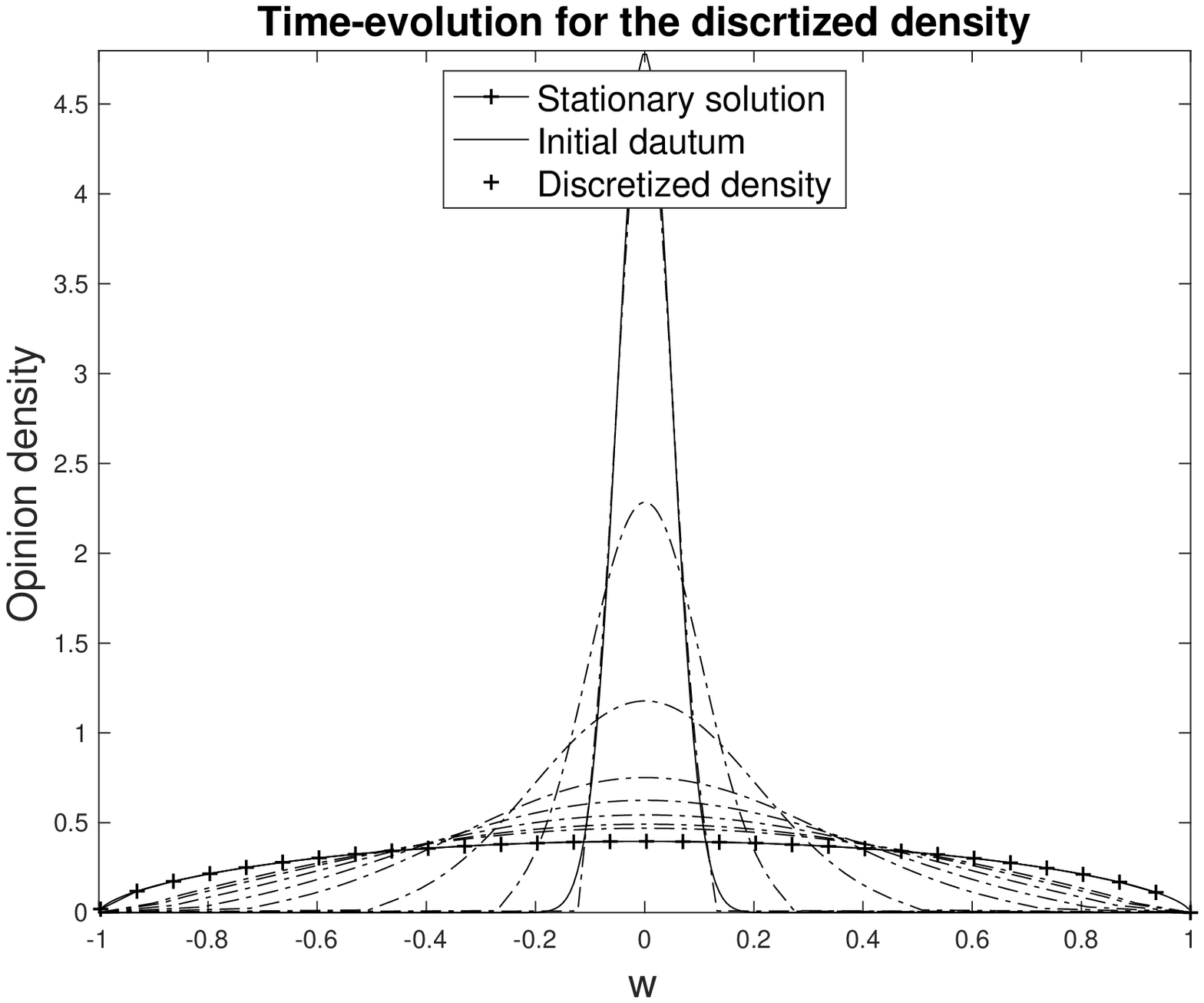}\par 
    \includegraphics[width=6cm,height=5cm]{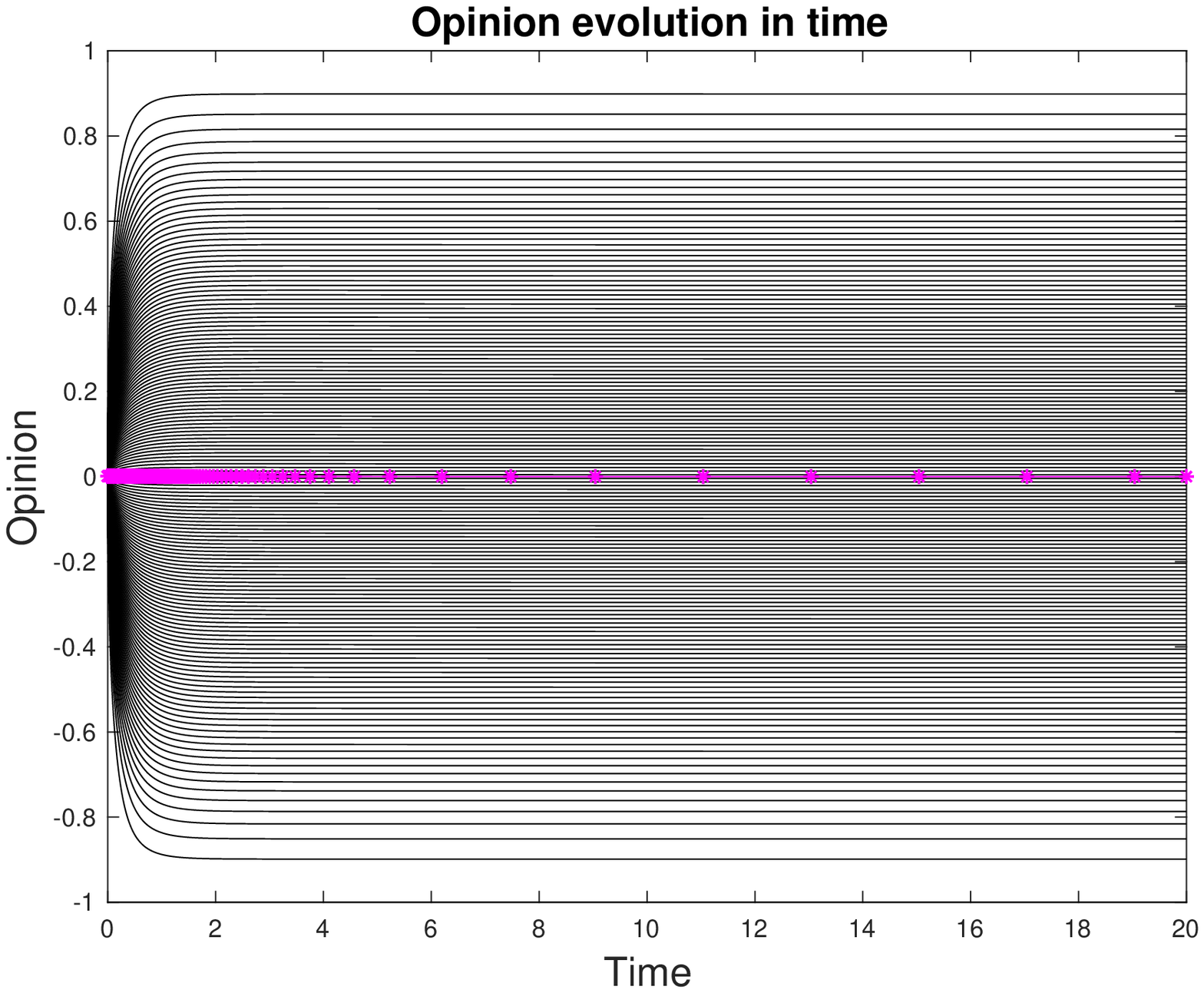}\par 
\end{multicols}
\begin{multicols}{2}
    \includegraphics[width=6cm,height=5cm]{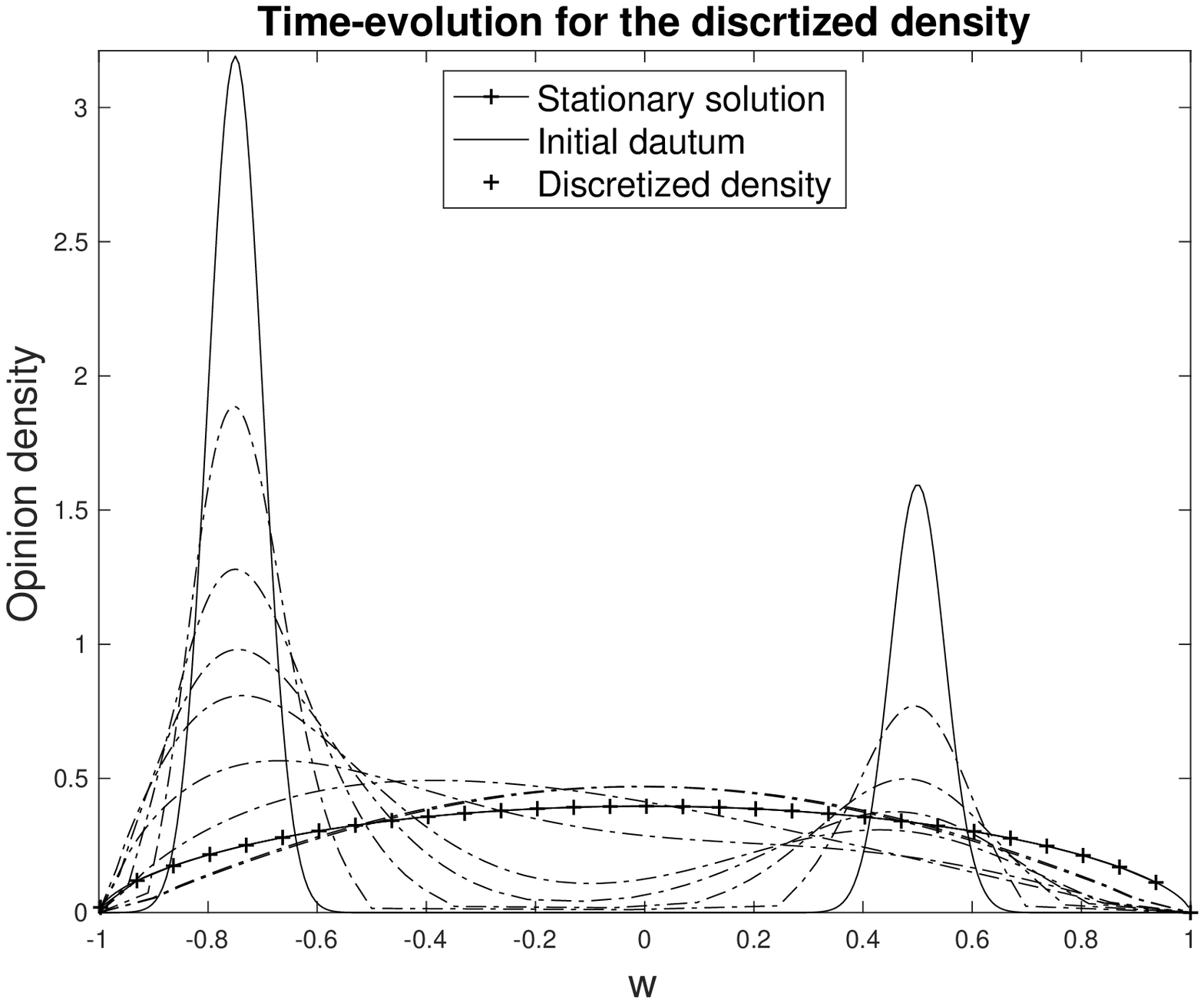}\par
    \includegraphics[width=6cm,height=5cm]{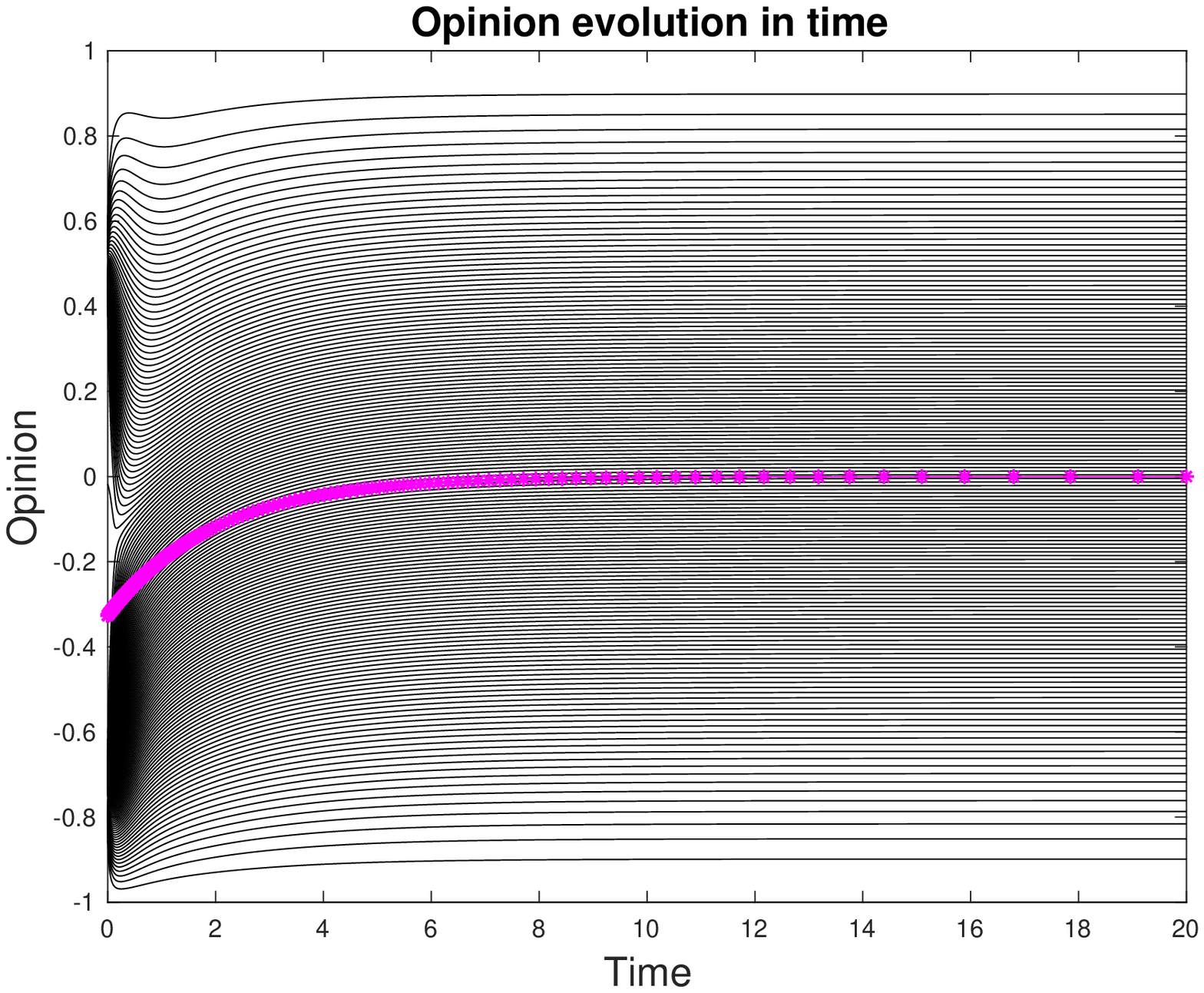}\par
\end{multicols}
\begin{multicols}{2}
    \includegraphics[width=6cm,height=5cm]{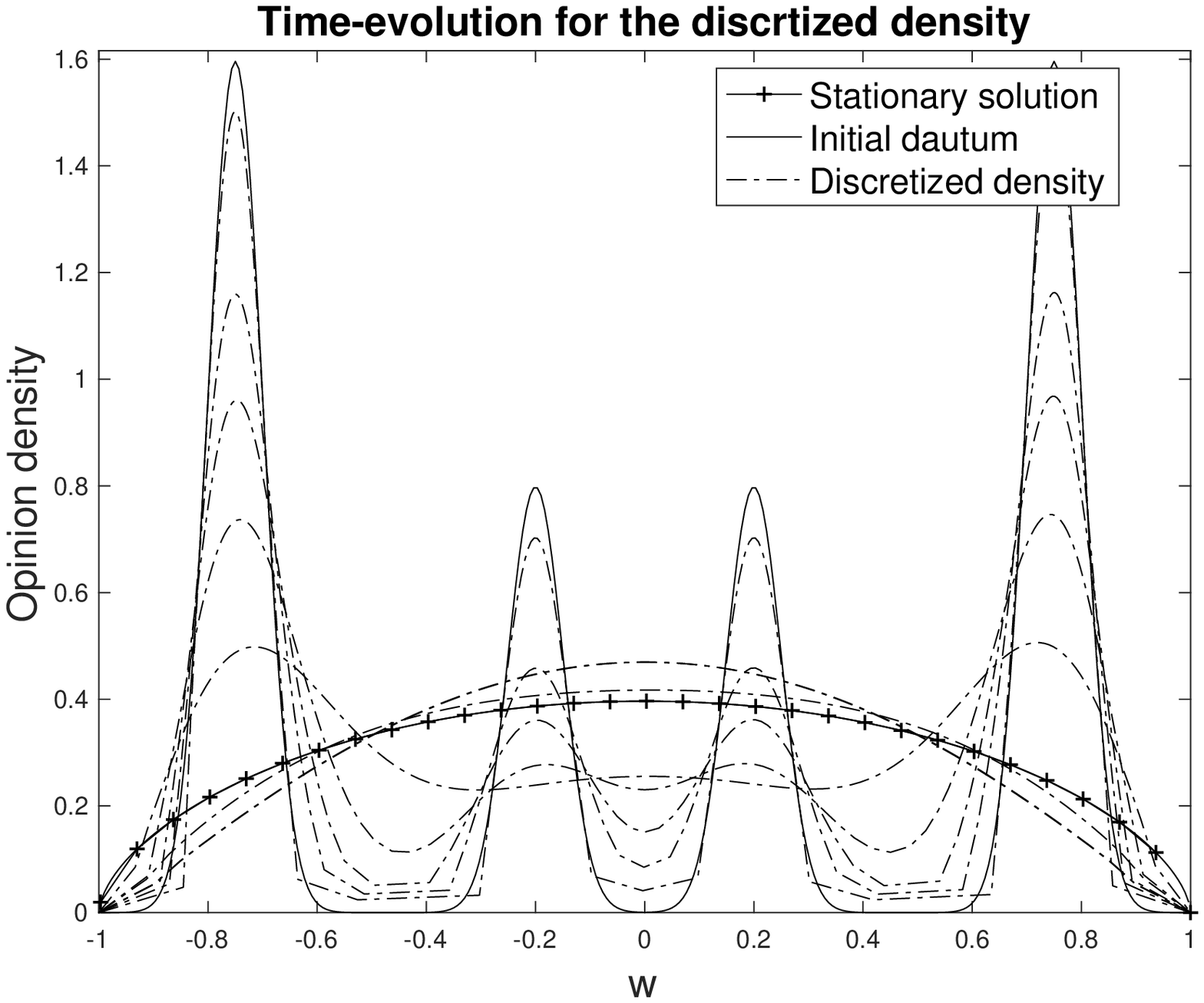}\par
    \includegraphics[width=6cm,height=5cm]{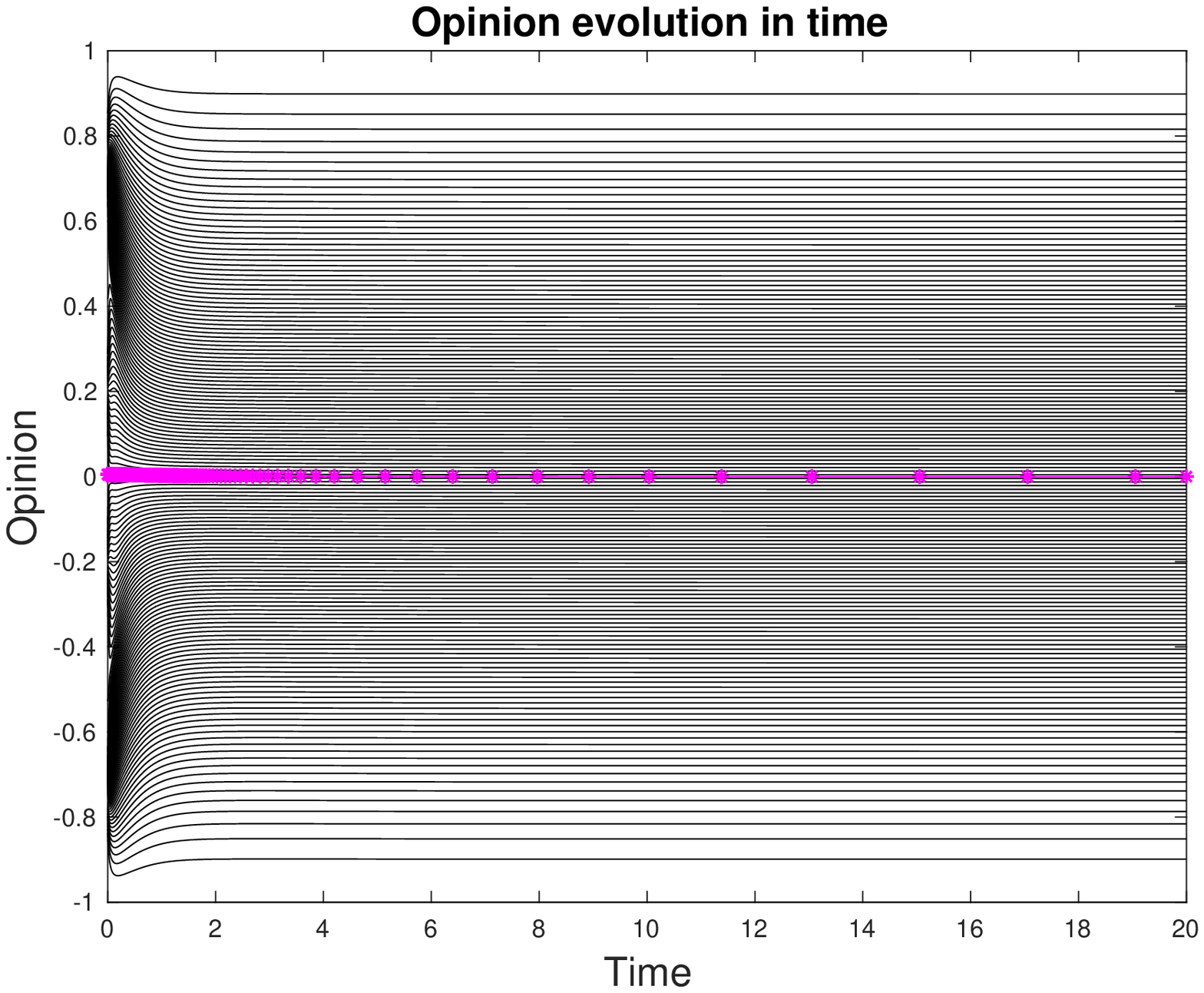}\par
\end{multicols}
\caption{Convergence for different initial data to the stationary state \eqref{eq:stat_a1}. The left column shows the evolution in  time for the reconstructed density, where the initial data are \eqref{eq:ini1}, \eqref{eq:ini2} and \eqref{eq:ini3} respectively, while the right column shows the opinions evolution in time. The (magenta) stars-line represent the mean opinion $m_1(t)$ in each case. Note that, in the second simulation, $m_1$ is not conserved in time but still converges to zero. }
\label{fig:lindiffa1}
\end{center}
\end{figure}
\begin{figure}[htbp]
\begin{center}
\begin{multicols}{2}
    \includegraphics[width=6cm,height=5cm]{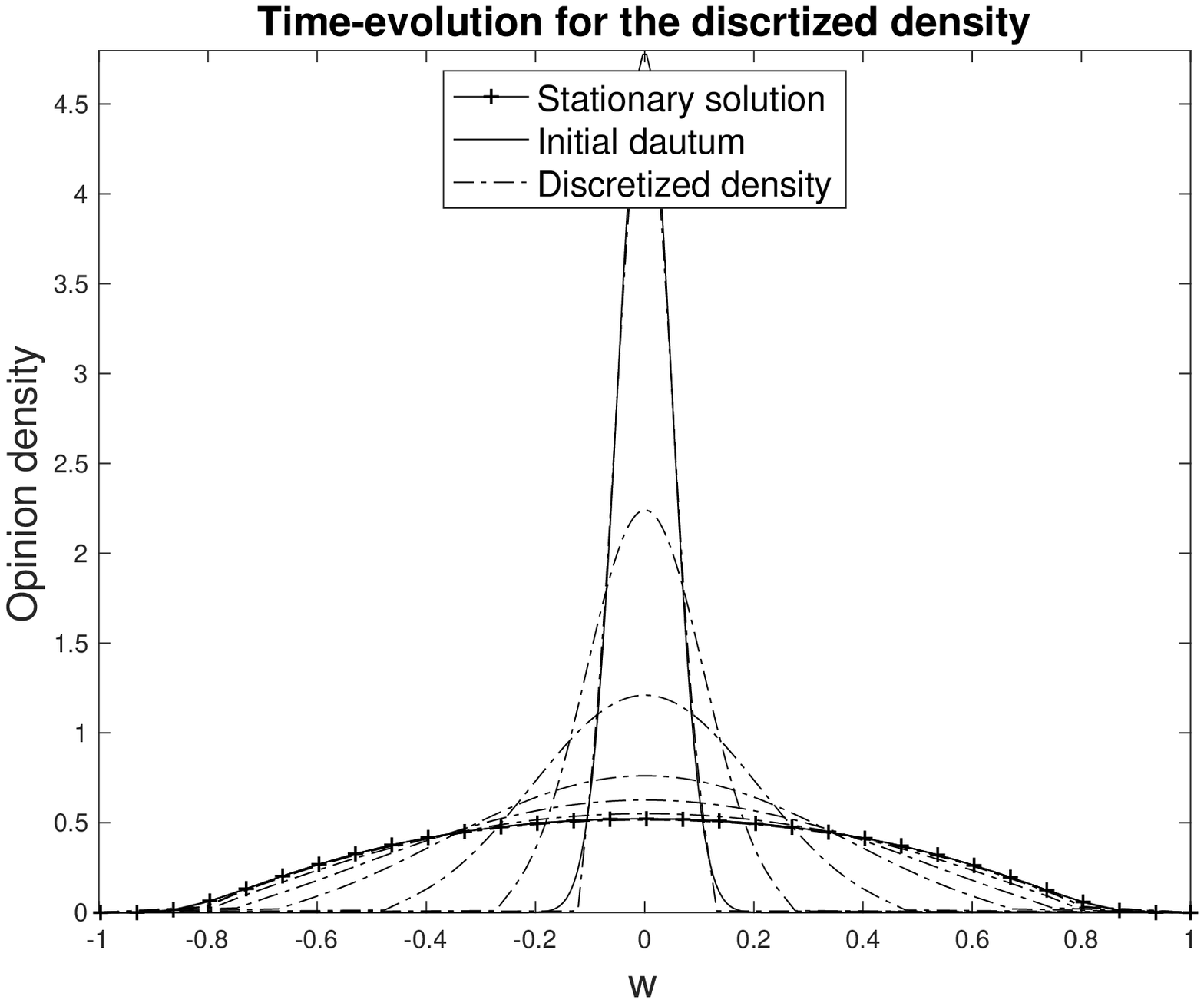}\par 
    \includegraphics[width=6cm,height=5cm]{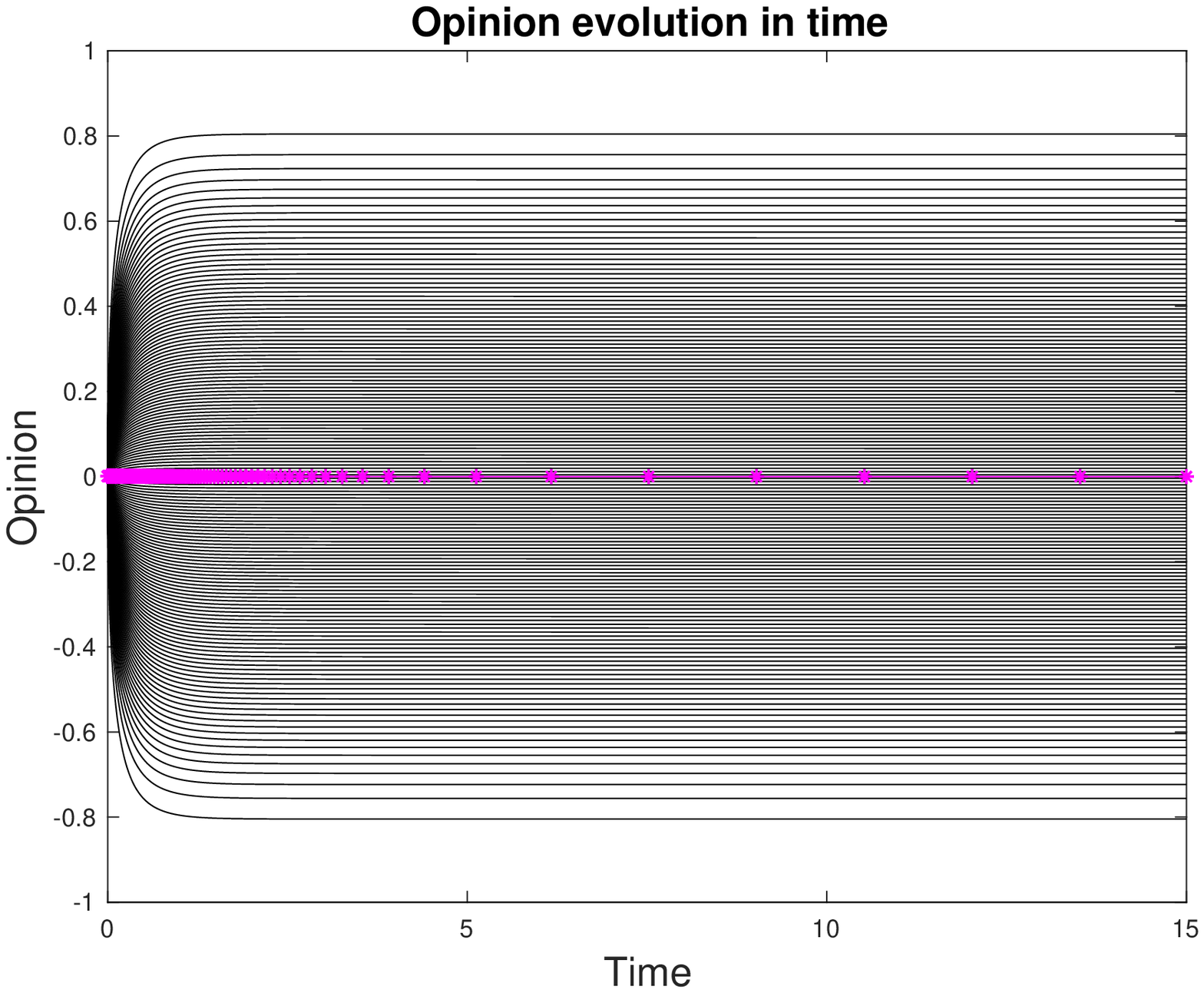}\par 
\end{multicols}
\begin{multicols}{2}
    \includegraphics[width=6cm,height=5cm]{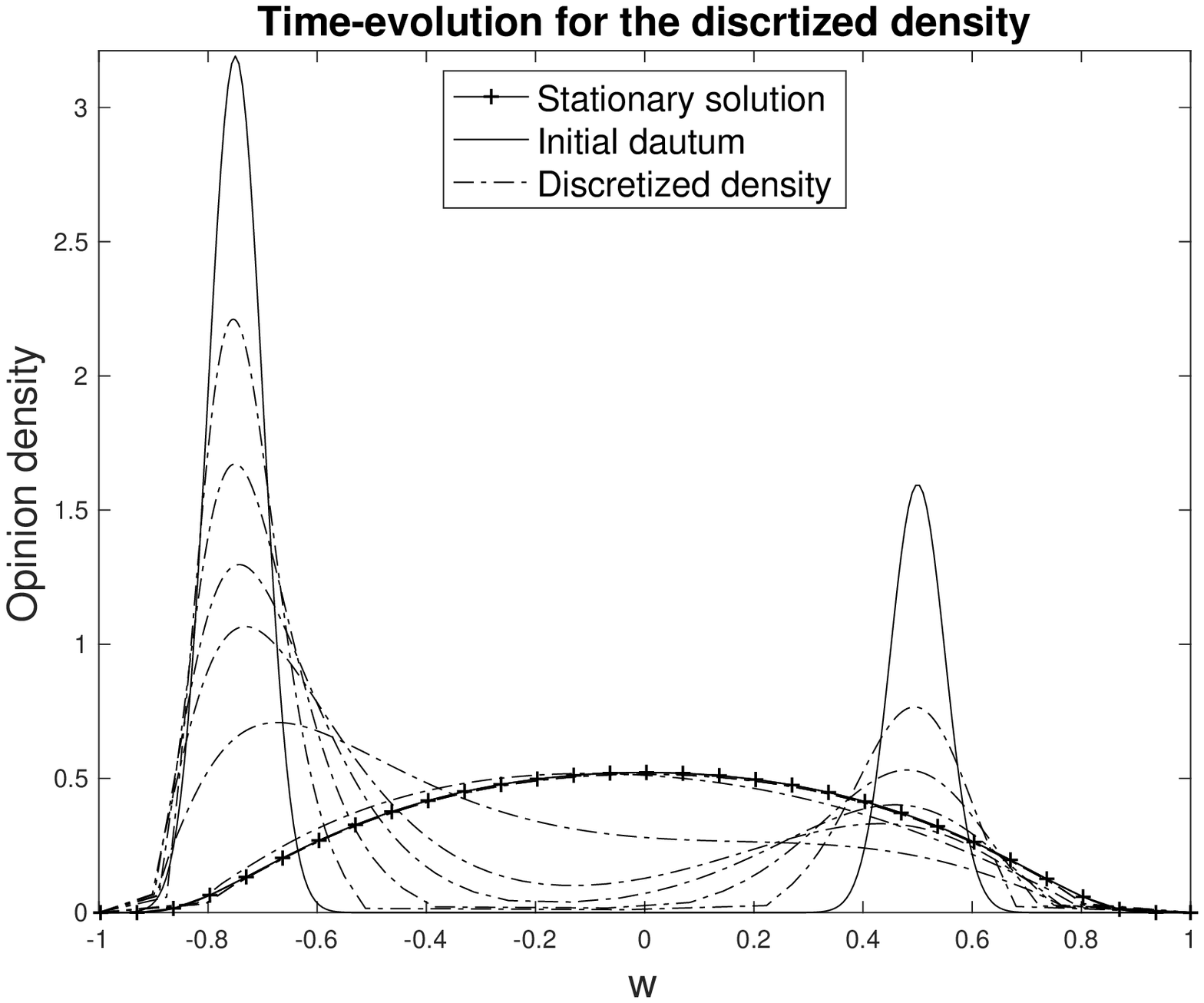}\par
    \includegraphics[width=6cm,height=5cm]{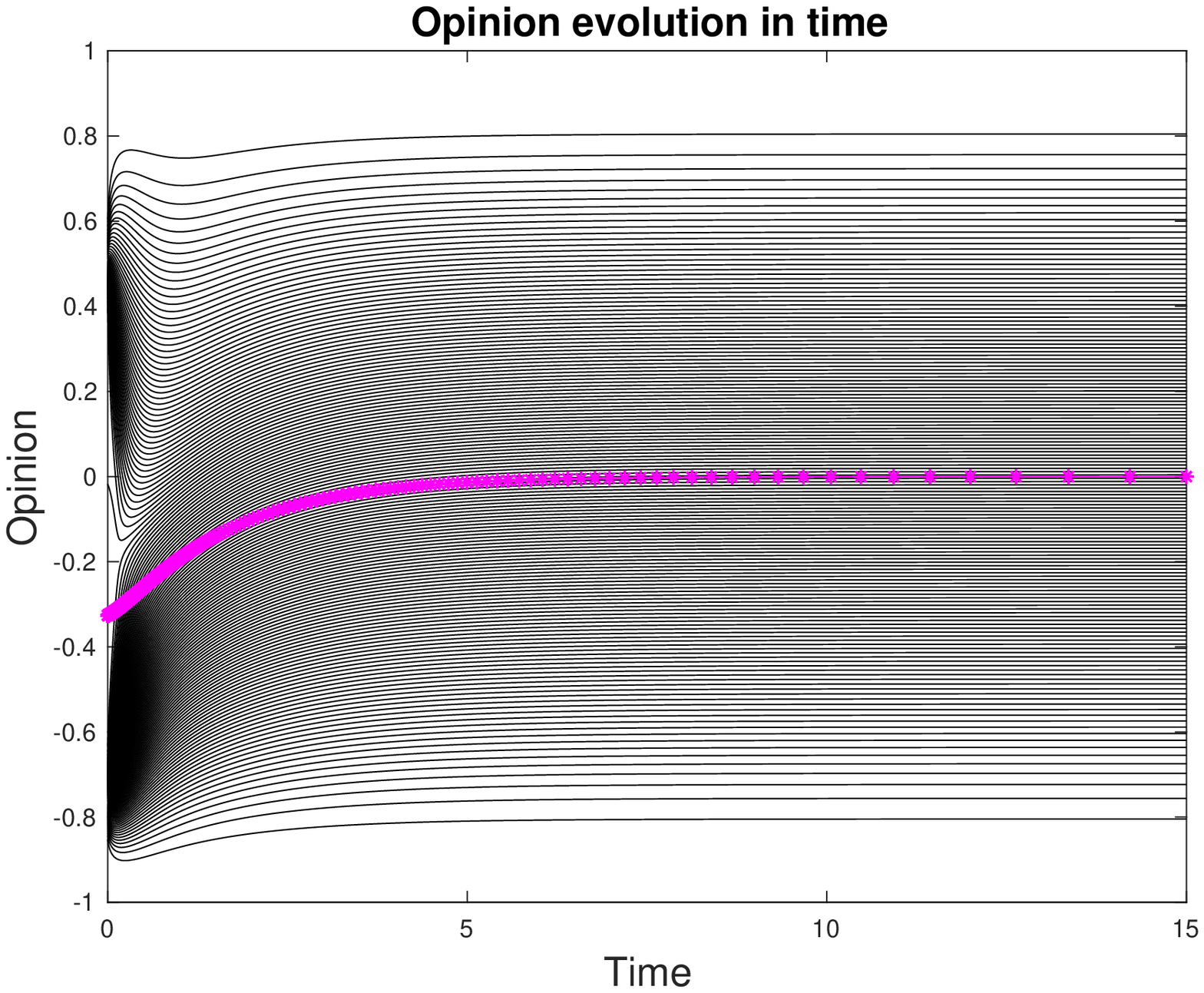}\par
\end{multicols}
\begin{multicols}{2}
    \includegraphics[width=6cm,height=5cm]{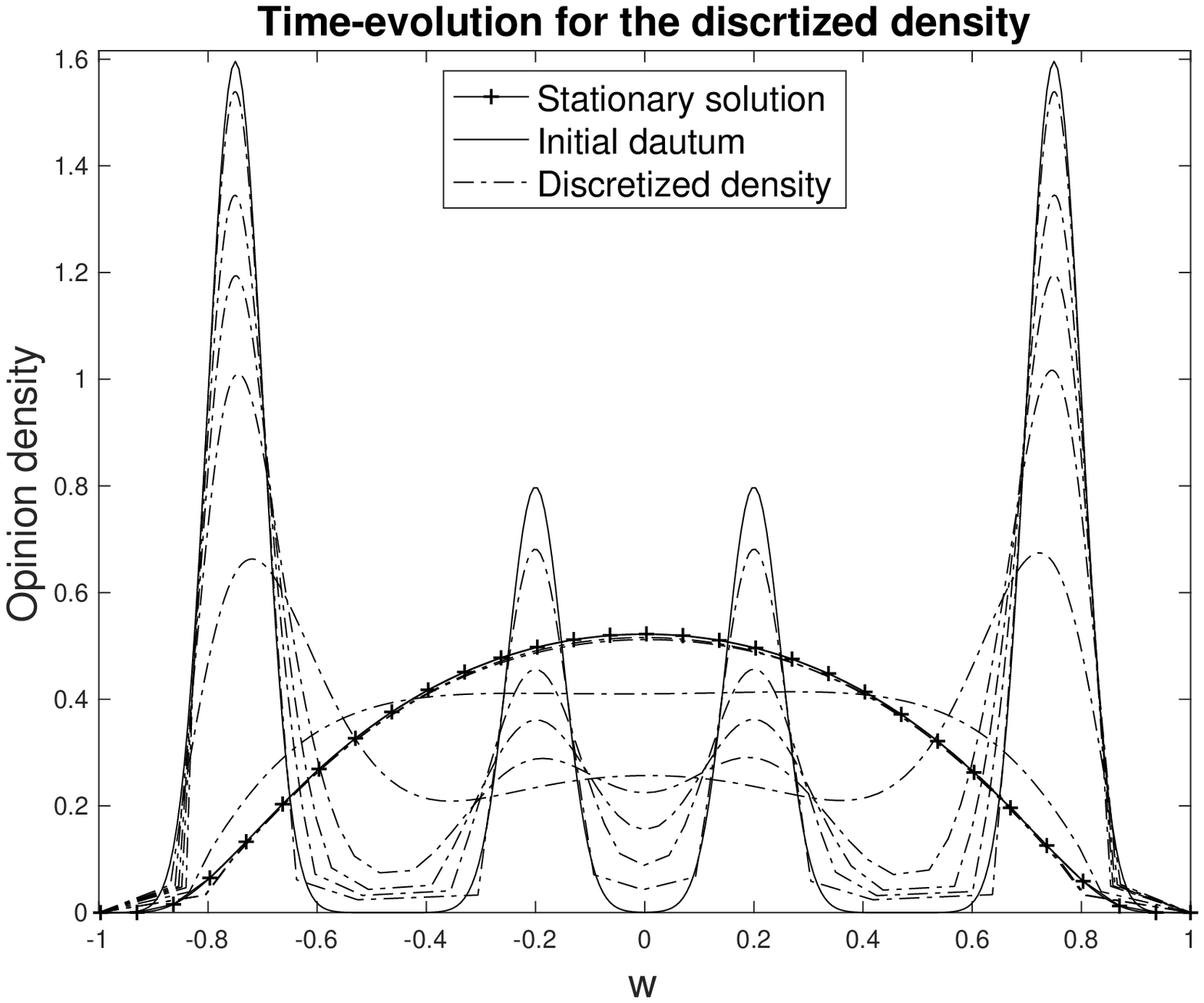}\par
    \includegraphics[width=6cm,height=5cm]{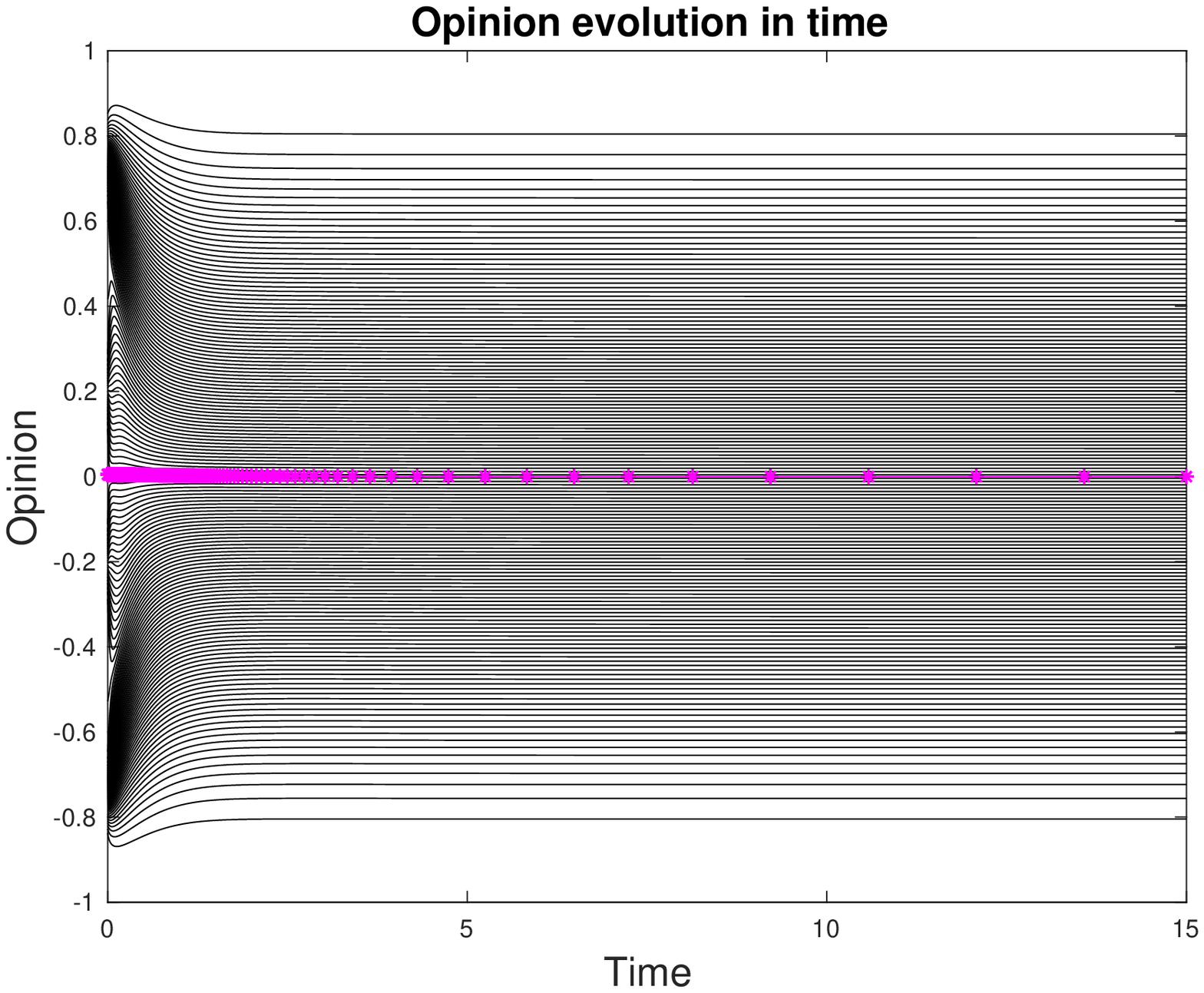}\par
\end{multicols}
\caption{Convergence for different initial data to the stationary state \eqref{eq:stat_a2}, where the initial data are \eqref{eq:ini1}, \eqref{eq:ini2} and \eqref{eq:ini3} respectively. Note that also in this case the mean opinion $m_1$ (star magenta line in the right column) converges to zero asymptotically.}
\label{fig:lindiff2}
\end{center}
\end{figure}
\begin{figure}[htbp]
\begin{center}
\begin{multicols}{2}
    \includegraphics[width=6cm,height=5cm]{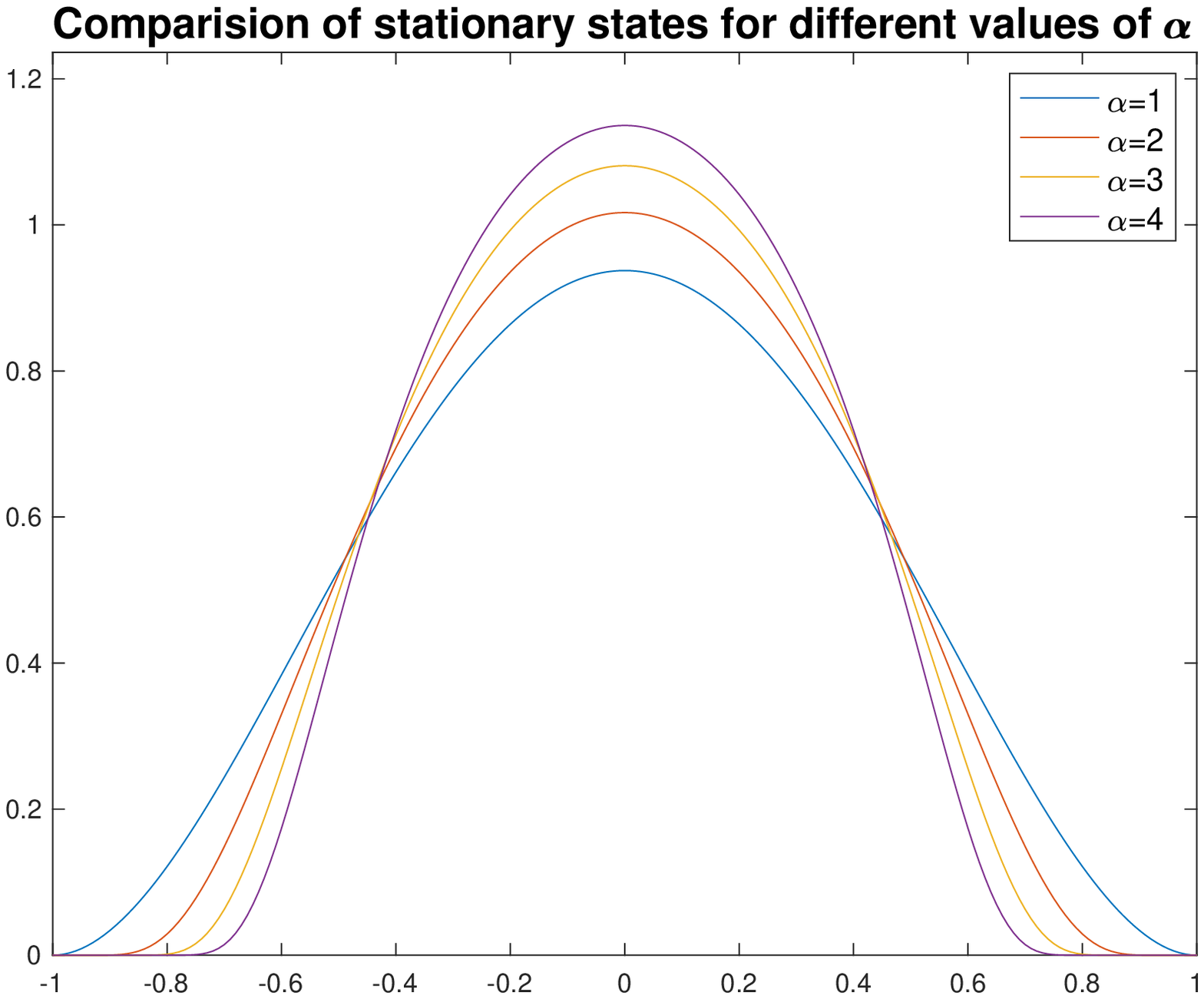}\par 
    \includegraphics[width=6cm,height=5cm]{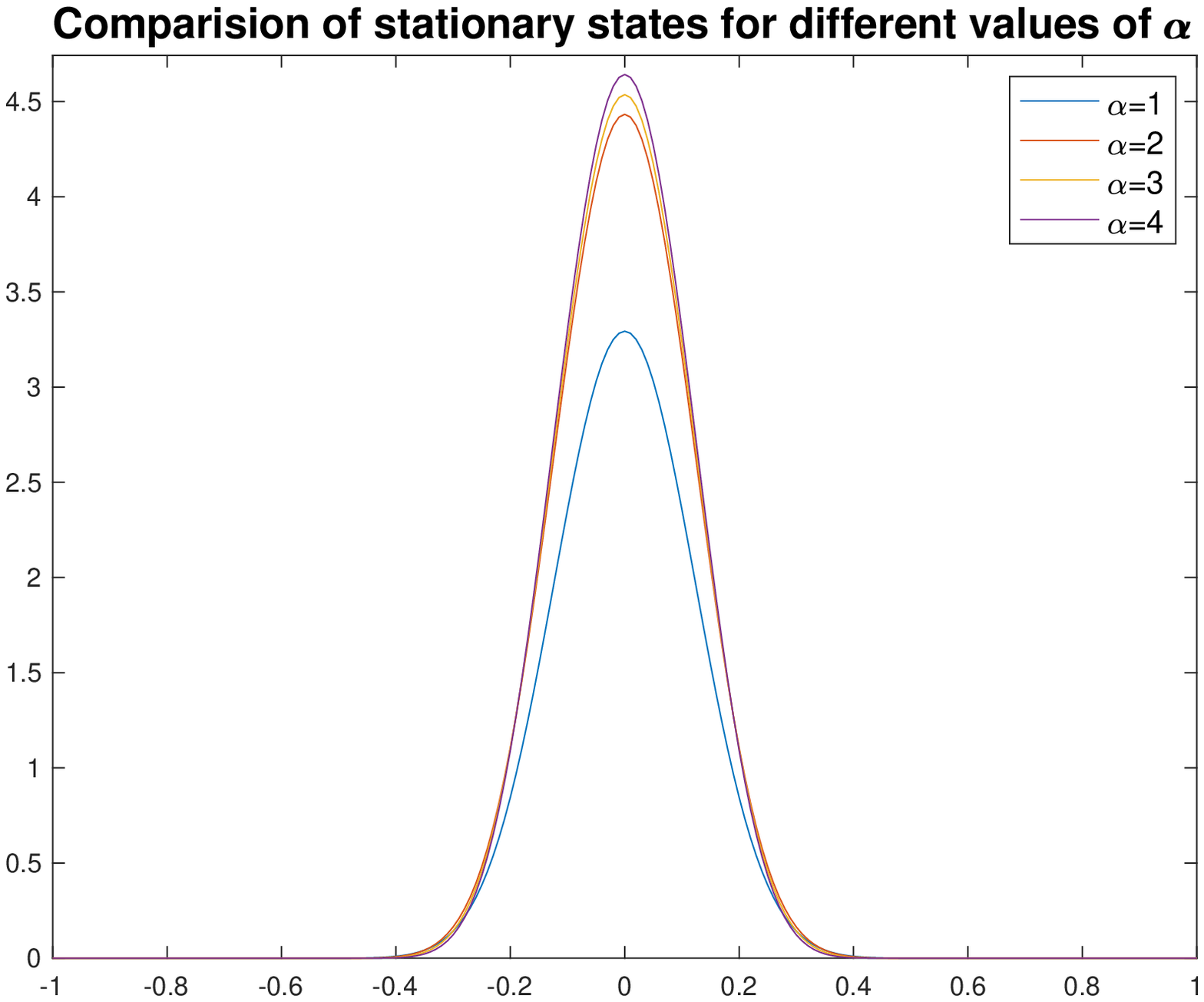}\par 
\end{multicols}
\caption{Comparison of different stationary states \eqref{eq:stat_1} for different values of $\alpha$. On the left the diffusion coefficient $\lambda^2=0.5$, on the right $\lambda^2=0.03$.}
\label{fig:diffalpha}
\end{center}
\end{figure}
We consider here the case of linear diffusion $\phi(u)=u$,
thus \eqref{eq:stationary} reduces to
\[
 \frac{\lambda^2}{2}D^2(w)\partial_w u(w) = (m_1^\infty-\sigma w) u(w).
\]
By standard separation of variable, we obtain
\begin{align*}
    & \frac{\lambda^2}{2}\log\left(u\right) = \mathfrak{D}_\alpha(w,m_1^\infty),
\end{align*}
which, in turn, gives
 \begin{equation}\label{eq:stat_1}
     u_\infty(w)=C_\infty\exp{\left(\frac{2}{\lambda^2}\mathfrak{D}_\alpha(w,m_1^\infty)\right)}, 
 \end{equation}
 where $C_\infty$ is the normalisation constant such that $\|u_\infty\|_{L^1}=\sigma$. In order to get a better understanding of the solutions, we consider below some explicit cases for the function $D$. 
\begin{itemize}
    \item Let us start by taking D as in \eqref{eq:mainD} for $\alpha = 1$
\[
D(w)=(1-w^2)^{\frac{1}{2}}.
\]
In this case, the integral in \eqref{eq:primD} becomes
\begin{align*}
    & \mathfrak{D}_\alpha(w,m_1^\infty) = \frac{m_1^\infty}{2} \log\left(\frac{1+w}{1-w}\right)+\frac{\sigma}{2}\log(1-w^2),
\end{align*}
and
 \begin{equation}\label{eq:stat_a1}
     u_\infty(w)=C_\infty(1+w)^{\frac{m_1^\infty+\sigma}{\lambda^2}}(1-w)^{\frac{\sigma-m_1^\infty}{\lambda^2}}. 
 \end{equation}
The well-posedness of the stationary state \eqref{eq:stat_a1} in $I$, is guaranteed  provided that 
$$
m_1^\infty+\sigma >0, \mbox{ and } \sigma-m_1^\infty>0,
$$
but, according to the consideration at the beginning of this section, we have that  
\begin{align*}
    & \frac{d}{dt}m_1(t) = -\int_{-1}^{1}\left(1-w^2\right) \partial_w u dw = -2 m_1(t).
\end{align*}
In particular, $m_1(t) \to m_1^\infty= 0$ as $t\to +\infty$, and \eqref{eq:stat_a1} reduces to 
 \begin{equation*}
     u_\infty(w)=C_\infty(1+w)^{\frac{\sigma}{\lambda^2}}(1-w)^{\frac{\sigma}{\lambda^2}}.
 \end{equation*}
In Figure \ref{fig:lindiffa1}, we provide some numerical evidence of the converge to the above stationary state from different initial data. We set the mass of opinion $\sigma=0.6$ and we choose 
\begin{enumerate}
    \item a single spike centred in the origin, miming a population  with opinions symmetrically distributed around the average opinion $w=0$
\begin{equation}\label{eq:ini1}
\bar{u}(w) =  \frac{\sigma}{\sqrt{2\pi(0.05)^2}}  e^\frac{-w^2}{2(0.05)^2},
\end{equation}
    \item two spikes with different weights $\sigma_1=0.4$ $\sigma_2=0.2$, non symmetric around the origin,
\begin{equation}\label{eq:ini2}
     \bar{u}(w) =  \frac{\sigma_1}{\sqrt{2\pi(0.05)^2}}  e^\frac{-(w+0.75)^2}{2(0.05)^2}+ \frac{\sigma_2}{\sqrt{2\pi(0.05)^2}}  e^\frac{-(w-0.5)^2}{2(0.05)^2},
     \end{equation}
         \item a combination of four spikes  symmetrically distributed around the origin, with weight $\sigma_1=0.2$ $\sigma_2=0.1$
    \begin{equation}\label{eq:ini3}
    \begin{split}
      \bar{u}(w) &=  \frac{\sigma_1}{\sqrt{2\pi(0.05)^2}}  e^\frac{-(w+0.75)^2}{2(0.05)^2}+ \frac{\sigma_2}{\sqrt{2\pi(0.05)^2}}  e^\frac{-(w+0.2)^2}{2(0.05)^2}\\
      &+ \frac{\sigma_2}{\sqrt{2\pi(0.05)^2}}  e^\frac{-(w-0.2)^2}{2(0.05)^2}+ \frac{\sigma_1}{\sqrt{2\pi(0.05)^2}}  e^\frac{-(w-0.75)^2}{2(0.05)^2}.
        \end{split}
\end{equation}
\end{enumerate}
All the simulations are performed using the deterministic particle approximation introduced in the previous sections as a numerical scheme for \eqref{main}. More precisely, given one of the above initial data, we construct an initial distribution of opinions according to \eqref{eq:ini_part}, then we solve the ODEs system \eqref{eq:particles_gen} and we reconstruct the density as in \eqref{rhoN}. 

    \item We now consider the mobility function as in \eqref{eq:mainD} for $\alpha=2$
\[
D(w)=(1-w^2).
\]
In this case, the stationary solution reduces to
 \begin{equation}\label{eq:stat_a2}
     u_\infty(w)=C_\infty(1+w)^{\frac{m_1^\infty}{2\lambda^2}}(1-w)^{\frac{-m_1^\infty}{2\lambda^2}}e^{\frac{m_1^\infty w-\sigma}{\lambda^2(1-w^2)}}, 
 \end{equation}
 where $C_\infty$ is the usual normalisation constant. The well-posedness of the steady state \eqref{eq:stat_a2} is not guaranteed \emph{a priori} and it seem to be strongly dependent on the relation between $m_1^\infty$ and $\sigma$. As mentioned above, the evolution equation for the second moment is not closed, here it is corresponds to
 \begin{equation*}
 \frac{d}{dt} m_1(t) = -2\lambda^2 m_1(t) + 2\lambda^2 m_3(t),
 \end{equation*}
where $m_3$ denotes the third order moment of $u$. Let us observe, that for any $k\geq 2$, the $k-$th moment evolves according to
 \begin{align*}
 \frac{d}{dt} m_k(t) =&\frac{\lambda^2}{2}(k-1) m_{k-2}(t)+m_1(t)m_{k-1}(t)\\
 &-k\left[\lambda^2(k+1)+\sigma\right]m_k(t)+\frac{\lambda^2}{2}(k+3) m_{k+2}(t),
 \end{align*}
that is a dynamical system with infinite dimension, whose study requires deeper investigations which exceed the scope of this paper. However, by introducing the variance function 
\[
 \mbox{Var}\left[u\right](t)=m_2(t) - (m_1(t))^2,
\]
it is easy to see that $\left(m_1(t)\right)^2\leq 1-\mbox{Var}\left[u\right](t)$, with $\mbox{Var}\left[u\right](t)\in \left[0,1\right]$. 
Since also $|m_3|\leq 1$, this suggests the decay in time of the mean opinion. This is supported by the numerical results in Figure \ref{fig:lindiff2}. 

\item The argument for generic values of $\alpha$ is more involved, but does not present further complications. For this reason, we decided to omit the proof in the present paper.
Nonetheless, we highlight that the parameter $\alpha$ plays an important modelling role: as shown in Figure \ref{fig:diffalpha} (left), the support of the stationary state shrinks when alpha increases, thus providing an higher \emph{consensus} around the limiting mean opinion (in the previous examples the mean opinion is $w=0$). A similar effect can be reached while decreasing the diffusion coefficient $\lambda$. Indeed, for smaller values of $\lambda$, the contribution of the reaction part is stronger and the stationary states are more concentrated, see Figure \ref{fig:diffalpha} (right).
\end{itemize}

\subsubsection{Nonlinear diffusion} 
\begin{figure}[htbp]
\begin{center}
\begin{multicols}{2}
    \includegraphics[width=6cm,height=5cm]{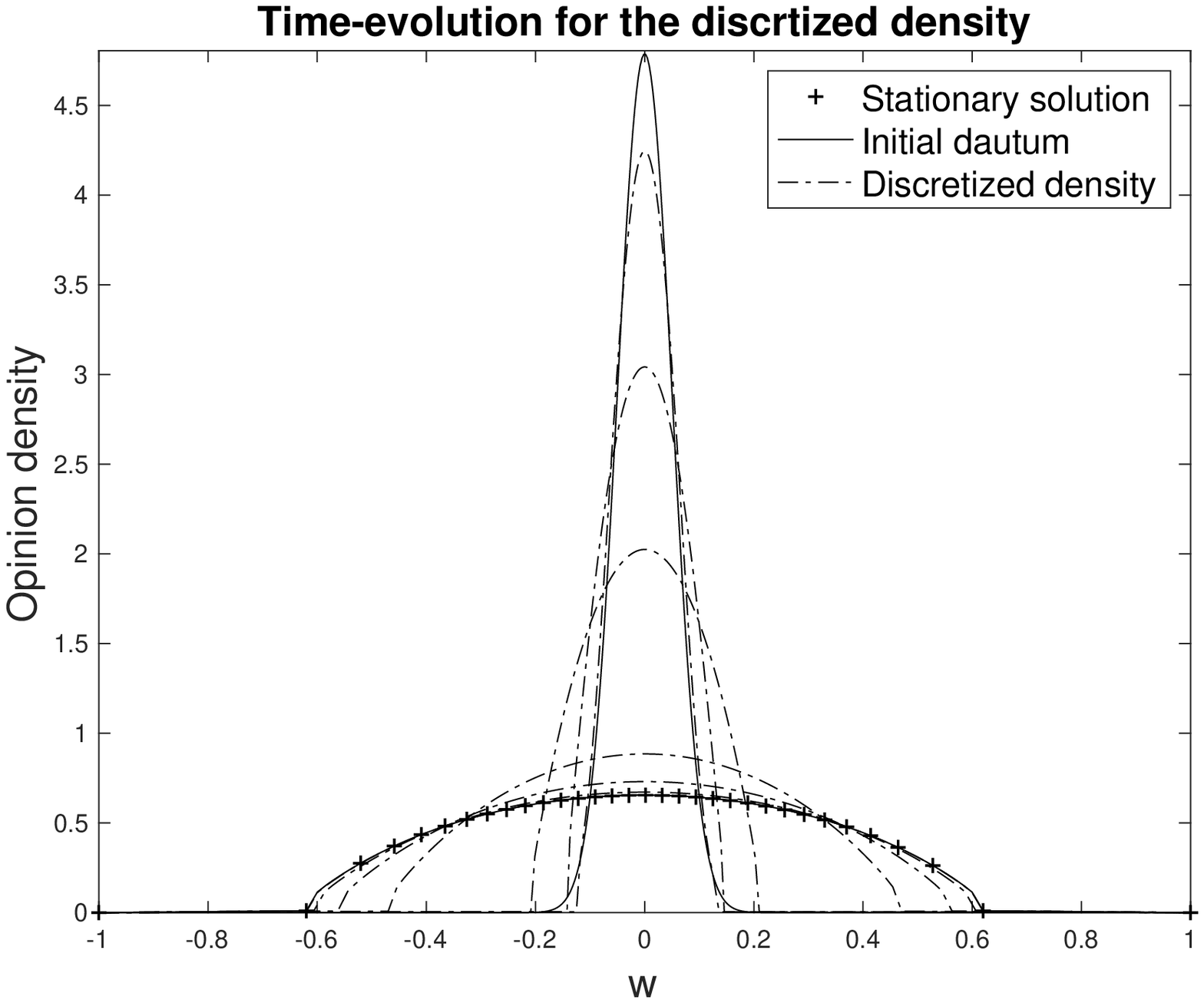}\par 
    \includegraphics[width=6cm,height=5cm]{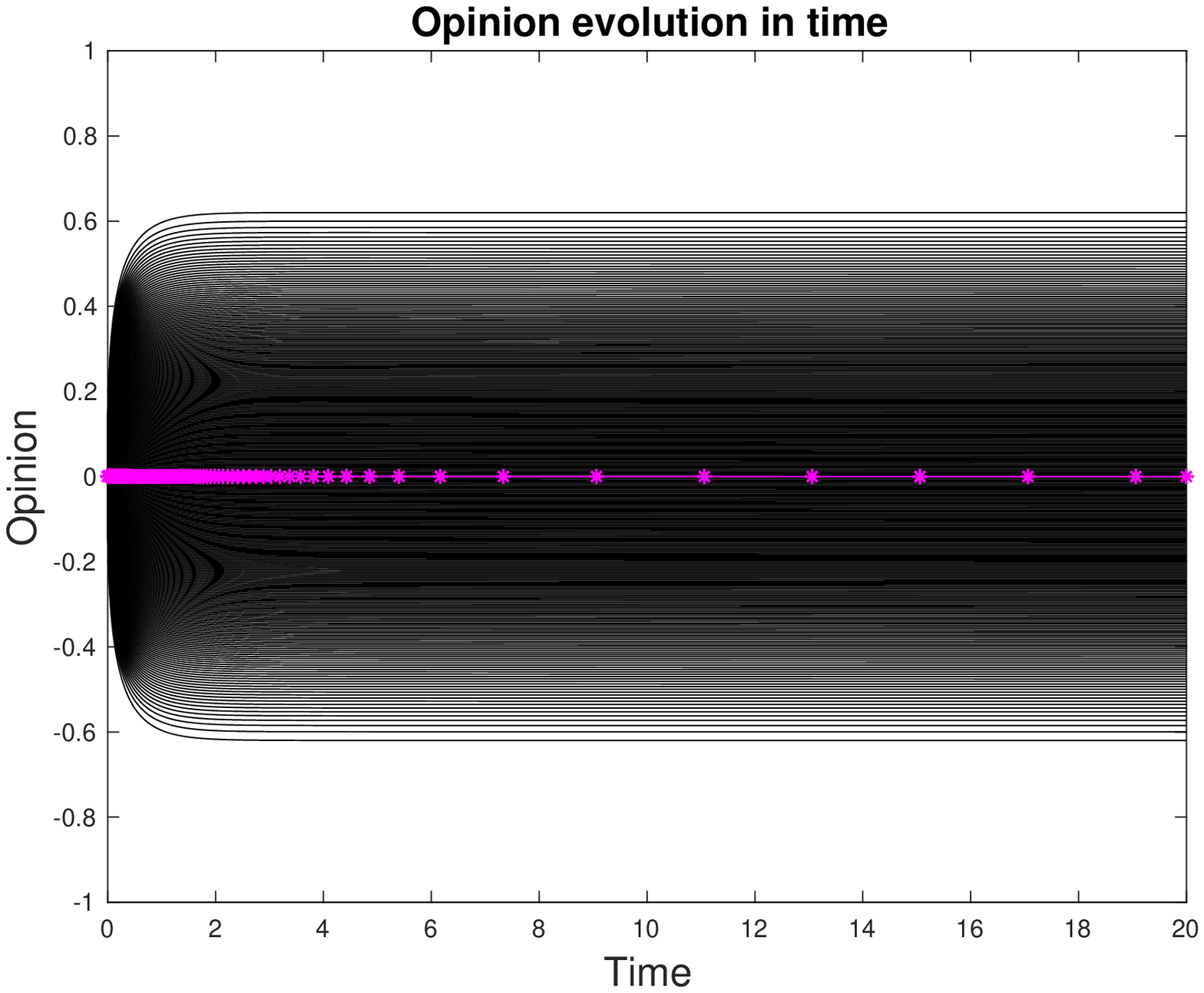}\par 
\end{multicols}
\begin{multicols}{2}
    \includegraphics[width=6cm,height=5cm]{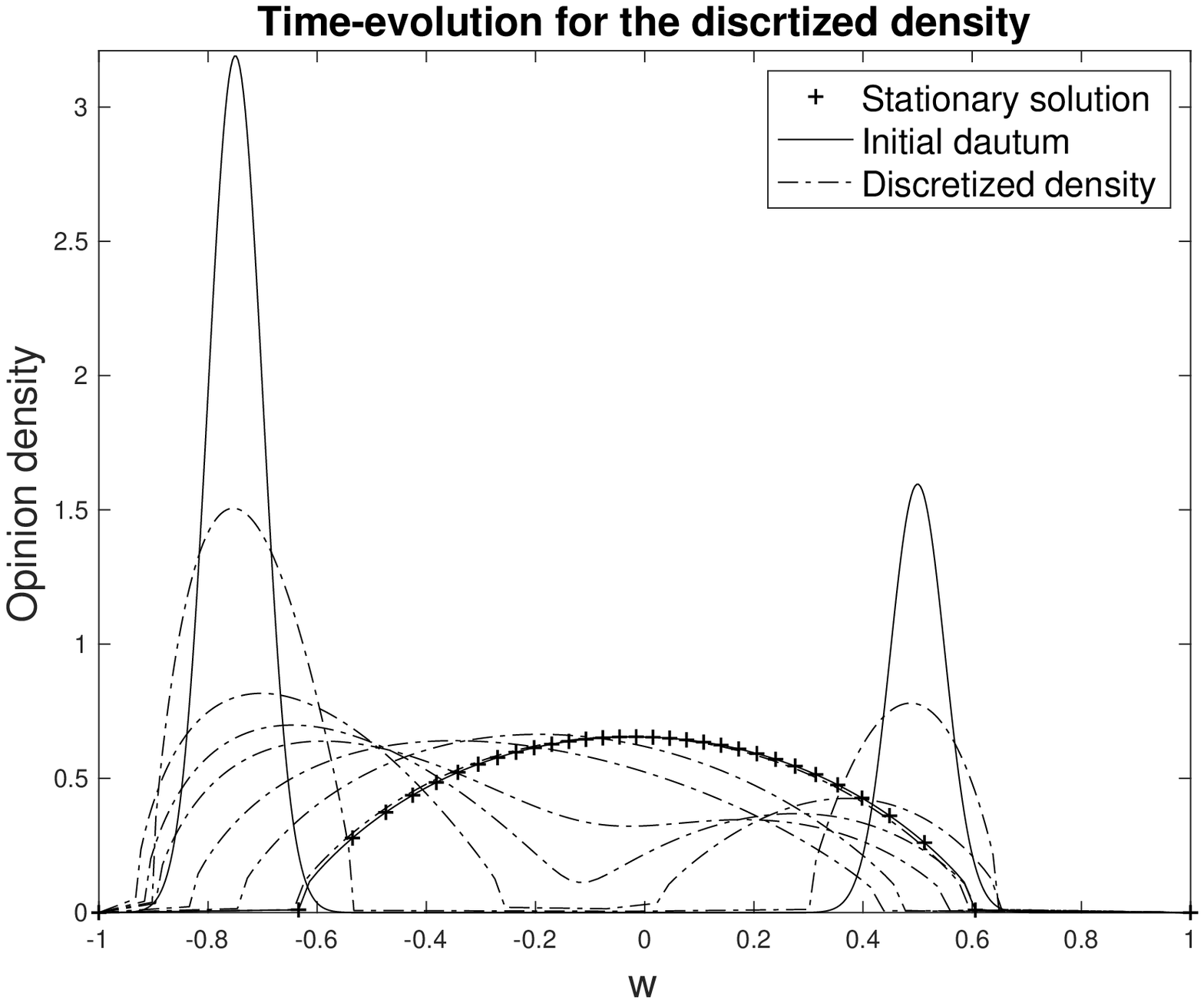}\par
    \includegraphics[width=6.5cm,height=5cm]{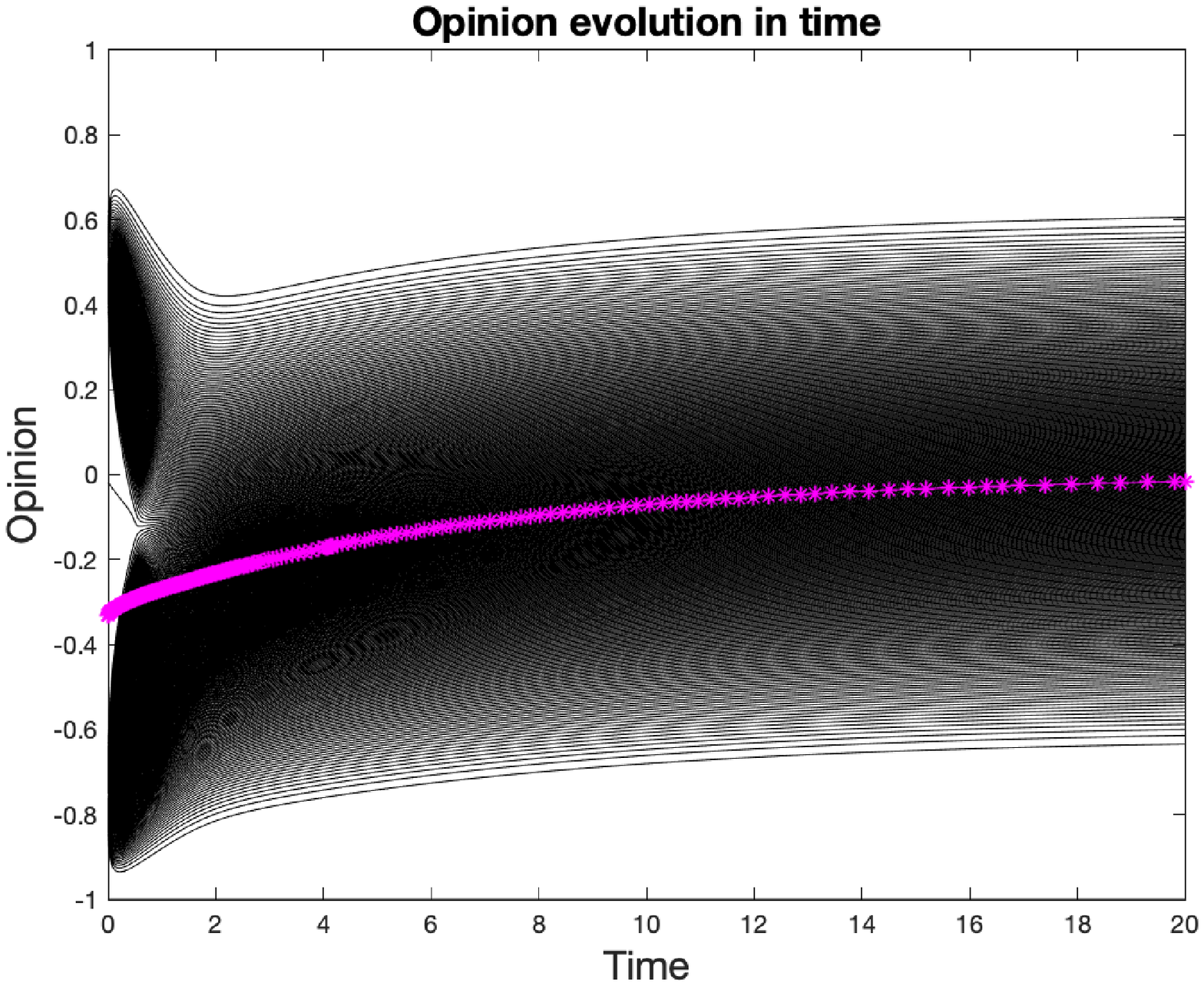}\par
\end{multicols}
\begin{multicols}{2}
    \includegraphics[width=6cm,height=5cm]{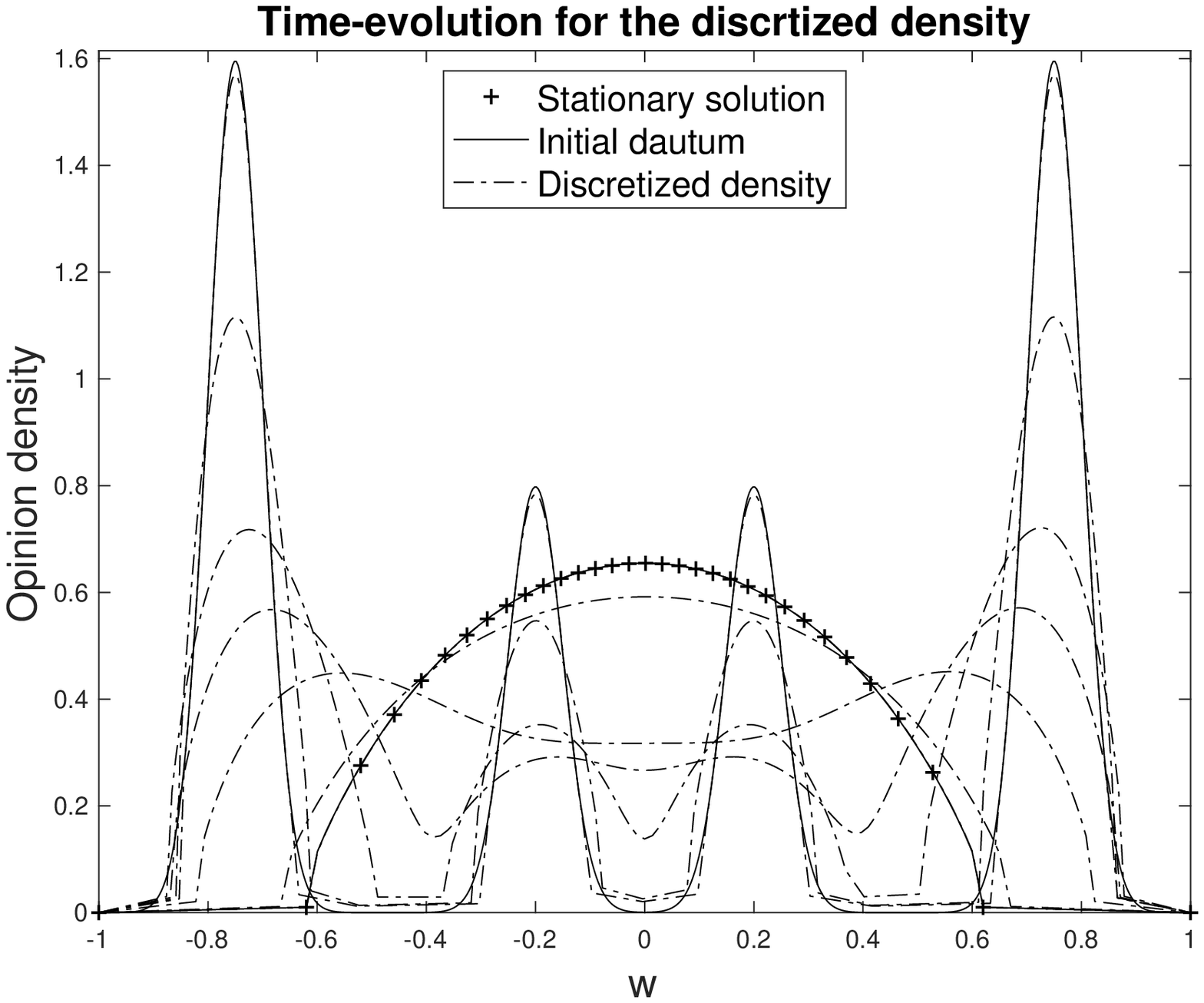}\par
    \includegraphics[width=6cm,height=5cm]{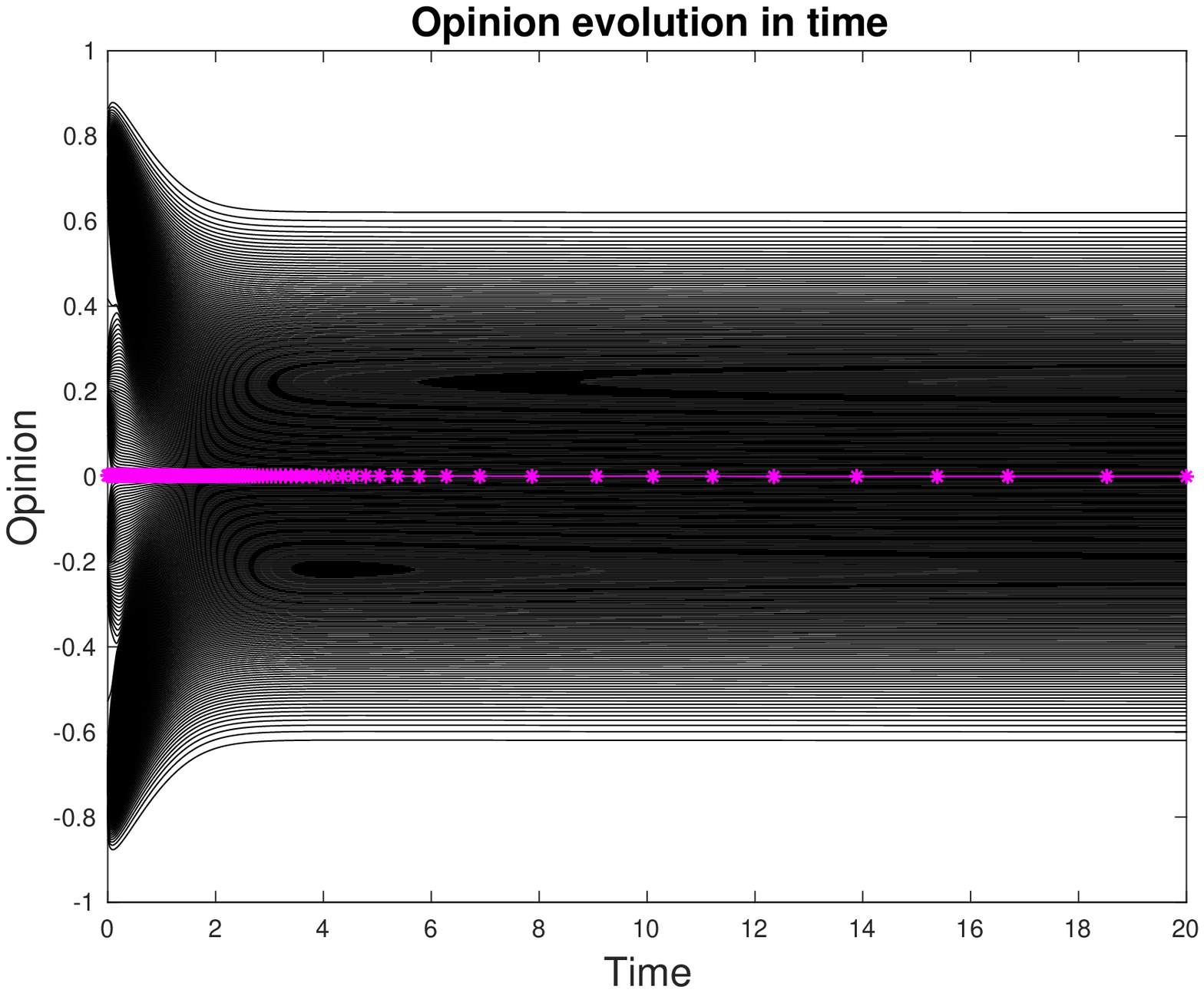}\par
\end{multicols}
\caption{Convergence for the initial data \eqref{eq:ini1}, \eqref{eq:ini2} and \eqref{eq:ini3} to the stationary state \eqref{eq:ststnon2}, with $\alpha=1$. Note that the nonlinearity in the diffusion produces stationary solutions with supports that are smaller than the one we have seen in the linear diffusion case. }
\label{fig:nlindiffa1}
\end{center}
\end{figure}
\begin{figure}[htbp]
\begin{center} 
    \includegraphics[scale=0.6]{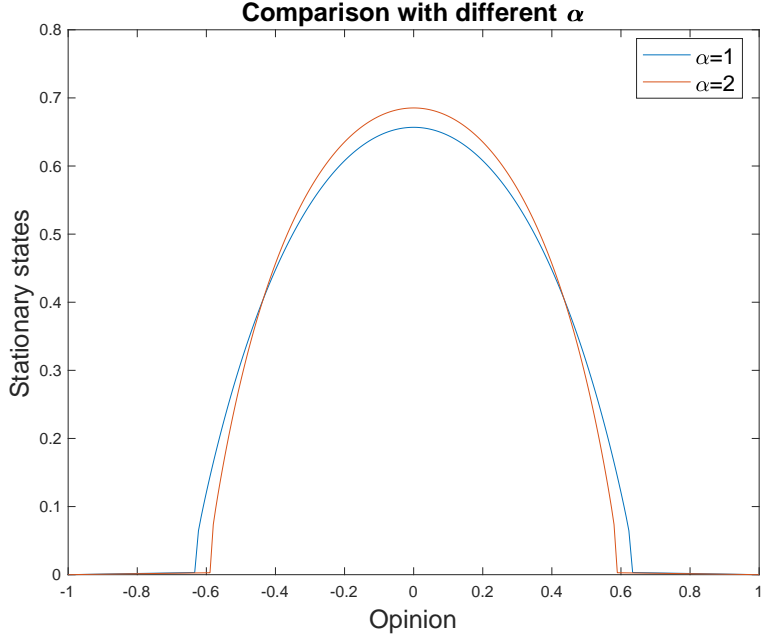}\par 
\caption{Stationary states in \eqref{eq:ststnon2} for different values of $\alpha$.}
\label{fig:ndiffalpha}
\end{center}
\end{figure}
Here we discuss the case of  nonlinear diffusion of a \emph{porous medium} type, namely
\[
  \phi (u)=\frac{u^\gamma}{\gamma}, \quad \gamma>1.
\] 
The stationary states, then, correspond to the solutions of
\begin{equation}\label{eq:ststnon1}
  \partial_w\left(\frac{\lambda^2}{2}D^2(w)u\partial_w \frac{u(w)^{\gamma-1}}{\gamma-1} -\left(m_1^\infty-\sigma w \right)u(w) \right) = 0,
\end{equation}
which can be rewritten as
\begin{equation*}
  \partial_w\left(\frac{\lambda^2}{2} \frac{u(w)^{\gamma-1}}{\gamma-1} -\mathfrak{D}_\alpha(w,m_1^\infty)  \right)= 0.
\end{equation*}
Since the physical solutions are non negative, we deduce
\begin{equation}\label{eq:ststnon2}
  u_\infty(w)=\left[\frac{2(\gamma-1)}{\lambda^2}\left(C+\mathfrak{D}_\alpha(w,m_1^\infty)\right)_{+}\right]^{\frac{1}{\gamma-1}},
\end{equation}
where $C$ is a suitable normalization constant. 
In Figure \ref{fig:nlindiffa1} we show the convergence towards the stationary state in the case $D(w)=(1-w^2)^{(1/2)}$ and $\gamma =2$, with the same initial data and diffusion coefficient of the previous examples. Observe, that the nonlinear diffusion induces a stronger \emph{consensus} around the compromise value $w=0$. In Figure \ref{fig:ndiffalpha} we plot the stationary states corresponding to $\alpha=1$ and $\alpha=2$.

\subsection{Simulations with many species} We now explore the large-time behaviour for the many species case. The analytical study in this case became more complicated, since it requires to check the solvability of a system of coupled equations in the form \eqref{eq:stationary}. We here focus our attention on the numerical comparison among the large time behaviours of systems \eqref{eq:sys_FL} and \eqref{eq:sys_FLT}.

\begin{figure}[htbp]
\begin{center}
\begin{multicols}{2}
    \includegraphics[width=6cm,height=5cm]{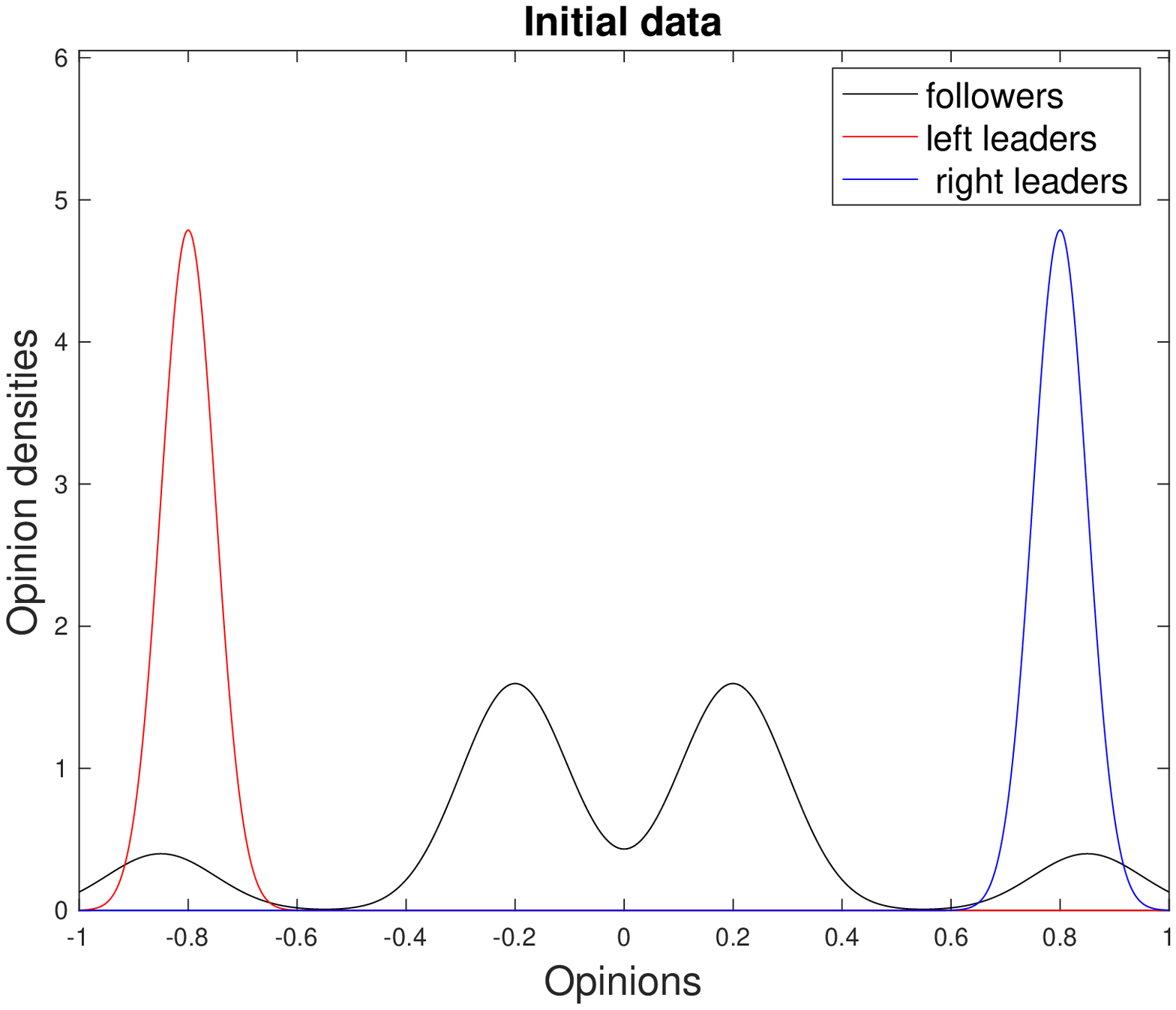}\par 
    \includegraphics[width=6cm,height=5cm]{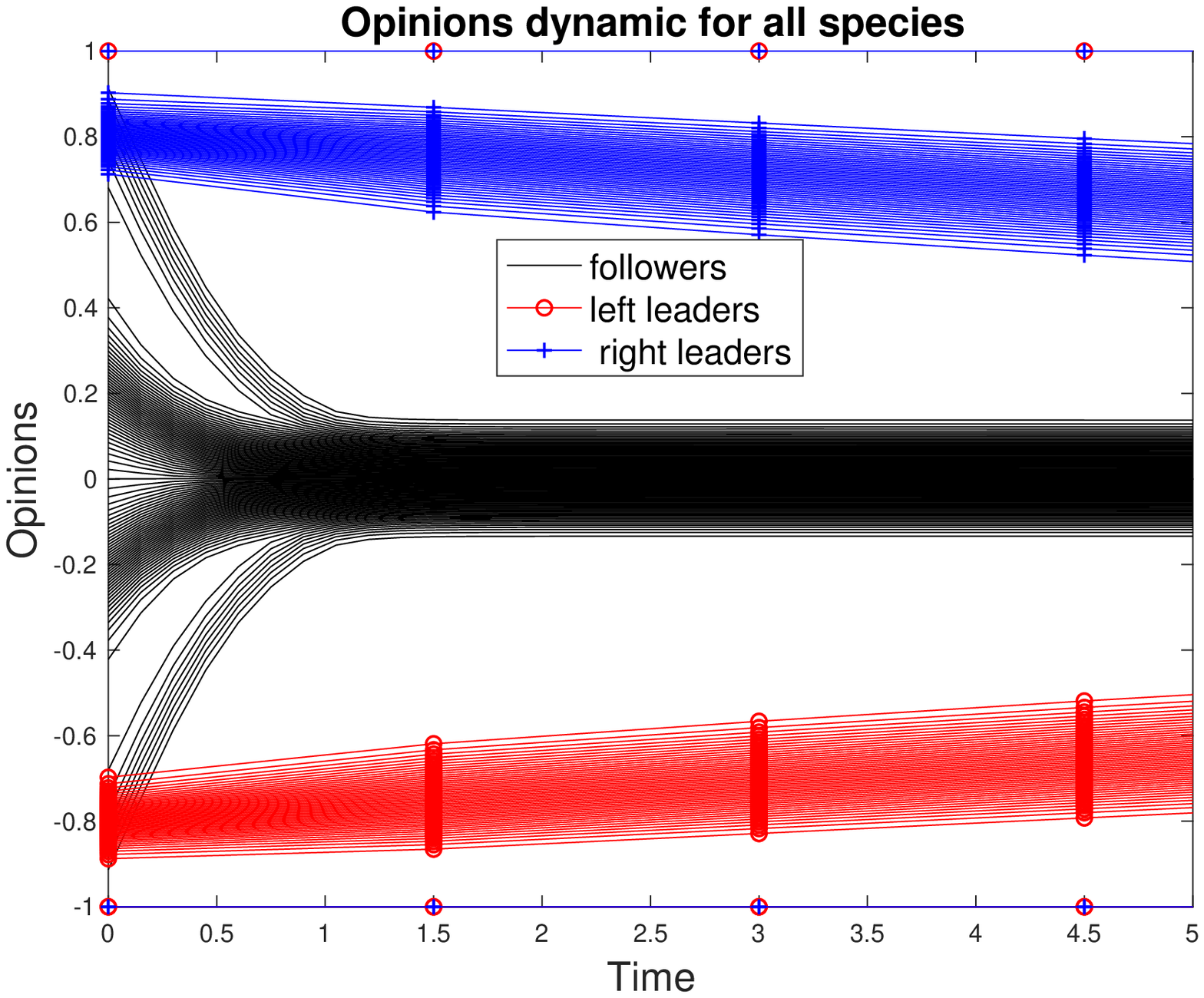}\par 
\end{multicols}
\begin{multicols}{2}
    \includegraphics[width=6cm,height=5cm]{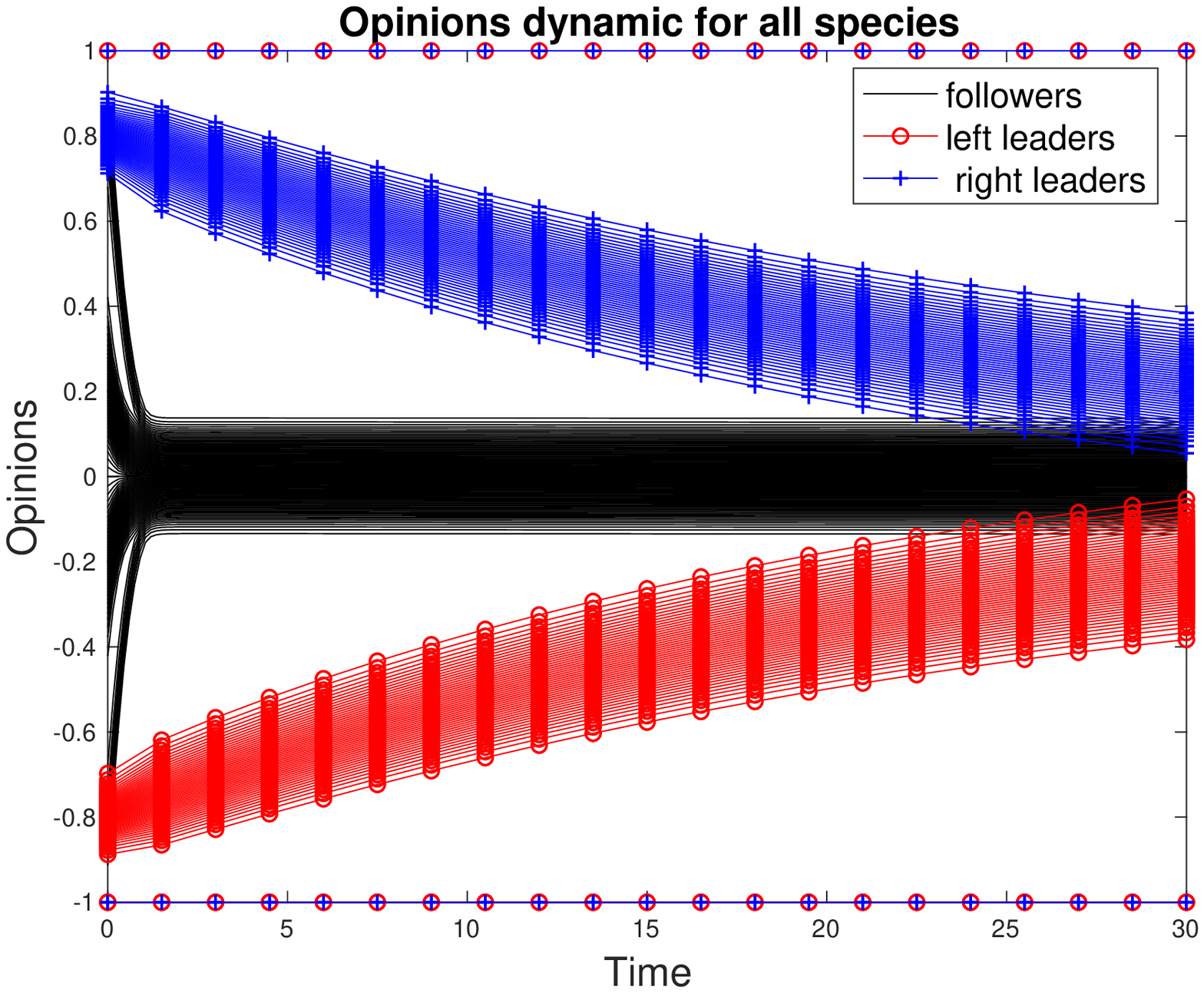}\par
    \includegraphics[width=6cm,height=5cm]{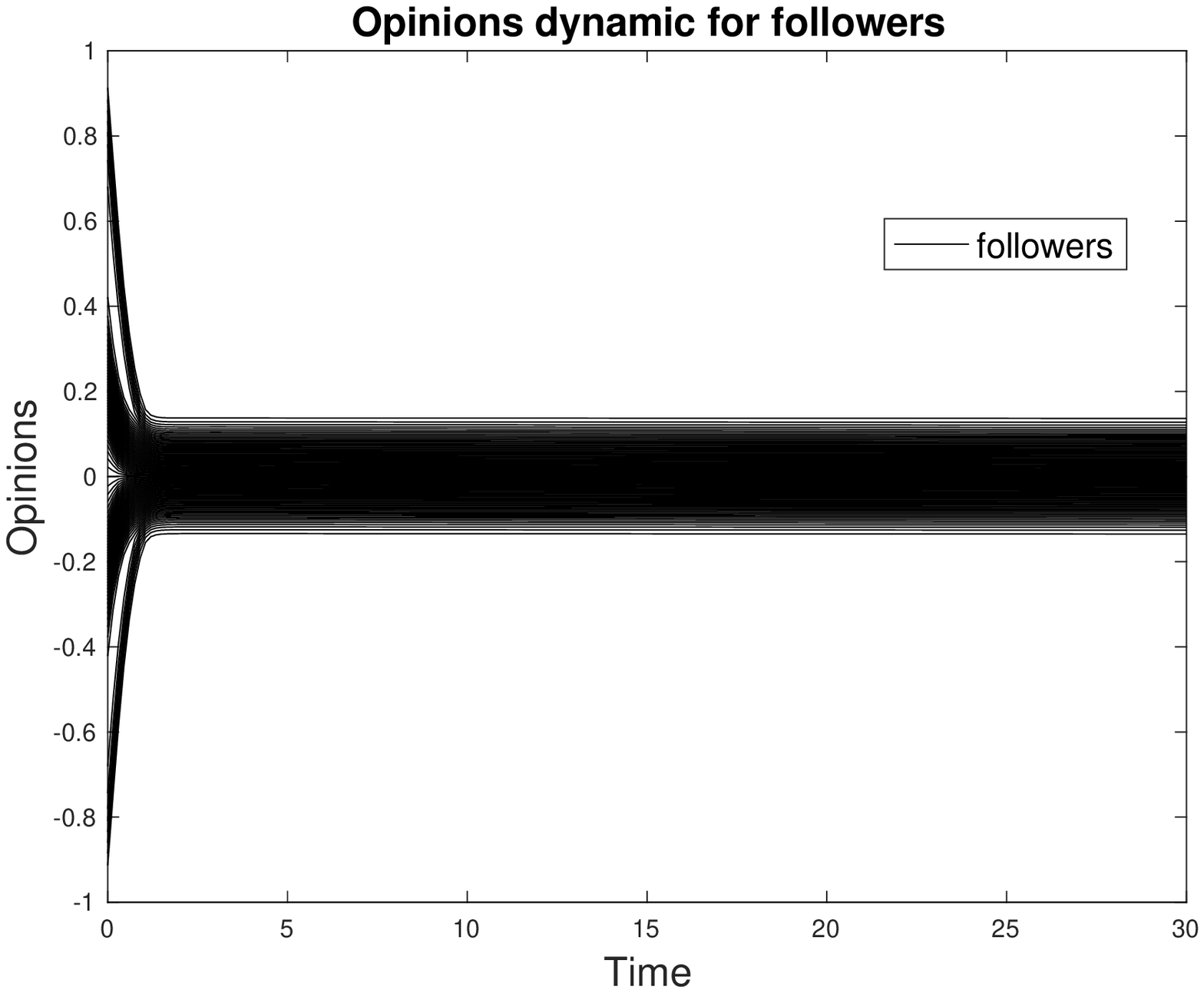}\par
\end{multicols}
\caption{Opinion dynamics in presence of \emph{equally-strong} leaders. Top-left initial data for followers(black), left leaders (red) and right leaders (blue). Top-right and bottom-left opinions evolutions for all the groups in short and long term-respectively. Bottom-right opinions evolutions for the followers.}
\label{fig:FLL1}
\end{center}
\end{figure}

In Figure \ref{fig:FLL1}, we show the opinions dynamics in the Follower-Leader system \eqref{eq:sys_FL}, that we rewrite below for the reader convenience,
\begin{equation*}
\begin{split}
& \partial_t f = \partial_w \Big(\frac{\lambda_f^2}{2}D_f^2 (w) \partial_w \phi_f(f) - f \big( \mathcal{P}_{ff}[f] + \mathcal{P}_{fl}[l] + \mathcal{P}_{fr}[r] \big)\Big)\\
&  \partial_t l = \partial_w \Big(\frac{\lambda_l^2}{2}D_l^2 (w) \partial_w \phi_l(l) - l \big(\mathcal{P}_{ll}[l] + \mathcal{P}_{lr}[r]\big)\Big) ,\\
&  \partial_t r = \partial_w \Big(\frac{\lambda_r^2}{2}D_r^2 (w) \partial_w \phi_r(r) - r \big(\mathcal{P}_{rr}[r] + \mathcal{P}_{rl}[l]\big)\Big).
\end{split}
\end{equation*}
 In particular, we assume that there are two groups of Leaders with equal mass and symmetric centres of mass. Moreover, we assume that the initial followers' opinions are symmetrically distributed, see Figure \ref{fig:FLL1} top-left. In this simulation, we fix the diffusion coefficients $\frac{\lambda_u^2}{2}=0.03$ and the mobility functions $D_u^2 (w)=(1-w^2)$ for $u\in\left\{f,l,r\right\}$, taking $\sigma_f=1$ and $\sigma_l=\sigma_r=0.6$. The compromise functions are the following: for the interactions among the same species we set
 \begin{align*}
  & P_{ff}(w,v)=P_{ll}(w,v)=P_{rr}(w,v)=1, 
  \end{align*}
while the follower-leaders interactions are given by  
  \begin{align*}
  & P_{fr}(w,v)=P_{fl}(w,v)=1-w^2.
  \end{align*}
  In this way, we are modelling the situation where a follower agent with an extreme opinion is less likely to revise its own opinion. 
  According to this, we set
    \begin{align*}
  & P_{lr}(w,v)=P_{rl}(w,v)=0.001(1-w^2).
 \end{align*}
 As expected, the opinions evolution is symmetric and solution converges to a compromise (the opinion $w=0$).
 
 In Figure \ref{fig:FLL2} we consider the same setting as before, but this time one leader group is \emph{stronger} than the other, namely $\sigma_l=0.6$ and $\sigma_r=0.2$. This difference strongly effects the evolution of the followers in the short period, see Figure \ref{fig:FLL2} top-right, while for long time we still observe convergence to the compromise value.
\begin{figure}[htbp]
\begin{center}
\begin{multicols}{2}
    \includegraphics[width=6cm,height=5cm]{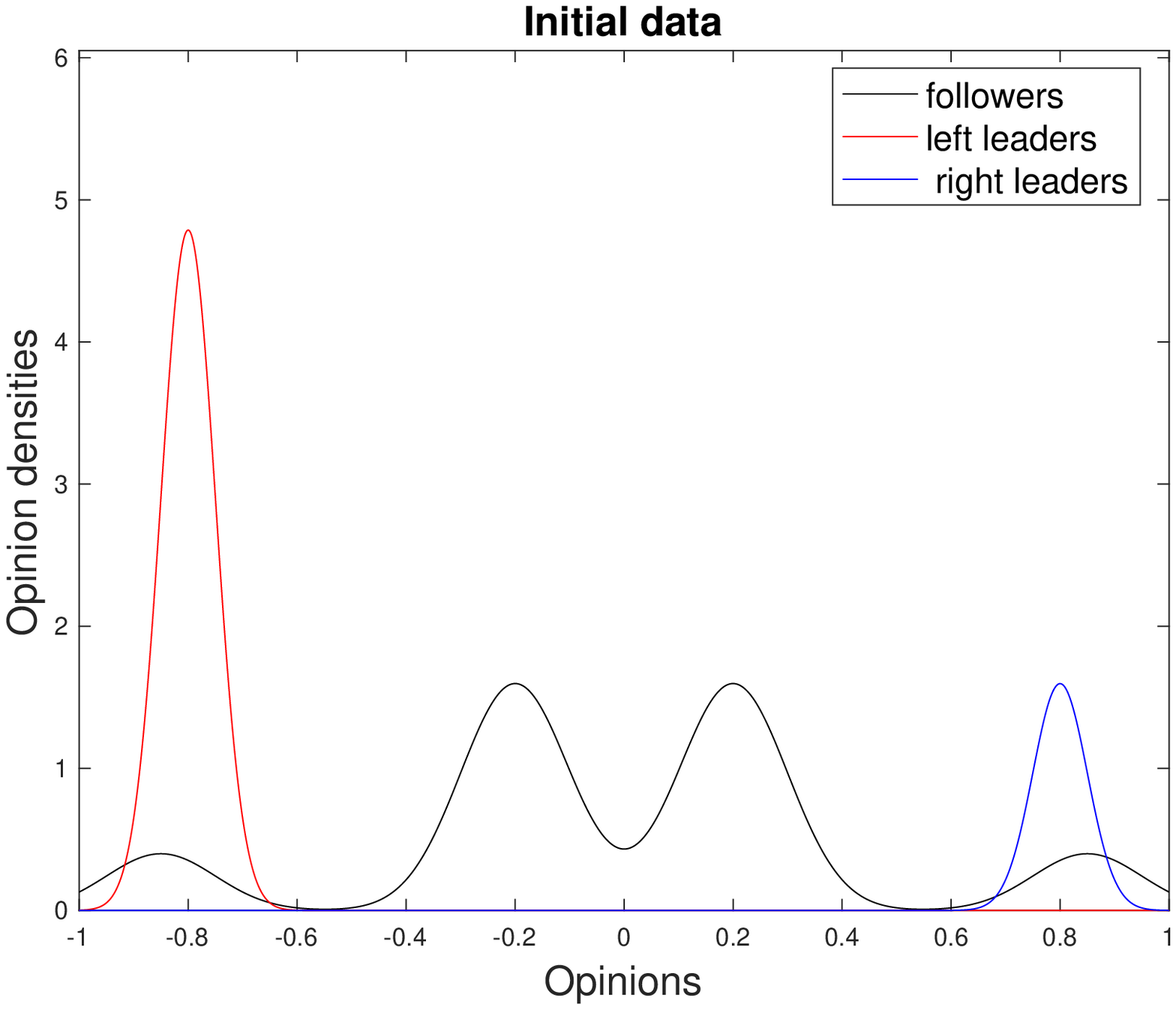}\par 
    \includegraphics[width=6cm,height=5cm]{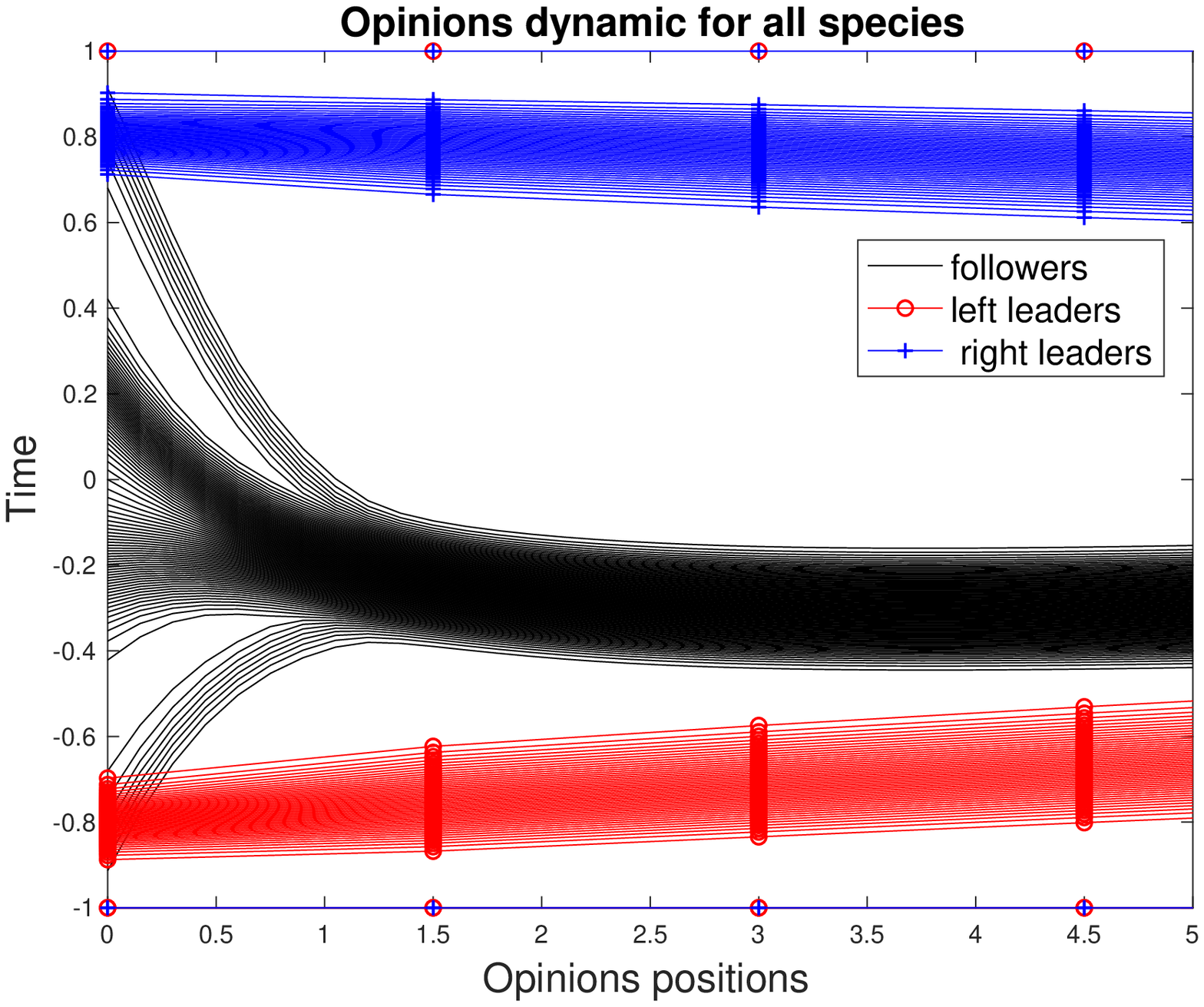}\par 
\end{multicols}
\begin{multicols}{2}
    \includegraphics[width=6cm,height=5cm]{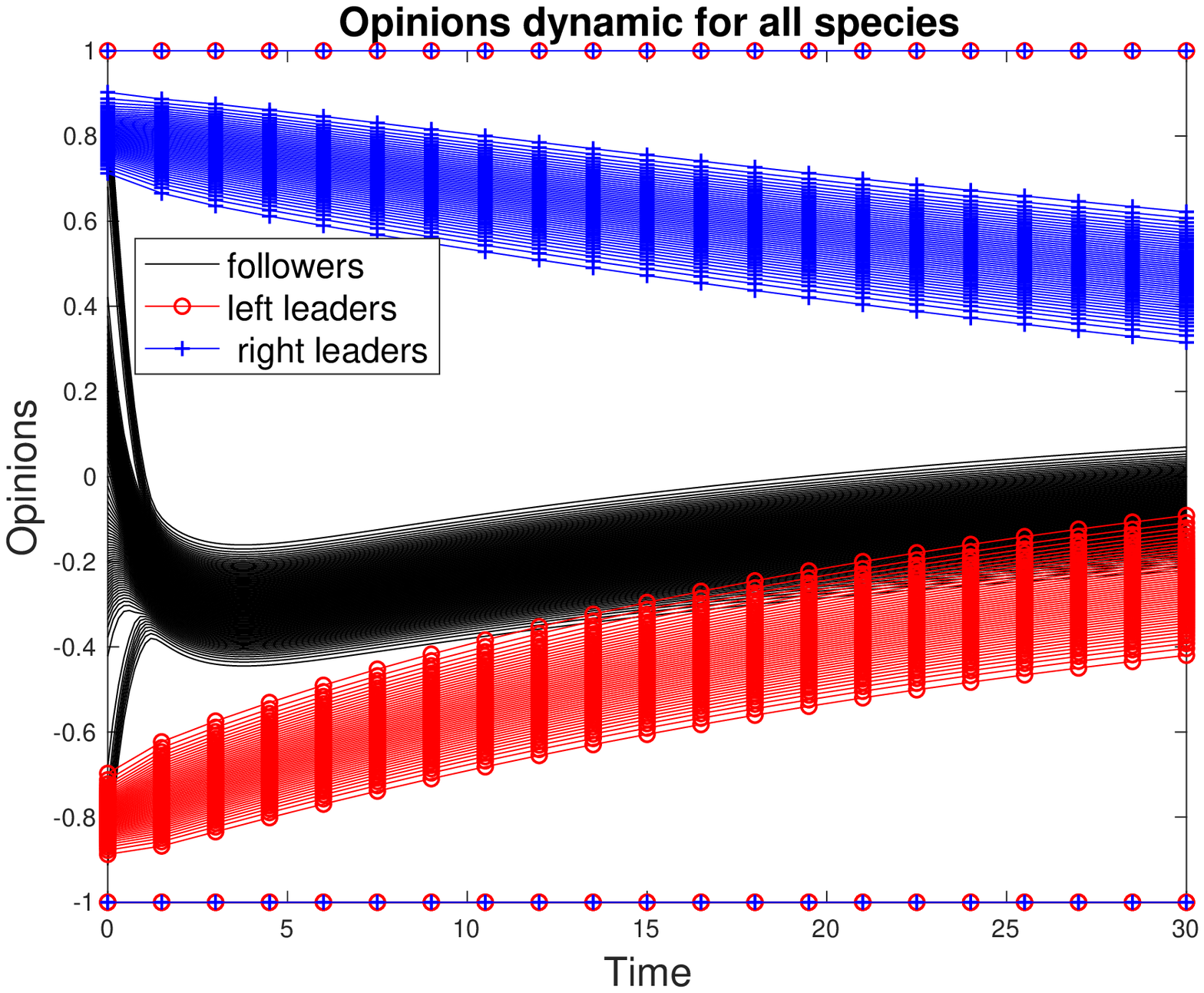}\par
    \includegraphics[width=6cm,height=5cm]{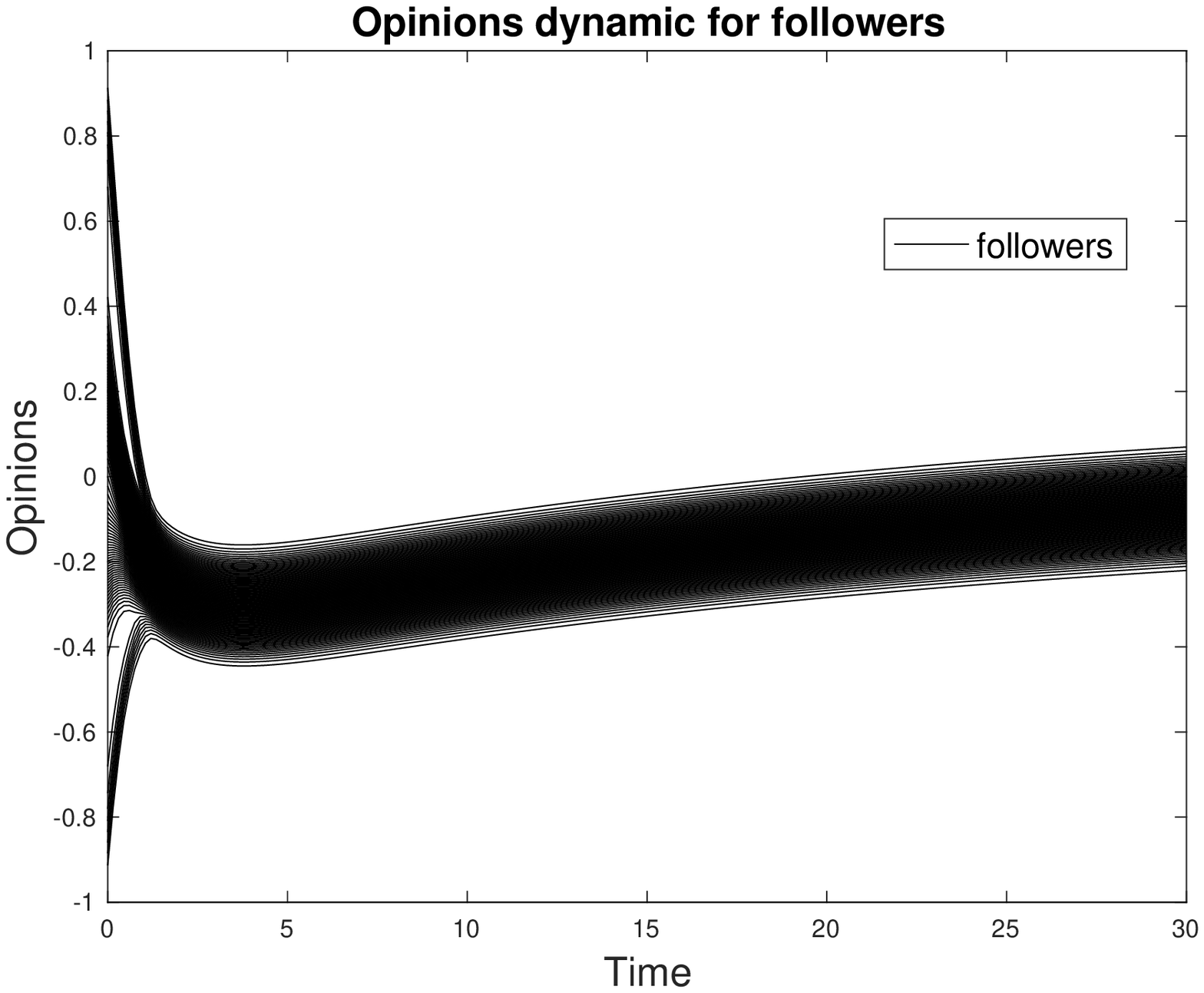}\par
\end{multicols}
\caption{Opinion dynamics in presence of one \emph{strong} group of leaders and one \emph{weak} group of leaders. Top-left initial data for followers(black), left leaders (red) and right leaders (blue). Top-right and bottom-left opinions evolutions for all the groups in short and long term-respectively. Bottom-right opinions evolutions for the followers.}
\label{fig:FLL2}
\end{center}
\end{figure}

We finally consider system \eqref{eq:sys_FLT}, where a small species of \emph{fake} agents is present. The agents of this species are called \emph{trolls}, they are indistinguishable from the followers and they interact with only one group of leaders. In this case, we assume that the fake species has mass $\sigma_q=0.3$ and that it interacts only with leaders sharing the right opinion ($w= +1$). The fact that trolls are not perceived by the followers is modelled by 
\[
P_{ft}(w,v)= P_{ff}(w,v). 
\]
On the other hand, trolls cannot diffuse their opinion, so, according to \eqref{eq:sys_FLT}, the evolution of their opinion is only driven by the compromise part. 
In this example we set
 \[
   P_{ft}(w,v)= 1, \qquad P_{tr}(w,v)=\left(1-\frac{1}{4}(w-v)^2\right).
 \]
 In Figure \ref{fig:FLLT} we show how the presence of trolls affects the behaviour of the population both in the short and the long term.

\begin{figure}[htbp]
\begin{center}
\begin{multicols}{2}
    \includegraphics[width=6cm,height=5cm]{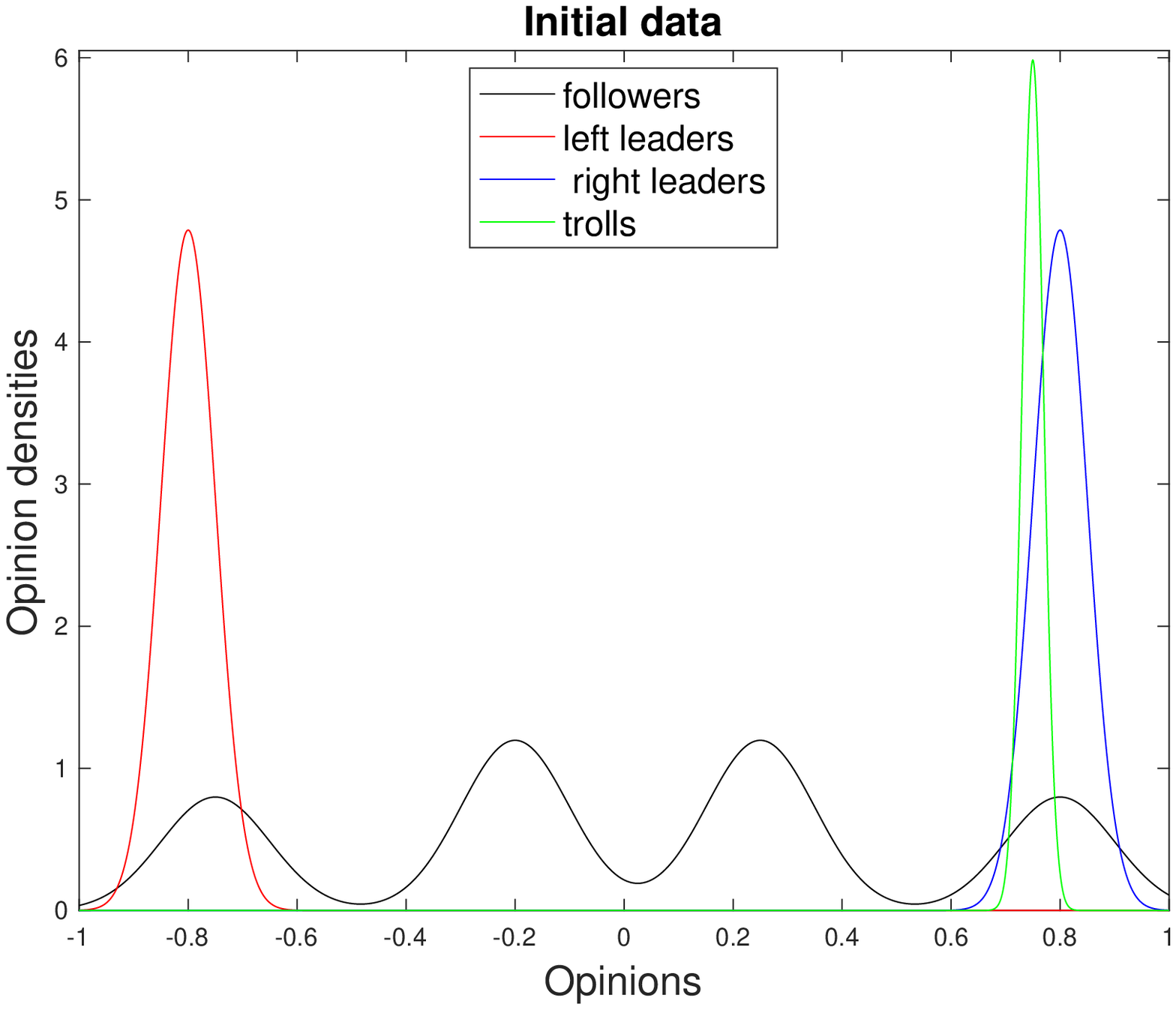}\par 
    \includegraphics[width=6cm,height=5cm]{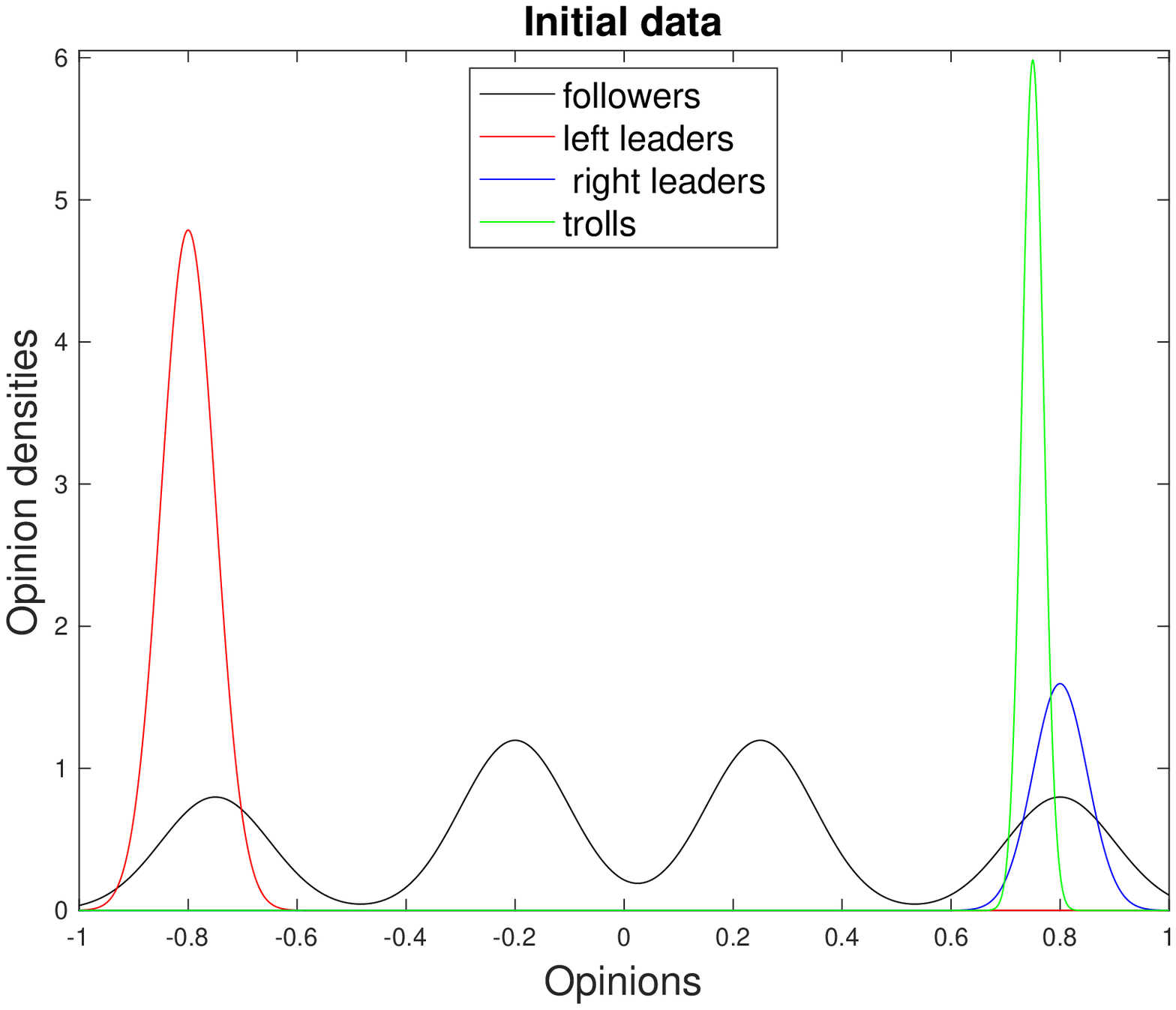}\par 
\end{multicols}
\begin{multicols}{2}
    \includegraphics[width=6cm,height=5cm]{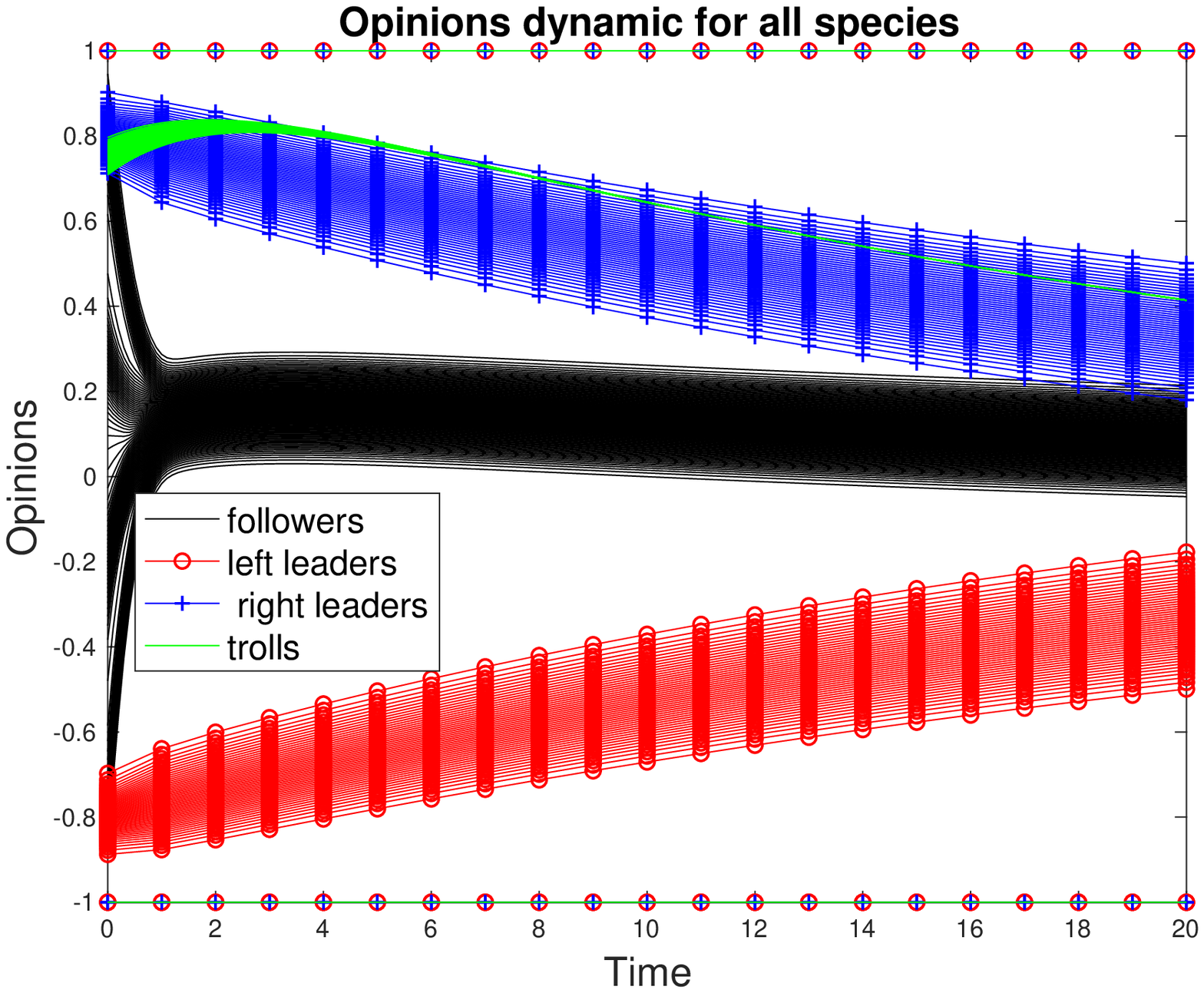}\par
    \includegraphics[width=6cm,height=5cm]{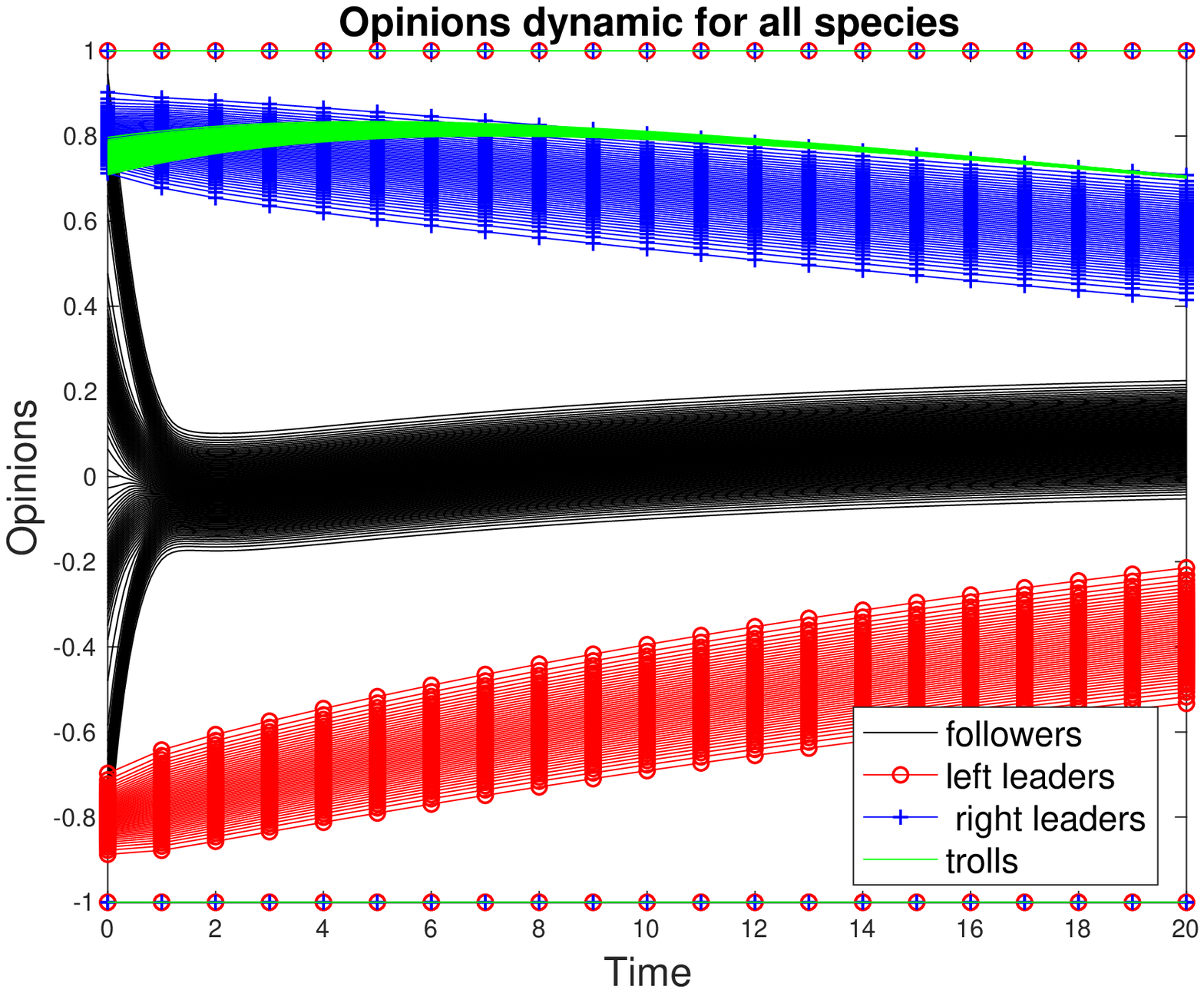}\par
\end{multicols}
\begin{multicols}{2}
    \includegraphics[width=6cm,height=5cm]{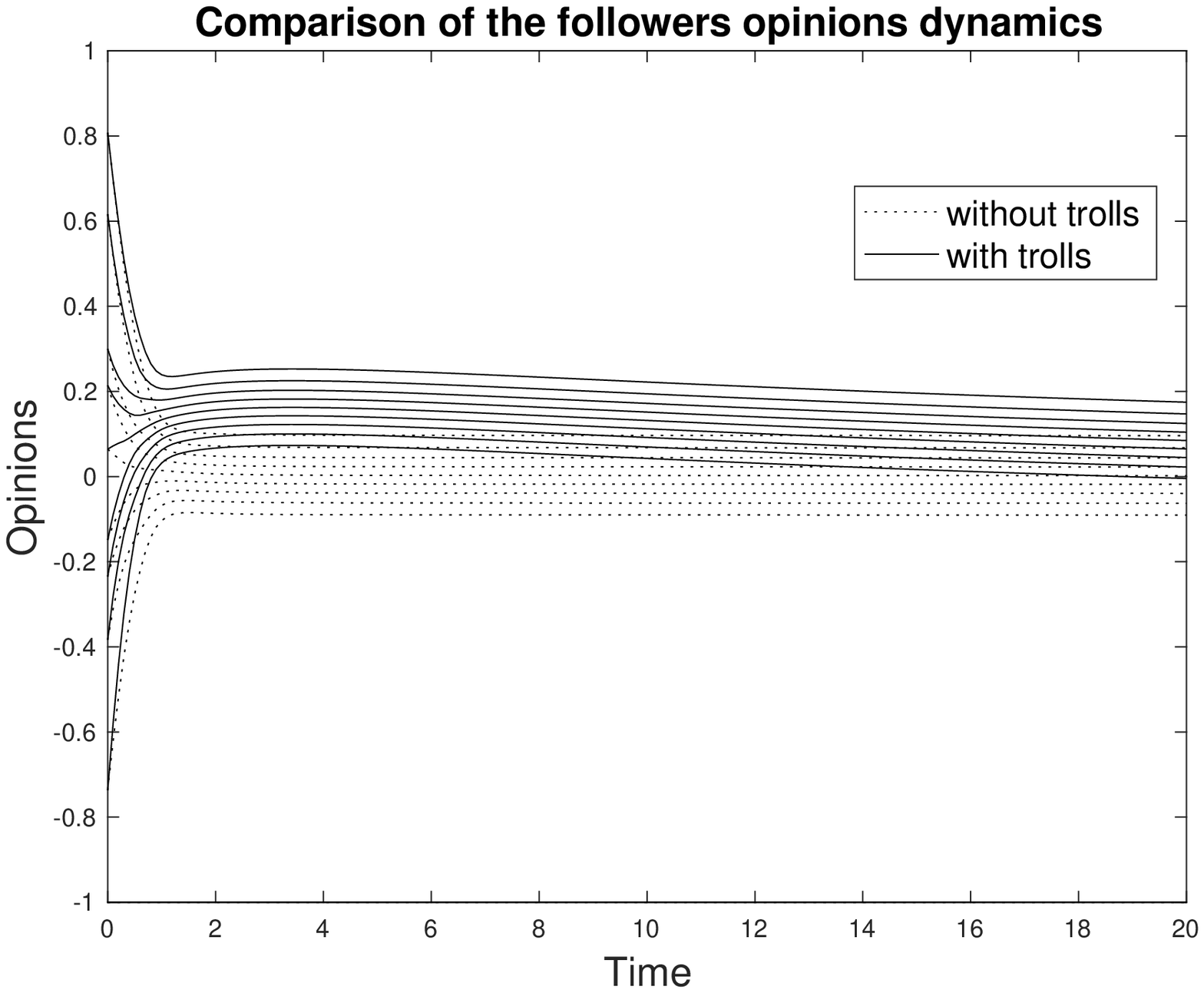}\par
    \includegraphics[width=6cm,height=5cm]{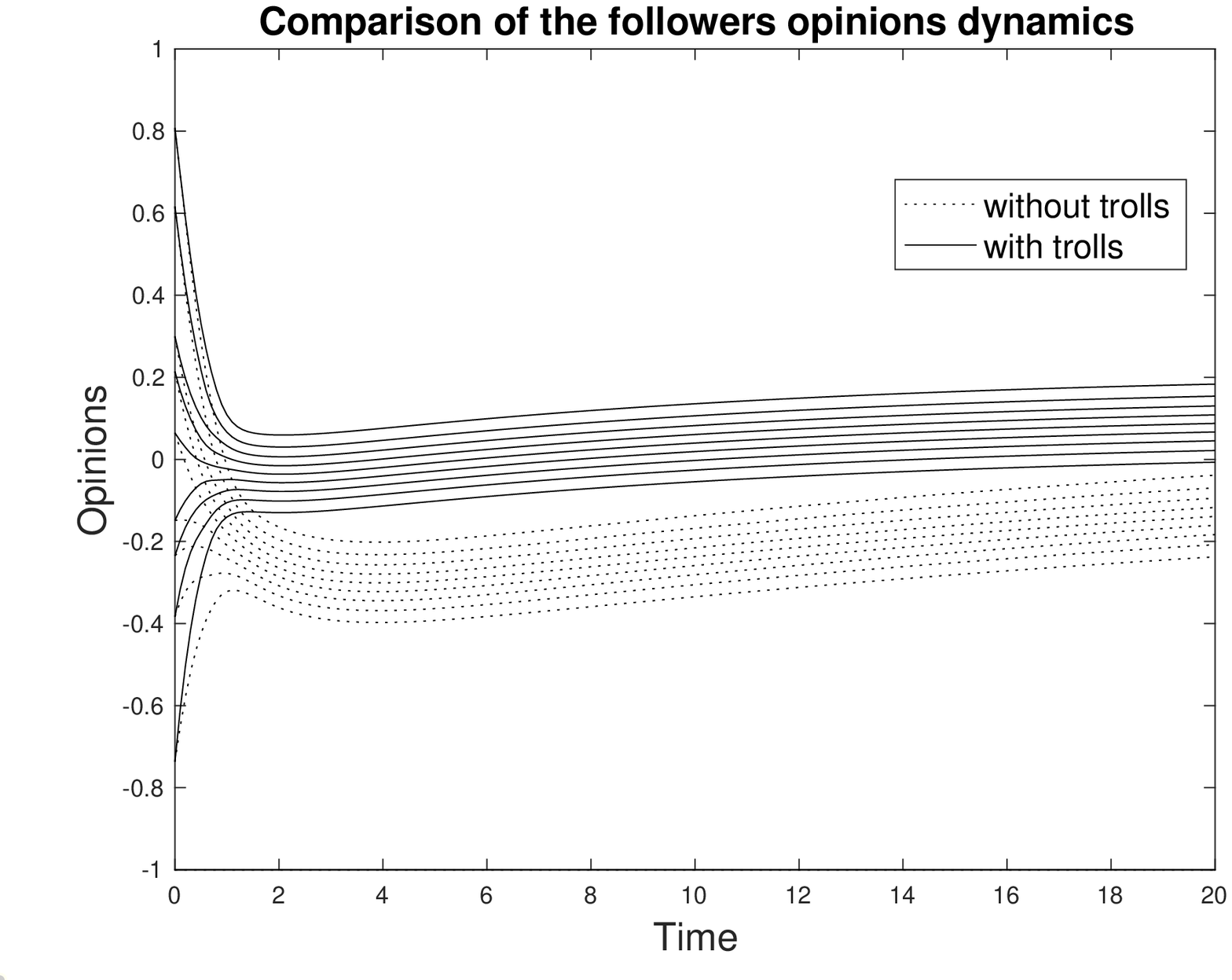}\par
\end{multicols}
\caption{Evolution of system \eqref{eq:sys_FLT}. The left column concerns the case of two equally strong groups of leaders, the right column instead describes the situation where the right leader is \emph{weaker}. Trolls are plotted in green and are associated to the right leaders. Top: initial data. Centre: opinions dynamic for all species. Bottom: comparison between the followers paths with or without trolls.}
\label{fig:FLLT}
\end{center}
\end{figure}


\end{document}